\documentclass[3p,times]{amsart}

\usepackage{amssymb}

\usepackage{amsthm}

\usepackage{enumerate} 

\usepackage{amsmath}
\usepackage{amssymb}

\newtheorem{theorem}{Theorem}[section] 
\newtheorem{lemma}[theorem]{Lemma}   

\newtheorem{proposition}[theorem]{Proposition}

\usepackage[normalem]{ulem} 

\usepackage{dsfont}
\usepackage[all]{xy}

\usepackage{empheq} 

\renewcommand{\mod}{\operatorname{mod}}

\newcommand{\Hom}{\operatorname{Hom}}

\newcommand{\Ext}{\operatorname{Ext}}
\newcommand{\Der}{\operatorname{Der}}

\newcommand{\alg}{\operatorname{alg}}
\newcommand{\soc}{\operatorname{soc}}
\newcommand{\op}{\operatorname{op}}

\newcommand{\rad}{\operatorname{rad}}

\newcommand{\GL}{\operatorname{GL}}
\newcommand{\charact}{\operatorname{char}}
\newcommand{\Ker}{\operatorname{Ker}}

\newcommand{\cB}{\mathcal{B}}

\newcommand{\cN}{\mathcal{N}}
\newcommand{\cT}{\mathcal{T}}
\newcommand{\bA}{\mathbb{A}}

\newcommand{\bG}{\mathbb{G}}
\newcommand{\bN}{\mathbb{N}}
\newcommand{\bQ}{\mathbb{Q}}
\newcommand{\bP}{\mathbb{P}}

\begin{document}


\title[Hochschild cohomology for periodic algebras]{Hochschild cohomology for periodic algebras of polynomial growth}

{\def\thefootnote{}
\footnote{The research was supported by the research grant
DEC-2011/02/A/ST1/00216 of the National Science Center Poland.}
}

\author[J. Bia\l kowski]{Jerzy Bia\l kowski}
\address[Jerzy Bia\l kowski]{Faculty of Mathematics and Computer Science,
   Nicolaus Copernicus University,
   Chopina~12/18,
   87-100 Toru\'n,
   Poland}
\email{jb@mat.uni.torun.pl}

\author[K. Edrmann]{Karin Erdmann}
\address[Karin Erdmann]{Mathematical Institute,
   Oxford University,
   ROQ, Oxford OX2 6GG,
   United Kingdom}
\email{erdmann@maths.ox.ac.uk}

\author[A. Skowro\'nski]{Andrzej Skowro\'nski}
\address[Andrzej Skowro\'nski]{Faculty of Mathematics and Computer Science,
   Nicolaus Copernicus University,
   Chopina~12/18,
   87-100 Toru\'n,
   Poland}
\email{skowron@mat.uni.torun.pl}


\maketitle

\begin{center}
\textit{Dedicated to Kunio Yamagata on the occasion of his 70th birthday}
\end{center}

\begin{abstract}
We describe the dimensions of  low Hochschild cohomology spaces
of exceptional periodic representation-infinite algebras
of polynomial growth.
As an application we obtain that an indecomposable non-standard
periodic representation-infinite algebra of polynomial growth
is not derived equivalent to a standard self-injective algebra.

\bigskip

 

\end{abstract}

\bigskip

\section{Introduction and the main results}%
\label{sec:intro}

Throughout this paper, $K$ will denote a fixed algebraically closed field.
By an algebra we  mean an associative, finite-dimensional $K$-algebra
with  identity, which we assume to be basic and indecomposable.
For an algebra $A$, we denote by $\mod A$ the category of
finite-dimensional right $A$-modules,
and by $D$ the standard duality $\Hom_K(-,K)$ on $\mod A$.
An algebra $A$ is called {self-injective}
if $A \cong D(A)$ in $\mod A$, that is
the projective modules in $\mod A$ are injective.
Further, $A$ is called symmetric if $A$ and $D(A)$
are isomorphic as $A$-$A$-bimodules.
Two self-injective algebras $A$ and $B$ are said to be
socle equivalent if the quotient algebras
$A/\soc (A)$ and $B/\soc (B)$ are isomorphic.

For an algebra $A$, we denote by $D^b(\mod A)$
the derived category of bounded complexes from $\mod A$.
Two algebras $A$ and $B$ are said to be derived equivalent
if $D^b(\mod A)$ and $D^b(\mod B)$ are equivalent
as triangulated categories.
Recall that by the criterion by Rickard \cite{R1}
two algebras $A$ and $B$ are derived equivalent
if and only if $B$ is the endomorphism algebra of a tilting
complex over $A$.
Since many  interesting properties of algebras are preserved
by derived equivalences, it is 
important to classify algebras up to derived equivalence.

For an algebra $A$, the syzygy operator $\Omega_A$ on $\mod A$
is an important tool to construct modules and relate them.
Recall that $\Omega_A$ assigns to a module $M$ in $\mod A$
the kernel $\Omega_A(M)$ of a
projective cover of $M$ in $\mod A$.
A module $M$ in $\mod A$ is called {periodic}
if $\Omega_A^n(M) \cong M$ in $\mod A$ for some $n \geq 1$,
and if so the minimal such $n$ is called the
period of $M$.

An algebra $A$ is defined to be periodic
if it is periodic viewed as a module over
the enveloping algebra $A^e = A^{\op} \otimes_K A$,
or equivalently,  as an $A$-$A$-bimodule.
If $A$ is a periodic algebra, then $A$ is self-injective,
and every nonprojective indecomposable
module in $\mod A$ is periodic.
We also note that
periodicity of  algebras is invariant under derived equivalence.
Periodic algebras have interesting connections with
group theory,
topology,
singularity theory,
cluster algebras
and
mathematical physics.
Periodic algebras include
all nonsimple representation-finite self-injective algebras \cite{Du1},
algebras of quaternion type \cite{ESk1},
weighted surface algebras \cite{ESk3},
preprojective algebras of Dynkin type \cite{ESn},
deformed preprojective algebras of generalized Dynkin type \cite{BES1},
and stable Auslander algebras of hypersurface singularities
of Cohen-Macaulay type \cite{Du2}.
The classification of all periodic algebras up to
isomorphism is an 
attractive open problem.
For tame algebras
one can apply
techniques and results which have been obtained
so far.
Recall that an algebra $A$ is tame if, for each
positive integer $d$, all but finitely many
isomorphism classes of
$d$-dimensional indecomposable modules
come in in a finite number of one-parameter
families \cite{CB1,Dr}.
The algebra  $A$ is said to be of {polynomial growth}
if there is a positive integer $m$
such that for each positive integer $d$,
there are at most $d^m$ such one-parameter families \cite{S1}.
Moreover, $A$ is representation-finite if $\mod A$
admits only a finite number of indecomposable
modules up to isomorphism.
For self-injective algebras,
tameness, polynomial growth and representation-finiteness
of algebras are invariant under derived equivalences.

An important invariant of an algebra $A$ is its
Hochschild cohomology algebra $HH^*(A)$,
which is the graded commutative $K$-algebra
\[
  HH^*(A) = \Ext_{A^e}^*(A,A) = \bigoplus_{i \geq 0} \Ext_{A^e}^i(A,A)
\]
with respect to the Yoneda product
(see \cite{H2,Hoc} for more details).
We note that
$HH^0(A)$ is the center $C(A)$ of $A$, 
$HH^1(A)$ is isomorphic to the quotient space
$\Der_K(A,A)/\Der_K^0(A,A)$
of the space $\Der_K(A,A)$ of $K$-linear derivations
of $A$ modulo the subspace $\Der_K^0(A,A)$ of inner derivations
of $A$, and the vector spaces $HH^n(A)$, $n \geq 2$,
control deformations of $A$ \cite{Ger}.
Moreover, $H^2(A,A)$ describes the Hochschild extensions of $A$
by \cite{SY2}.
It follows from a theorem by Rickard \cite{R3} that,
if $A$ and $B$ are derived equivalent algebras,
then $HH^*(A)$ and $HH^*(B)$ are isomorphic
as graded $K$-algebras.
Moreover, if $A$ is a periodic algebra, then
$HH^*(A)$ is a noetherian graded algebra \cite{Sc}. Let $\cN(A)$
be the ideal generated by all homogeneous nilpotent elements, then 
the quotient algebra $HH^*(A)/\cN(A)$
of $HH^*(A)$ 
is isomorphic, as a graded algebra, to the  polynomial algebra
$K[x]$ where the degree of $x$ is the period of $A$
\cite{Ca}.

Recently, the periodic representation-infinite tame algebras
of polynomial growth have been classified in \cite{BES2}.
By general theory,  self-injective algebras split into two classes:
the standard algebras which admit simply connected Galois coverings,
and the remaining non-standard algebras.
It follows from \cite[Theorem~1.1]{BES2}
that the standard periodic representation-infinite algebras
of polynomial growth are exactly the orbit algebras
$\widehat{B}/G$ of the repetitive categories $\widehat{B}$
of tubular algebras $B$
with respect to admissible infinite cyclic automorphism groups $G$
of $\widehat{B}$. The
 non-standard periodic representation-infinite algebras
of polynomial growth are socle deformations of the corresponding
periodic standard algebras.
Non-standard periodic representation-infinite algebras
of polynomial growth occur only in characteristic $2$ and $3$
(see \cite{BiS3,S5}),
and every such an algebra $\Lambda$ is a geometric socle deformation
of exactly one periodic representation-infinite standard algebra
$\Lambda'$ of polynomial growth, called the standard form of $\Lambda$
(see \cite{BHS2,BiS2,S5}).
We call these algebras $\Lambda$ and $\Lambda'$
exceptional periodic algebras of polynomial growth.
We refer to
Section \ref{sec:exceptional}
for bound quiver presentations of these algebras.

The main aim of this paper is to describe the dimensions
of the  Hochschild cohomology spaces $HH^n(\Lambda)$ and
$HH^n(\Lambda')$ for $n \in \{0,1,2\}$ of the
exceptional periodic algebras $\Lambda$ and $\Lambda'$
of polynomial growth and derive the theorems below.

\begin{theorem}
\label{thm:main1}
Let $\Lambda$ be a
non-standard, representation-infinite periodic algebra
of polynomial growth
and $\Lambda'$ the standard form of $\Lambda$.
Then the following  hold:
\begin{enumerate}[(i)]
\item $\dim_K HH^0 (\Lambda) = \dim_K HH^0 (\Lambda')$;
\item $\dim_K HH^1 (\Lambda) < \dim_K HH^1 (\Lambda')$;
\item $\dim_K HH^2 (\Lambda) < \dim_K HH^2 (\Lambda')$.
\end{enumerate}\end{theorem}

The analogue of the above theorem
for the representation-finite self-injective algebras
was established by Al-Kadi in \cite{AK1}.

Since the dimensions  of the Hochschild cohomology
spaces are derived invariants,
we obtain the following consequence of
Theorem~\ref{thm:main1} and the theory of self-injective
algebras of polynomial growth.

\begin{theorem}
\label{thm:main2}
Let $A$ be a standard self-injective algebra
and $\Lambda$ be a non-standard periodic representation-infinite
algebra of polynomial growth.
Then $A$ and $\Lambda$ are not derived equivalent.
\end{theorem}

For the symmetric algebras, the above theorem is a direct
consequence of \cite[Theorem~1.1]{BES2} and the main result
of \cite{HS2}.
The main working tool in \cite{HS2}
were so-called K\"ulshammer ideals of the centers of algebras,
which for  symmetric algebras
over algebraically closed fields
of positive characteristic are derived invariants, this was 
shown by Zimmermann \cite{Zi}.
The analogues of the above theorem
are valid for  representation-finite self-injective
algebras (see \cite{AK1,As,HS1})
and for  representation-infinite non-periodic (domestic)
self-injective algebras of polynomial growth
(see \cite{BoS,HS1}).

The paper is organized as follows.
In Section~\ref{sec:pre-results} we present general results on
bimodule resolutions of algebras, Hochschild extension algebras,
and derived equivalences of algebras, needed for the proofs 
of the main results of the paper.
Section~\ref{sec:periodic} is devoted to the description 
of representation-infinite periodic algebras of polynomial growth
and their Auslander-Reiten quivers.
In Section~\ref{sec:exceptional} we introduce the exceptional periodic
algebras of polynomial growth and describe their basic properties.
In Sections 
\ref{sec:type2222},
\ref{sec:type333},
\ref{sec:type244},
\ref{sec:type236}
we determine the dimensions of the low Hochschild cohomology
spaces of the exceptional periodic algebras of polynomial growth
and prove  Theorems \ref{thm:main1} and \ref{thm:main2}.


For general background on the relevant representation theory
we refer to the books \cite{ASS,Ri,SS1,SS2,SY,SY2}.

\section{Preliminary results}%
\label{sec:pre-results}

Let $A = K Q/I$, where $Q$ is
a finite quiver and $I$ is an admissible ideal in the path algebra $K Q$
of $Q$ over $K$ (see \cite[Chapter~II]{ASS}).
Let $Q_0$ be the set of vertices of $Q$.
We denote by $e_i$ for $i \in Q_0$, the associated complete set
of pairwise orthogonal primitive idempotents of $A$, and by $S_i = e_i A/e_i \rad A$ for $i \in Q_0$,
the complete family  of pairwise non-isomorphic simple modules in $\mod A$.
Then $P_i:= e_iA$ for $i\in Q_0$ gives a complete set of
pairwise non-isomorphic indecomposable projective right $A$-modules.
The $e_i \otimes e_j$, $i,j \in Q_0$, form a set
of pairwise orthogonal primitive idempotents of
the enveloping algebra $A^e = A^{\op} \otimes_K A$
whose sum is the identity of $A^e$, and consequently
$P(i,j) = (e_i \otimes e_j) A^e = A e_i \otimes e_j A$,
for $i,j \in Q_0$, form a complete set of pairwise non-isomorphic indecomposable
projective modules in $\mod A^e$ (see \cite[Proposition~IV.11.3]{SY}).

The following result by Happel \cite[Lemma~1.5]{H2} describes
the terms of a minimal projective resolution of $A$ in $\mod A^e$.

\begin{proposition}
\label{prop:2.1}
Let $A = K Q/I$ be a bound quiver algebra.
Then there is in $\mod A^e$ a minimal projective resolution of $A$
of the form
\[
  \cdots \rightarrow
  \bP_n \xrightarrow{d_n}
  \bP_{n-1} \xrightarrow{ }
  \cdots \rightarrow
  \bP_1 \xrightarrow{d_1}
  \bP_0 \xrightarrow{d_0}
  A \rightarrow 0,
\]
where
\[
  \bP_n = \bigoplus_{i,j \in Q_0}
          P(i,j)^{\dim_K \Ext_A^n(S_i,S_j)}
\]
for any $n \in \bN$.
\end{proposition}

There is  no general recipe for
the differentials $d_n$, except the first three.
Before describing  the differentials $d_0, d_1, d_2$,
we recall an important property of the syzygy
modules
(see \cite[Lemma~IV.11.16]{SY}).

\begin{lemma}
\label{lem:2.2}
Let $A$ be an algebra.
For any positive integer $n$, the module  $\Omega_{A^e}^n(A)$
is  projective as a left $A$-module and also as a right
$A$-module.
\end{lemma}

We describe now the start of the resolution in Proposition~\ref{prop:2.1}.
We have
\[
  \bP_0 = \bigoplus_{i \in Q_0} P(i,i)
        = \bigoplus_{i \in Q_0} A e_i \otimes e_i A .
\]
The homomorphism $d_0 : \bP_0 \to A$ in $\mod A^e$ defined by
$d_0 (e_i \otimes e_i) = e_i$ for all $i \in Q_0$
is a minimal projective cover of $A$ in $\mod A^e$.
Let $Q_1$ be the set of arrows of the quiver $Q$. For an arrow $\alpha$ of $Q$,
we denote by $s(\alpha)$ and $t(\alpha)$ the source
and target of $\alpha$.
Recall also that, for two vertices $i$ and $j$ in $Q$,
the number of arrows from $i$ to $j$ in $Q$ is equal
to $\dim_K \Ext_A^1(S_i,S_j)$
(see \cite[Lemma~III.2.12]{ASS}).
Hence we have
\[
  \bP_1 = \bigoplus_{\alpha \in Q_1} P\big(s(\alpha),t(\alpha)\big)
        = \bigoplus_{\alpha \in Q_1} A e_{s(\alpha)} \otimes e_{t(\alpha)} A
        .
\]
Then we have the following known fact.

\begin{lemma}
\label{lem:2.3}
Let $A = K Q/I$ be a bound quiver algebra, and
$d_1 : \bP_1 \to \bP_0$ the homomorphism in $\mod A^e$
defined by
\[
 d_1(e_{s(\alpha)} \otimes e_{t(\alpha)}) =
   \alpha \otimes e_{t(\alpha)} - e_{s(\alpha)} \otimes \alpha
\]
for any arrow $\alpha$ in $Q$.
Then $d_1$ induces a minimal projective cover
$d_1 : \bP_1 \to \Omega_{A^e}^1(A)$ of
$\Omega_{A^e}^1(A) = \Ker d_0$ in $\mod A^e$.
In particular, we have
$\Omega_{A^e}^2(A) \cong \Ker d_1$ in $\mod A^e$.
\end{lemma}

We will denote  the  homomorphism
$d_1 : \bP_1 \to \bP_0$ by $d$.
For the algebras $A$ we will consider, the kernel
$\Omega_{A^e}^2(A)$ of $\Ker d_1$ will be generated,
as an $A$-$A$-bimodule, by some elements of $\bP_1$
associated to a set of relations generating the
admissible ideal $I$.
Recall that a relation in the path algebra $KQ$
is an element of the form
\[
  \mu = \sum_{r=1}^n c_r \mu_r
  ,
\]
where $c_1, \dots, c_r$ are elements of $K$ and
$\mu_r = \alpha_1^{(r)} \alpha_2^{(r)} \dots \alpha_{m_r}^{(r)}$
are paths in $Q$ of length $m_r \geq 2$, $r \in \{1,\dots,n\}$,
having a common source and a common target.
Then the admissible ideal $I$ is generated by a finite set
of relations in $K Q$ (see \cite[Corollary~II.2.9]{ASS}).
In particular, the bound quiver algebra $A = K Q/I$ is given
by the path algebra $K Q$ and a finite number of equalities
$\sum_{r=1}^n c_r \mu_r = 0$ given by a finite set of relations
generating the ideal $I$.
Consider the $K$-linear homomorphism $\varrho : K Q \to \bP_1$
which assigns to a path $\alpha_1 \alpha_2 \dots \alpha_m$ in $Q$
the element
\[
  \varrho(\alpha_1 \alpha_2 \dots \alpha_m)
   = \sum_{k=1}^m \alpha_1 \alpha_2 \dots \alpha_{k-1}
                  \otimes \alpha_{k+1} \dots \alpha_m
\]
in $\bP_1$, where $\alpha_0 = e_{s(\alpha_1)}$
and $\alpha_{m+1} = e_{t(\alpha_m)}$.
The $k$-th term in the sum lies in the summand of $\bP_1$ corresponding
to the arrow $\alpha_k$. Observe that
$\varrho(\alpha_1 \alpha_2 \dots \alpha_m) \in e_{s(\alpha_1)} \bP_1 e_{t(\alpha_m)}$.
Then, for a relation $\mu = \sum_{r=1}^n c_r \mu_r$
in $K Q$ lying in $I$, we have an element
\[
  \varrho(\mu) = \sum_{r=1}^n c_r \varrho(\mu_r) \in e_i \bP_1 e_j ,
\]
where $i$ is the common source and $j$ is the common
target of the paths $\mu_1,\dots,\mu_r$.
The following lemma will be useful; we will leave the proof to the
reader as it is straightforward.

\begin{lemma}
\label{lem:2.4}
Let $A = K Q/I$ be a bound quiver algebra and
$d_1 : \bP_1 \to \bP_0$ the homomorphism in $\mod A^e$
defined above.
Then for any relation $\mu$ in $K Q$ lying in $I$,
we have $d_1(\varrho(\mu)) = 0$.
\end{lemma}

For an algebra $A = K Q/I$ in our context, we will see that
there exists a family of relations $\mu^{(1)},\dots,\mu^{(q)}$
generating the ideal $I$ such that the associated elements
$\varrho(\mu^{(1)}), \dots, \varrho(\mu^{(q)})$ generate
the $A$-$A$-bimodule $\Omega_{A^e}^2(A) = \Ker d_1$.
Moreover, we have
\[
  \bP_2 = \bigoplus_{j = 1}^q P\big(s(\mu^{(j)}),t(\mu^{(j)})\big)
        = \bigoplus_{j = 1}^q A e_{s(\mu^{(j)})} \otimes e_{t(\mu^{(j)})} A
        ,
\]
and the homomorphism $d_2 : \bP_2 \to \bP_1$ in $\mod A^e$ such that
\[
  d_2 \big(e_{s(\mu^{(j)})} \otimes e_{t(\mu^{(j)})}\big) = \varrho(\mu^{(j)})
  ,
\]
for $j \in \{1,\dots,q\}$, defines a projective cover
$d_2 : \bP_2 \to \Omega_{A^e}^2(A)$ of $\Omega_{A^e}^2(A)$ in $\mod A^e$.
We will abbreviate this homomorphism $d_2$ by $R$.
In particular, we will have $\Omega_{A^e}^3(A) \cong \Ker R$.

To compute the dimensions of  Hochschild cohomology spaces,
we will use the following.

\begin{proposition}
\label{prop:2.5}
Let $A = KQ/I$ be a bound quiver algebra over $K$
and
\[
  \bP_0 = \bigoplus_{i \in Q_0} A e_i \otimes e_i A
  ,
  \ \
  \bP_1 = \bigoplus_{\alpha \in Q_1} A e_{s(\alpha)} \otimes e_{t(\alpha)} A
        .
\]
Then there are exact sequences in $\mod K$ of the forms
\begin{gather*}
 0
  \to HH^0(A)
  \to \Hom_{A^e}(\bP_0,A)
  \to \Hom_{A^e}\big(\Omega_{A^e}(A),A\big)
  \to HH^1(A)
  \to 0
 ,\\
 0
  \to \Hom_{A^e}\big(\Omega_{A^e}(A),A\big)
  \to \Hom_{A^e}(\bP_1,A)
  \to \Hom_{A^e}\big(\Omega_{A^e}^2(A),A\big)
  \to HH^2(A)
  \to 0
  .
\end{gather*}
\end{proposition}

We present now some results applied in the proofs of our
main results.

The following two results are due to Rickard \cite{R3}.

\begin{theorem}
\label{thm:2.6}
Let $A$ and $B$ be derived equivalent algebras over a field $K$.
Then $HH^*(A)$ and $HH^*(B)$ are isomorphic as graded $K$-algebras.
\end{theorem}

\begin{theorem}
\label{thm:2.7}
Let $A$ and $B$ be derived equivalent algebras.
Then $A$ is symmetric if and only if $B$ is symmetric.
\end{theorem}

The following known result is also due to Rickard \cite{R2}.

\begin{theorem}
\label{thm:2.8}
Let $A$ and $B$ be derived equivalent self-injective algebras.
Then $A$ and $B$ are stably equivalent.
\end{theorem}

The next result is a consequence of the results proved
by Krause and Zwara \cite{KZ}.

\begin{theorem}
\label{thm:2.9}
Let $A$ and $B$ be stably equivalent  algebras.
Then $A$ is of polynomial growth if and only if $B$ is of polynomial
growth.
\end{theorem}

We end this section with the following known result
(see \cite[Theorem~2.9]{ESk2} for a proof).

\begin{theorem}
\label{thm:2.10}
Let $A$ and $B$ be derived equivalent indecomposable algebras.
Then $A$ is a periodic algebra if and only if $B$ is a periodic algebra.
Moreover, if so then 
their periods coincide.
\end{theorem}

\section{Periodic algebras of polynomial growth}%
\label{sec:periodic}

We will now  describe the indecomposable
representation-infinite periodic algebras of polynomial growth
and their Auslander-Reiten quivers.

The following classification result was established by
the authors in \cite[Theorem~1.1]{BES2}.

\begin{theorem}
\label{thm:pre1}
Let $A$ be a basic, indecomposable, representation-infinite
self-injective algebra of polynomial growth
over an algebraically closed field $K$.
The following statements are equivalent:
\begin{enumerate}[(i)]
\item $A$ is a periodic algebra;
\item $A$ is socle equivalent to an orbit algebra
    $\widehat{B}/G$ for a tubular algebra $B$ and
    an admissible infinite cyclic automorphism group $G$
    of the repetitive category $\widehat{B}$ of $B$.
\end{enumerate}
Moreover, if $K$ is of characteristic different from
$2$ and $3$,
we may replace in (ii) ``socle equivalent''
by ``isomorphic''.
\end{theorem}

In particular, we have the following consequence
of the above theorem and the main result of \cite{S2}.

\begin{theorem}
\label{thm:pre2}
Let $A$ be a basic, indecomposable, representation-infinite
self-injective algebra of polynomial growth.
The following statements are equivalent:
\begin{enumerate}[(i)]
\item $A$ is a standard periodic algebra;
\item $A$ is isomorphic to an orbit algebra
    $\widehat{B}/G$ for a tubular algebra $B$ and
    an admissible infinite cyclic automorphism group $G$
    of the repetitive category $\widehat{B}$ of $B$.
\end{enumerate}
\end{theorem}

Recall that the tubular algebras $B$ are algebras of global dimension $2$ with
the Grothendieck groups $K_0(B)$ of ranks $6$, $8$, $9$, or $10$,
corresponding to the tubular types
$(2,2,2,2)$,
$(3,3,3)$,
$(2,4,4)$,
or
$(2,3,6)$,
introduced by Ringel in \cite{Ri}.
The representation theory of tubular algebras is described in \cite[Section~5]{Ri}.
If $B$ is such an algebra, then the module category
$\mod \widehat{B}$
of $\widehat{B}$ may be thought of as a glueing of infinitely many
module categories of the corresponding tubular algebras,
obtained from $B$ by iterated reflections (see \cite{NS,S2}).
The module category of
$\widehat{B}/G$ is then obtained
via the push-down functor
$F_{\lambda} : \mod \widehat{B} \to \mod \widehat{B}/G$
(associated to the Galois covering
$F : \widehat{B} \to \widehat{B}/G$). This functor
 is dense by the density theorem of Dowbor and
Skowro\'nski \cite{DS1,DS2}.
Then by a theorem of Gabriel \cite{Ga} the Auslander-Reiten
quiver $\Gamma_{\widehat{B}/G}$ of $\widehat{B}/G$
is the orbit quiver $\Gamma_{\widehat{B}}/G$ of
the Auslander-Reiten quiver $\Gamma_{\widehat{B}}$
of $\widehat{B}$ with respect to the induced action of $G$.
Socle equivalent  algebras
have isomorphic  Auslander-Reiten quivers.
Hence, if $A$ is an algebra as in Theorem~\ref{thm:pre1}, applying
\cite{NS,S2}, we conclude that
the Auslander-Reiten quiver $\Gamma_A$ of a periodic
indecomposable
representation-infinite self-injective algebra $A$
of polynomial growth
has the following clock structure
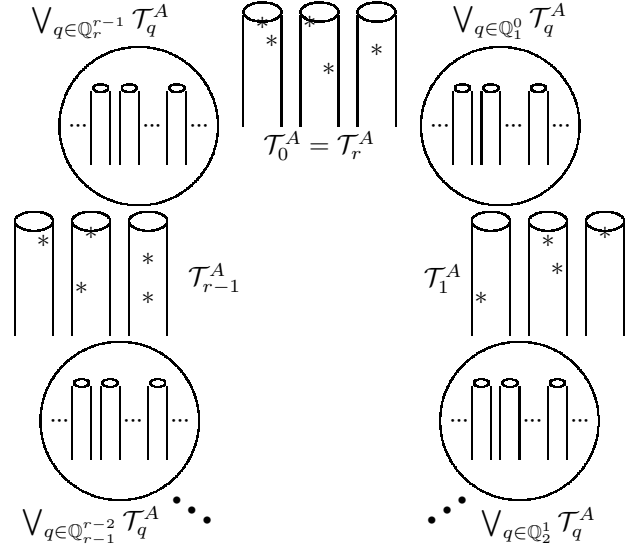
\[
\xymatrix@M=0pc@L=0pc@C=.30pc@R=0.30pc{
&&&&&&&& &&&&&&&&&& &&&&&&
  \ar@{-}[dddddddddddd] &&{}\save[] *{\xycircle<.60pc,.30pc>{}} \restore&& \ar@{-}[dddddddddddd] &&
    \ar@{-}[dddddddddddd] &&{}\save[] *{\xycircle<.60pc,.30pc>{}} \restore&& \ar@{-}[dddddddddddd] &&
    \ar@{-}[dddddddddddd] &&{}\save[] *{\xycircle<.60pc,.30pc>{}} \restore&& \ar@{-}[dddddddddddd] &&
  \\
 &&&&&&&&& {}\save[] +<0pc,0pc> *{\bigvee_{q \in \mathbb{Q}^{r-1}_{r}}\mathcal{T}_q^A} \restore
 &&&&&&&&&&&&&&&&& {}\save[] +<0pc,-.05pc> *{\ast} \restore
 &&&&& {}\save[] +<0pc,0pc> *{\ast} \restore &&&&&&&&&&&&&&&&&&&&&
    {}\save[] +<0pc,0pc> *{\bigvee_{q \in \mathbb{Q}^{0}_{1}}\mathcal{T}_q^A} \restore
  \\ \\
 &&&&&&&&&&&&&&&&&&&&&&&&&&&  {}\save[] +<0pc,0pc> *{\ast} \restore
 \\
 &&&&&&&&&&&&&&&&&&&&&&&&&&&&&&&&&&&&&&  {}\save[] +<0pc,0pc> *{\ast} \restore
 \\ \\
 &&&&&&&&&&&&&&&&&&&&&&&&&&&&&&&&&  {}\save[] +<0pc,0pc> *{\ast} \restore
  \\ \\
&&&&&&&&
  \ar@{-}[dddddddd] & {}\save[] *{\xycircle<.24pc,.12pc>{}} \restore & \ar@{-}[dddddddd] &
  \ar@{-}[dddddddd] & {}\save[] *{\xycircle<.24pc,.12pc>{}} \restore & \ar@{-}[dddddddd] &&&
  \ar@{-}[dddddddd] & {}\save[] *{\xycircle<.24pc,.12pc>{}} \restore & \ar@{-}[dddddddd]
  &&&&&& &&&& && &&&& && &&&& &&&&&&
  \ar@{-}[dddddddd] & {}\save[] *{\xycircle<.24pc,.12pc>{}} \restore & \ar@{-}[dddddddd] &
  \ar@{-}[dddddddd] & {}\save[] *{\xycircle<.24pc,.12pc>{}} \restore & \ar@{-}[dddddddd] &&&
  \ar@{-}[dddddddd] & {}\save[] *{\xycircle<.24pc,.12pc>{}} \restore & \ar@{-}[dddddddd]
     \\ \\ \\ \\
&&&&&&&&
  {}\save[] +<-.2pc,0pc> *{\cdot} +<-.2pc,0pc> *{\cdot} +<-.2pc,0pc> *{\cdot} \restore
  && & && {}\save[] *{\xycircle<2.5pc,2.5pc>{}} \restore
    &{}\save[] +<-.08pc,0pc> *{\cdot} +<.2pc,0pc> *{\cdot} +<.2pc,0pc> *{\cdot} \restore&& &&
  {}\save[] +<.21pc,0pc> *{\cdot} +<.2pc,0pc> *{\cdot} +<.2pc,0pc> *{\cdot} \restore
  &&&&&& &&&& && &&&& && &&&& &&&&&&
  {}\save[] +<-.2pc,0pc> *{\cdot} +<-.2pc,0pc> *{\cdot} +<-.2pc,0pc> *{\cdot} \restore
  && & && {}\save[] *{\xycircle<2.5pc,2.5pc>{}} \restore
    &{}\save[] +<-.08pc,0pc> *{\cdot} +<.2pc,0pc> *{\cdot} +<.2pc,0pc> *{\cdot} \restore&& &&
  {}\save[] +<.21pc,0pc> *{\cdot} +<.2pc,0pc> *{\cdot} +<.2pc,0pc> *{\cdot} \restore
 \\ \\
   &&&&&&&&&&&&&&&&&&&&&&&&&&&&&&&&  {}\save[] +<0pc,0pc> *{\cT_0^A = \cT_r^A} \restore
 \\ \\
&&&&&&&&
  &&&&&& && & && &&& &&   &&&& && &&&& && &&&&  &&&&&&  && & && &&& && \\
 \\ \\
 \\ \\
\\
  \ar@{-}[dddddddddddd] &&{}\save[] *{\xycircle<.60pc,.30pc>{}} \restore&& \ar@{-}[dddddddddddd] &&
    \ar@{-}[dddddddddddd] &&{}\save[] *{\xycircle<.60pc,.30pc>{}} \restore&& \ar@{-}[dddddddddddd] &&
    \ar@{-}[dddddddddddd] &&{}\save[] *{\xycircle<.60pc,.30pc>{}} \restore&& \ar@{-}[dddddddddddd] &&
  &&&&&&&&&&&&&&&&&&&&&&&&&&&&&&
  \ar@{-}[dddddddddddd] &&{}\save[] *{\xycircle<.60pc,.30pc>{}} \restore&& \ar@{-}[dddddddddddd] &&
    \ar@{-}[dddddddddddd] &&{}\save[] *{\xycircle<.60pc,.30pc>{}} \restore&& \ar@{-}[dddddddddddd] &&
    \ar@{-}[dddddddddddd] &&{}\save[] *{\xycircle<.60pc,.30pc>{}} \restore&& \ar@{-}[dddddddddddd] &&
  \\
   &&&&&&&&  {}\save[] +<0pc,-0.05pc> *{\ast} \restore
   &&&&&&&&&&&&&&&&&&&&&&&&&&&&&&&&&&&&&&&&&&&&&&&&&&&&&&  {}\save[] +<0pc,-0.05pc> *{\ast} \restore
  \\
   &&&  {}\save[] +<0pc,0pc> *{\ast} \restore
   &&&&&&&&&&&&&&&&&&&&&&&&&&&&&&&&&&&&&&&&&&&&&&&&&&&&&  {}\save[] +<0pc,0pc> *{\ast} \restore
  \\ \\
   &&&&&&&&&&&&&&  {}\save[] +<0pc,0pc> *{\ast} \restore
  \\
   &&&&&&&&&&&&&&&&&&&&&&&&&&&&&&&&&&&&&&&&&&&&&&&&&&&&&&&&&  {}\save[] +<0pc,0pc> *{\ast} \restore
  \\
  &&&&&&&&&&&&&&&&&&&&& {}\save[] +<0pc,0pc> *{\cT_{r-1}^A} \restore
  &&&&&&&&&&&&&&&&&&&&&&&& {}\save[] +<0pc,0pc> *{\cT_{1}^A} \restore
  \\
   &&&&&&&  {}\save[] +<0pc,0pc> *{\ast} \restore
  \\
   &&&&&&&&&&&&&&  {}\save[] +<0pc,0pc> *{\ast} \restore
   &&&&&&&&&&&&&&&&&&&&&&&&&&&&&&&&&&&  {}\save[] +<0pc,0pc> *{\ast} \restore
     \\ \\ \\ \\
  &&&& && &&&& && &&&&
  &&&&&&&&&&&&&&&&&&&&&&&&&&&&&&&&
  &&&& && &&&& && &&&&
  \\
  \\ \\ \\ \\
&&&&&&
  \ar@{-}[dddddddd] & {}\save[] *{\xycircle<.24pc,.12pc>{}} \restore & \ar@{-}[dddddddd] &
  \ar@{-}[dddddddd] & {}\save[] *{\xycircle<.24pc,.12pc>{}} \restore & \ar@{-}[dddddddd] &&&
  \ar@{-}[dddddddd] & {}\save[] *{\xycircle<.24pc,.12pc>{}} \restore & \ar@{-}[dddddddd]
  &&&&&&&& &&&& && &&&& && &&&& &&&&&&&&
  \ar@{-}[dddddddd] & {}\save[] *{\xycircle<.24pc,.12pc>{}} \restore & \ar@{-}[dddddddd] &
  \ar@{-}[dddddddd] & {}\save[] *{\xycircle<.24pc,.12pc>{}} \restore & \ar@{-}[dddddddd] &&&
  \ar@{-}[dddddddd] & {}\save[] *{\xycircle<.24pc,.12pc>{}} \restore & \ar@{-}[dddddddd]
     \\ \\ \\ \\
&&&&&&
  {}\save[] +<-.2pc,0pc> *{\cdot} +<-.2pc,0pc> *{\cdot} +<-.2pc,0pc> *{\cdot} \restore
  && & && {}\save[] *{\xycircle<2.5pc,2.5pc>{}} \restore
    &{}\save[] +<-.08pc,0pc> *{\cdot} +<.2pc,0pc> *{\cdot} +<.2pc,0pc> *{\cdot} \restore&& &&
  {}\save[] +<.21pc,0pc> *{\cdot} +<.2pc,0pc> *{\cdot} +<.2pc,0pc> *{\cdot} \restore
  &&&&&&&& &&&& && &&&& && &&&& &&&&&&&&
  {}\save[] +<-.2pc,0pc> *{\cdot} +<-.2pc,0pc> *{\cdot} +<-.2pc,0pc> *{\cdot} \restore
  && & && {}\save[] *{\xycircle<2.5pc,2.5pc>{}} \restore
    &{}\save[] +<-.08pc,0pc> *{\cdot} +<.2pc,0pc> *{\cdot} +<.2pc,0pc> *{\cdot} \restore&& &&
  {}\save[] +<.21pc,0pc> *{\cdot} +<.2pc,0pc> *{\cdot} +<.2pc,0pc> *{\cdot} \restore
 \\ \\ \\ \\
&&&&&&&&
  &&&&&& && & && &&& &&   &&&& && &&&& && &&&&  &&&&&&  && & && &&& && && \\
  \\ \\ \\ \\
 &&&&&&&&&&&&&&&&&& {}\save[] +<-.3pc,.2pc> *{\mbox{\tiny$\bullet$}} +<.48pc,-.32pc> *{\mbox{\tiny$\bullet$}} +<.48pc,-.32pc> *{\mbox{\tiny$\bullet$}} \restore
  &&&&&&&&&& &&&&
 &&&&&&& &&&&&&& {}\save[] +<.3pc,.2pc> *{\mbox{\tiny$\bullet$}} +<-.48pc,-.32pc> *{\mbox{\tiny$\bullet$}} +<-.48pc,-.32pc> *{\mbox{\tiny$\bullet$}} \restore \\
  \\
  &&&&&&&&
  {}\save[] +<0pc,0pc> *{\bigvee_{q \in \mathbb{Q}^{r-2}_{r-1}}\mathcal{T}_q^A} \restore
  &&&&&&&&&&&&&&&&&&&&&&&&&&&&&&&&&&&&&&&&&&&&&&&
  {}\save[] +<0pc,0pc> *{\bigvee_{q \in \mathbb{Q}^{1}_{2}}\mathcal{T}_q^A} \restore
  }
\]
where $\ast$ denote projective-injective modules,
$r \geq 3$, $\bQ_i^{i-1} = \bQ \cap (i-1,i)$
for any $i \in \{1,\dots,r\}$, and
\begin{enumerate}[(a)]
 \item
  for each $i \in \{ 0,1,\dots,r-1 \}$,
  $\mathcal{T}_i^A$ is a $\mathbb{P}_1(K)$-family of quasi-tubes
  (the stable parts are stable tubes);
 \item
  for each $q \in \bQ_i^{i-1}$, $i \in \{1,\dots,r\}$,
  $\mathcal{T}_q^A$ is a $\mathbb{P}_1(K)$-family of stable tubes;
 \item
  all $\mathbb{P}_1(K)$-families $\mathcal{T}_q^A$, $q \in \bQ \cap [0,r]$,
  have the same tubular type
  $(2,2,2,2)$, $(3,3,3)$, $(2,4,4)$, or $(2,3,6)$.
\end{enumerate}
We also mention that the number of $\mathbb{P}_1(K)$-families
of quasi-tubes in $\Gamma_A$ containing projective
modules in controlled by the reflection sequence of tubular
algebras given by a tubular algebra $B$ and the generator
of the admissible infinite cyclic automorphism group
$G$ of $\widehat{B}$ with $\widehat{B}/G$ being socle
equivalent to $A$ (see \cite{Bi,BiS1,BiS2,LS2,NS,S2}).

\section{Exceptional periodic algebras of polynomial growth}%
\label{sec:exceptional}

We  introduce the exceptional periodic algebras
of polynomial growth and describe their basic properties.

Consider the following family of bound quiver algebras
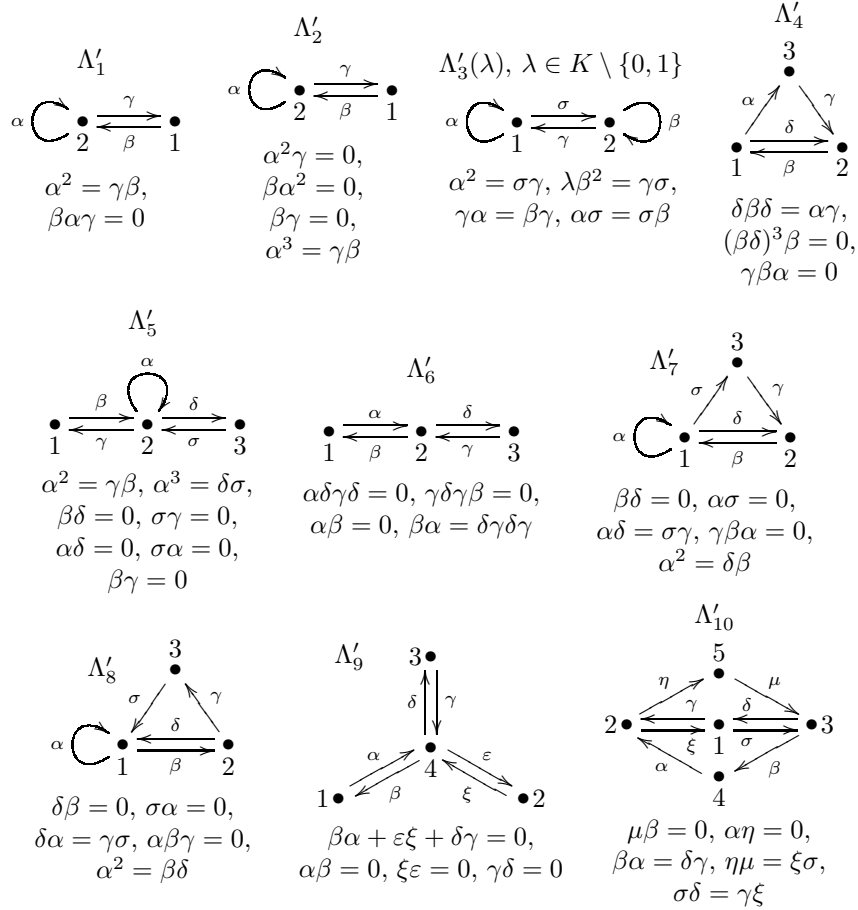
\[
\begin{minipage}{.2\columnwidth}
\center
$\Lambda'_1$
\[
\xymatrix{
  \bullet \ar@(dl,ul)[]^{\alpha} \ar@<+.5ex>[r]^{\gamma} \save[] +<0mm,-2.8mm> *{2} \restore
  & \bullet \ar@<+.5ex>[l]^{\beta} \save[] +<0mm,-2.8mm> *{1} \restore
}
\]
$\alpha^2 = \gamma \beta$,
$\beta \alpha \gamma = 0$
\end{minipage}
\quad
\begin{minipage}{.2\columnwidth}
\center
$\Lambda'_2$
\[
\xymatrix{
  \bullet \ar@(dl,ul)[]^{\alpha} \ar@<+.5ex>[r]^{\gamma} \save[] +<0mm,-2.8mm> *{2} \restore
  & \bullet \ar@<+.5ex>[l]^{\beta} \save[] +<0mm,-2.8mm> *{1} \restore
}
\]
\center
$\alpha^2 \gamma = 0$,
$\beta \alpha^2 = 0$,
$\beta \gamma = 0$,
$\alpha^3 = \gamma \beta$
\end{minipage}
\quad
\begin{minipage}{.27\columnwidth}
\center
\mbox{}\!$\Lambda'_3(\lambda)$, $\lambda \in K \setminus \{ 0, 1 \}$\!\!
\[
\xymatrix{ \bullet \ar@(dl,ul)[]^{\alpha} \ar@<+.5ex>[r]^{\sigma} \save[] +<0mm,-2.8mm> *{1} \restore
       & \bullet \ar@<+.5ex>[l]^{\gamma} \ar@(ur,dr)[]^{\beta} \save[] +<0mm,-2.8mm> *{2} \restore
}
\]
$\alpha^2 = \sigma \gamma$,
$\lambda \beta^2 = \gamma \sigma$,
$\gamma \alpha = \beta \gamma$,
$\alpha \sigma = \sigma \beta$
\end{minipage}
\quad
\begin{minipage}{.15\columnwidth}%
\vspace{1mm}
\center
$\Lambda'_4$%
\[
  \xymatrix@C=.75pc@R=1.5pc{
      & \bullet \ar@<+0.ex>[dr]^{\gamma} \save[] +<0mm,2.8mm> *{3} \restore \\
      \bullet \ar@<+0.ex>[ur]^{\alpha} \ar@<+.5ex>[rr]^{\delta} \save[] +<0mm,-2.8mm> *{1} \restore
        && \bullet \ar@<+.5ex>[ll]^{\beta} \save[] +<0mm,-2.8mm> *{2} \restore
  }
\]
$\delta \beta \delta = \alpha \gamma$,
$(\beta \delta)^3 \beta = 0$,
$\gamma \beta \alpha = 0$
\end{minipage}
\]
\[
\begin{minipage}{.26\columnwidth}%
\vspace{1mm}
\center
$\Lambda'_5$
\[
  \xymatrix{
     \bullet \ar@<+.5ex>[r]^{\beta} \save[] +<0mm,-2.8mm> *{1} \restore
       & \bullet \ar@(ul,ur)[]^{\alpha} \ar@<+.5ex>[l]^{\gamma} \ar@<+.5ex>[r]^{\delta}
          \save[] +<0mm,-2.8mm> *{2} \restore
       & \bullet \ar@<+.5ex>[l]^{\sigma} \save[] +<0mm,-2.8mm> *{3} \restore
  }
\]
\center
$\alpha^2 = \gamma \beta$,
$\alpha^3 = \delta \sigma$,
$\beta \delta = 0$,
$\sigma \gamma = 0$,
$\alpha \delta = 0$,
$\sigma \alpha = 0$,
$\beta \gamma = 0$
\end{minipage}
\quad
\begin{minipage}{.26\columnwidth}
\center
$\Lambda'_6$%
\[
  \xymatrix{ \bullet \ar@<+.5ex>[r]^{\alpha} \save[] +<0mm,-2.8mm> *{1} \restore
	 & \bullet \ar@<+.5ex>[l]^{\beta} \ar@<+.5ex>[r]^{\delta}  \save[] +<0mm,-2.8mm> *{2} \restore
 	 & \bullet \ar@<+.5ex>[l]^{\gamma}  \save[] +<0mm,-2.8mm> *{3} \restore
}
\]
$\alpha \delta \gamma \delta = 0$,
$\gamma \delta \gamma \beta = 0$,
$\alpha \beta = 0$,
$\beta \alpha = \delta \gamma \delta \gamma$
\end{minipage}
\quad
\begin{minipage}{.28\columnwidth}%
\center
\[
  \xymatrix@C=.75pc@R=1.5pc{
      \save[] *\txt{\!\!\!\!\!\!\!\!$\Lambda'_{7}$} \restore
      & \bullet \ar@<+0.ex>[dr]^{\gamma}  \save[] +<0mm,2.8mm> *{3} \restore \\
      \bullet \ar@(dl,ul)[]^{\alpha} \ar@<+0.ex>[ur]^{\sigma} \ar@<+.5ex>[rr]^{\delta}
             \save[] +<0mm,-2.8mm> *{1} \restore
        && \bullet \ar@<+.5ex>[ll]^{\beta}   \save[] +<0mm,-2.8mm> *{2} \restore
  }
\]
$\beta \delta = 0$,
$\alpha \sigma = 0$,
$\alpha \delta = \sigma \gamma$,
$\gamma \beta \alpha = 0$,
$\alpha^2 = \delta \beta$
\end{minipage}
\]
\[
\begin{minipage}{.27\columnwidth}%
\center
\[
  \xymatrix@C=.75pc@R=1.5pc{
      \save[] *\txt{\!\!\!\!\!\!\!\!$\Lambda'_{8}$} \restore
      & \bullet \ar@{<-}@<+0.ex>[dr]^{\gamma}  \save[] +<0mm,2.8mm> *{3} \restore \\
      \bullet \ar@(dl,ul)[]^{\alpha} \ar@{<-}@<+0.ex>[ur]^{\sigma} \ar@{<-}@<+.5ex>[rr]^{\delta}
         \save[] +<0mm,-2.8mm> *{1} \restore
        && \bullet \ar@{<-}@<+.5ex>[ll]^{\beta}  \save[] +<0mm,-2.8mm> *{2} \restore
  }
\]
$\delta \beta = 0$,
$\sigma \alpha = 0$,
$\delta \alpha = \gamma \sigma$,
$\alpha \beta \gamma = 0$,
$\alpha^2 = \beta \delta$
\end{minipage}
\quad
\begin{minipage}{.28\columnwidth}%
\center
\[
  \xymatrix@C=2pc@R=.75pc{
      \save[] *\txt{$\Lambda'_9$\!\!\!\!\!} \restore
      & \bullet  \ar@<+.5ex>[dd]^{\gamma} \save[] +<-2.1mm,0mm> *{3} \restore \\ \\
      & \bullet  \ar@<+.5ex>[uu]^{\delta} \save[] +<0mm,-2.8mm> *{4} \restore
           \ar@<+.5ex>[dr]^{\varepsilon}
           \ar@<+.5ex>[dl]^{\beta} \\
      \bullet  \ar@<+.5ex>[ru]^{\alpha} \save[] +<-2.1mm,0mm> *{1} \restore
      &&  \bullet  \ar@<+.5ex>[lu]^{\xi} \save[] +<2.1mm,0mm> *{2} \restore
  }
\]
$\beta \alpha + \varepsilon \xi + \delta \gamma =  0$,
$\alpha \beta = 0$,
$\xi \varepsilon = 0$,
$\gamma \delta = 0$\vspace{1mm}
\end{minipage}
\quad
\begin{minipage}{.27\columnwidth}%
\center
$\Lambda'_{10}$
\[
  \xymatrix@C=2pc@R=.75pc{
      & \bullet  \ar@<+.25ex>[dr]^{\mu} \save[] +<0mm,2.8mm> *{5} \restore \\
      \bullet \ar@<+.25ex>[ur]^{\eta} \ar@<-.5ex>[r]_(.7){\xi}
             \save[] +<-2.1mm,0mm> *{2} \restore
        & \bullet \ar@<-.5ex>[l]_(.3){\gamma} \ar@<-.5ex>[r]_(.3){\sigma}
             \save[] +<0mm,-2.8mm> *{1} \restore
        & \bullet  \ar@<-.5ex>[l]_(.7){\delta} \ar@<+.25ex>[dl]^{\beta}
            \save[] +<2.1mm,0mm> *{3} \restore \\
      & \bullet  \ar@<+.5ex>[lu]^{\alpha} \save[] +<0mm,-2.8mm> *{4} \restore
  }
\]
$\mu \beta = 0$,
$\alpha \eta = 0$,
$\beta \alpha = \delta \gamma$,
$\eta \mu = \xi \sigma$,
$\sigma \delta = \gamma \xi$
\end{minipage}
\]

The following theorem describes the basic properties of the above algebras.

\begin{theorem}
\label{thm:4.1}
The following statements hold.
\begin{enumerate}[(i)]
\item
  $\Lambda'_1$ and $\Lambda'_2$ are non-isomorphic, derived equivalent,
  standard symmetric algebras of tubular type $(3,3,3)$.
\item
  For any $\lambda \in K \setminus \{ 0,1 \}$, $\Lambda'_3(\lambda)$ is a
  standard symmetric algebra of tubular type $(2,2,2,2)$.
\item
  $\Lambda'_4$, $\Lambda'_5$, $\Lambda'_6$, $\Lambda'_7$, $\Lambda'_8$
  are pairwise non-isomorphic, derived equivalent,
  standard symmetric algebras of tubular type $(2,4,4)$.
\item
  $\Lambda'_9$ is a standard weakly symmetric algebra
  of tubular type $(3,3,3)$.
  Moreover, $\Lambda'_9$ is a symmetric algebra
  if and only if $\charact K = 2$.
\item
  $\Lambda'_{10}$ is a standard self-injective algebra
  of tubular type $(2,3,6)$.
\end{enumerate}
\end{theorem}

\begin{proof}
The statements on the derived equivalences follow from
\cite[Theorem]{BHS1} and its proof.
The remaining statemets are consequences of
\cite[Theorems 1 and 2]{BiS2}.
\end{proof}

The following theorem summarizes  results proved in 
\cite[Sections~7-10]{BES2}.

\begin{theorem}
\label{thm:4.2-new}
The following statements hold.
\begin{enumerate}[(i)]
\item
  $\Lambda'_1$ and $\Lambda'_2$ are periodic algebras
  of period $6$.
\item
  For any $\lambda \in K \setminus \{ 0,1 \}$, $\Lambda'_3(\lambda)$ is a
  periodic algebra of period $4$. 
\item
  $\Lambda'_4$, $\Lambda'_5$, $\Lambda'_6$, $\Lambda'_7$, $\Lambda'_8$
  are periodic algebras of period $8$.
\item
  $\Lambda'_9$ is a periodic algebra of period $3$ 
   if $\charact K = 2$ and of period $6$ if $\charact K \neq 2$.   
\item
  $\Lambda'_{10}$ is a periodic algebra of period $3$ 
   if $\charact K = 2$ and of period $6$ if $\charact K \neq 2$.   
\end{enumerate}
\end{theorem}

For a positive integer $d$, we denote by $\alg_d(K)$
the affine variety of associative algebra structures with identity
on the affine space $K^d$. The general linear group $\GL_d(K)$
acts on $\alg_d(K)$ by transport of structure, and the $\GL_d(K)$-orbits
in $\alg_d(K)$ correspond to the isomorphism classes of $d$-dimensional
algebras (see \cite{Kr} for more details).
For two $d$-dimensional algebras $\Lambda$ and $A$, we say that
$\Lambda$ is a \emph{deformation} of $A$ (or $A$ is a \emph{degeneration} of $\Lambda$)
if $A$ belongs to the
closure of the $\GL_d(K)$-orbit of $\Lambda$ in the Zariski topology
in $\alg_d(K)$.

Non-standard, basic, indecomposable, periodic representation-infinite
self-injective algebras of polynomial growth
occur only in characteristic 2 and 3 and are (socle)
deformations of the corresponding standard self-injective algebras
of tubular type, introduced above.

Consider the following family of bound quiver algebras
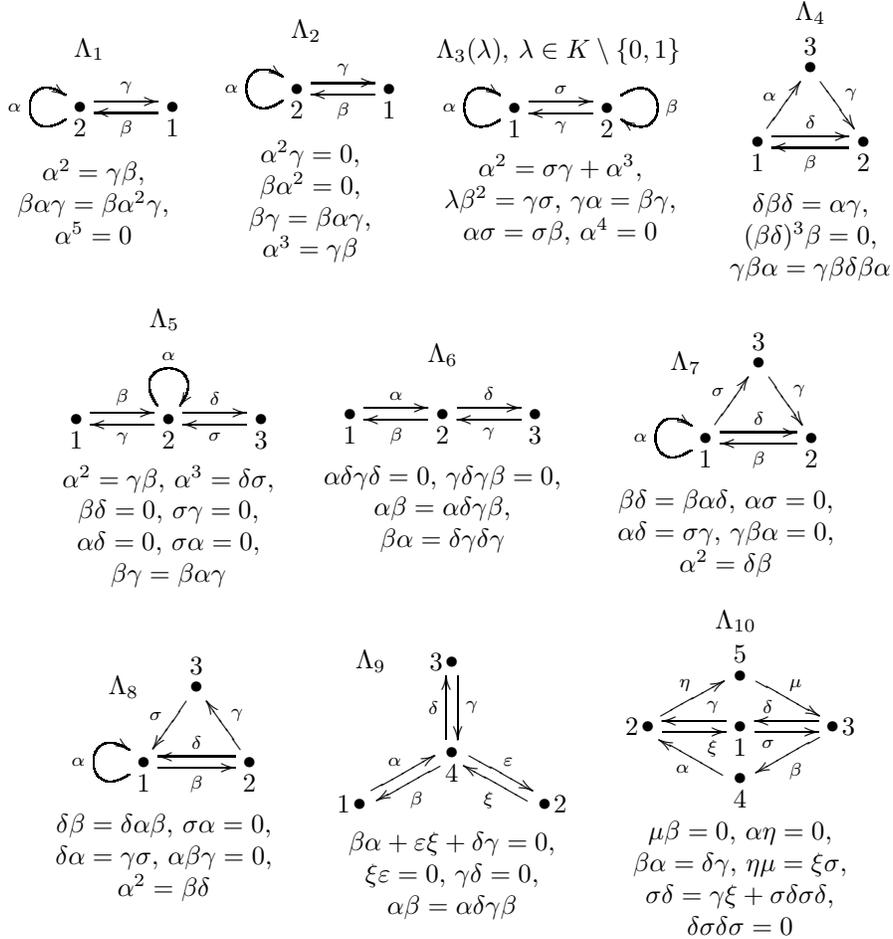
\[
\begin{minipage}{.2\columnwidth}
\center
$\Lambda_1$
\[
\xymatrix{
  \bullet \ar@(dl,ul)[]^{\alpha} \ar@<+.5ex>[r]^{\gamma} \save[] +<0mm,-2.8mm> *{2} \restore
  & \bullet \ar@<+.5ex>[l]^{\beta} \save[] +<0mm,-2.8mm> *{1} \restore
}
\]
$\alpha^2 = \gamma \beta$,
$\beta \alpha \gamma = \beta \alpha^2 \gamma$,
$\alpha^5 = 0$
\end{minipage}
\quad
\begin{minipage}{.2\columnwidth}
\center
$\Lambda_2$
\[
\xymatrix{
  \bullet \ar@(dl,ul)[]^{\alpha} \ar@<+.5ex>[r]^{\gamma} \save[] +<0mm,-2.8mm> *{2} \restore
  & \bullet \ar@<+.5ex>[l]^{\beta} \save[] +<0mm,-2.8mm> *{1} \restore
}
\]
\center
$\alpha^2 \gamma = 0$,
$\beta \alpha^2 = 0$,
$\beta \gamma = \beta \alpha \gamma$,
$\alpha^3 = \gamma \beta$
\end{minipage}
\quad
\begin{minipage}{.27\columnwidth}
\center
\mbox{}\!$\Lambda_3(\lambda)$, $\lambda \in K \setminus \{ 0, 1 \}$\!\!
\[
\xymatrix{ \bullet \ar@(dl,ul)[]^{\alpha} \ar@<+.5ex>[r]^{\sigma} \save[] +<0mm,-2.8mm> *{1} \restore
       & \bullet \ar@<+.5ex>[l]^{\gamma} \ar@(ur,dr)[]^{\beta} \save[] +<0mm,-2.8mm> *{2} \restore
}
\]
$\alpha^2 = \sigma \gamma + \alpha^3$,
$\lambda \beta^2 = \gamma \sigma$,
$\gamma \alpha = \beta \gamma$,
$\alpha \sigma = \sigma \beta$,
$\alpha^4 =  0$
\end{minipage}
\quad
\begin{minipage}{.2\columnwidth}%
\center
$\Lambda_4$%
\[
  \xymatrix@C=.75pc@R=1.5pc{
      & \bullet \ar@<+0.ex>[dr]^{\gamma} \save[] +<0mm,2.8mm> *{3} \restore \\
      \bullet \ar@<+0.ex>[ur]^{\alpha} \ar@<+.5ex>[rr]^{\delta} \save[] +<0mm,-2.8mm> *{1} \restore
        && \bullet \ar@<+.5ex>[ll]^{\beta} \save[] +<0mm,-2.8mm> *{2} \restore
  }
\]
$\delta \beta \delta = \alpha \gamma$,
$(\beta \delta)^3 \beta = 0$,
$\gamma \beta \alpha = \gamma \beta \delta \beta \alpha$
\end{minipage}
\]
\[
\begin{minipage}{.26\columnwidth}
\center
$\Lambda_5$
\[
  \xymatrix{
     \bullet \ar@<+.5ex>[r]^{\beta} \save[] +<0mm,-2.8mm> *{1} \restore
       & \bullet \ar@(ul,ur)[]^{\alpha} \ar@<+.5ex>[l]^{\gamma} \ar@<+.5ex>[r]^{\delta}
          \save[] +<0mm,-2.8mm> *{2} \restore
       & \bullet \ar@<+.5ex>[l]^{\sigma} \save[] +<0mm,-2.8mm> *{3} \restore
  }
\]
\center
$\alpha^2 = \gamma \beta$,
$\alpha^3 = \delta \sigma$,
$\beta \delta = 0$,
$\sigma \gamma = 0$,
$\alpha \delta = 0$,
$\sigma \alpha = 0$,
$\beta \gamma = \beta \alpha \gamma$
\end{minipage}
\quad
\begin{minipage}{.26\columnwidth}
\center
$\Lambda_6$%
\[
  \xymatrix{ \bullet \ar@<+.5ex>[r]^{\alpha} \save[] +<0mm,-2.8mm> *{1} \restore
	 & \bullet \ar@<+.5ex>[l]^{\beta} \ar@<+.5ex>[r]^{\delta}  \save[] +<0mm,-2.8mm> *{2} \restore
 	 & \bullet \ar@<+.5ex>[l]^{\gamma}  \save[] +<0mm,-2.8mm> *{3} \restore
   }
\]
$\alpha \delta \gamma \delta = 0$,
$\gamma \delta \gamma \beta = 0$,
$\alpha \beta = \alpha \delta \gamma \beta$,
$\beta \alpha = \delta \gamma \delta \gamma$
\end{minipage}
\quad
\begin{minipage}{.28\columnwidth}%
\center
\[
  \xymatrix@C=.75pc@R=1.5pc{
      \save[] *\txt{\!\!\!\!\!\!\!\!$\Lambda_{7}$} \restore
      & \bullet \ar@<+0.ex>[dr]^{\gamma}  \save[] +<0mm,2.8mm> *{3} \restore \\
      \bullet \ar@(dl,ul)[]^{\alpha} \ar@<+0.ex>[ur]^{\sigma} \ar@<+.5ex>[rr]^{\delta}
             \save[] +<0mm,-2.8mm> *{1} \restore
        && \bullet \ar@<+.5ex>[ll]^{\beta}   \save[] +<0mm,-2.8mm> *{2} \restore
  }
\]
$\beta \delta = \beta \alpha \delta$,
$\alpha \sigma = 0$,
$\alpha \delta = \sigma \gamma$,
$\gamma \beta \alpha = 0$,
$\alpha^2 = \delta \beta$
\end{minipage}
\]
\[
\begin{minipage}{.27\columnwidth}%
\center
\[
  \xymatrix@C=.75pc@R=1.5pc{
      \save[] *\txt{\!\!\!\!\!\!\!\!$\Lambda_{8}$} \restore
      & \bullet \ar@{<-}@<+0.ex>[dr]^{\gamma}  \save[] +<0mm,2.8mm> *{3} \restore \\
      \bullet \ar@(dl,ul)[]^{\alpha} \ar@{<-}@<+0.ex>[ur]^{\sigma} \ar@{<-}@<+.5ex>[rr]^{\delta}
         \save[] +<0mm,-2.8mm> *{1} \restore
        && \bullet \ar@{<-}@<+.5ex>[ll]^{\beta}  \save[] +<0mm,-2.8mm> *{2} \restore
  }
\]
$\delta \beta = \delta \alpha \beta$,
$\sigma \alpha = 0$,
$\delta \alpha = \gamma \sigma$,
$\alpha \beta \gamma = 0$,
$\alpha^2 = \beta \delta$
\end{minipage}
\quad
\begin{minipage}{.28\columnwidth}%
\center
\[
  \xymatrix@C=2pc@R=.75pc{
      \save[] *\txt{$\Lambda_9$\!\!\!\!\!} \restore
      & \bullet  \ar@<+.5ex>[dd]^{\gamma} \save[] +<-2.1mm,0mm> *{3} \restore \\ \\
      & \bullet  \ar@<+.5ex>[uu]^{\delta} \save[] +<0mm,-2.8mm> *{4} \restore
           \ar@<+.5ex>[dr]^{\varepsilon}
           \ar@<+.5ex>[dl]^{\beta} \\
      \bullet  \ar@<+.5ex>[ru]^{\alpha} \save[] +<-2.1mm,0mm> *{1} \restore
      &&  \bullet  \ar@<+.5ex>[lu]^{\xi} \save[] +<2.1mm,0mm> *{2} \restore
  }
\]
$\beta \alpha + \varepsilon \xi + \delta \gamma =  0$,
$\xi \varepsilon = 0$,
$\gamma \delta = 0$,
$\alpha \beta = \alpha \delta \gamma \beta$
\end{minipage}
\quad
\begin{minipage}{.27\columnwidth}%
\center
$\Lambda_{10}$
\[
  \xymatrix@C=2pc@R=.75pc{
      & \bullet  \ar@<+.25ex>[dr]^{\mu} \save[] +<0mm,2.8mm> *{5} \restore \\
      \bullet \ar@<+.25ex>[ur]^{\eta} \ar@<-.5ex>[r]_(.7){\xi}
             \save[] +<-2.1mm,0mm> *{2} \restore
        & \bullet \ar@<-.5ex>[l]_(.3){\gamma} \ar@<-.5ex>[r]_(.3){\sigma}
             \save[] +<0mm,-2.8mm> *{1} \restore
        & \bullet  \ar@<-.5ex>[l]_(.7){\delta} \ar@<+.25ex>[dl]^{\beta}
            \save[] +<2.1mm,0mm> *{3} \restore \\
      & \bullet  \ar@<+.5ex>[lu]^{\alpha} \save[] +<0mm,-2.8mm> *{4} \restore
  }
\]
$\mu \beta = 0$,
$\alpha \eta = 0$,
$\beta \alpha = \delta \gamma$,
$\eta \mu = \xi \sigma$,
$\sigma \delta = \gamma \xi + \sigma \delta \sigma \delta$,
$\delta \sigma \delta \sigma = 0$
\end{minipage}
\]

We have the following consequence of \cite[Theorem~1.1]{BiS3}.

\begin{theorem}
\label{thm:4.2}
Let $\Lambda$ be a basic, indecomposable, self-injective algebra.
Then $\Lambda$ is a non-standard algebra socle equivalent to
a standard self-injective algebra of tubular type if and only
if one of the following cases holds:
\begin{enumerate}[(i)]
\item
  $K$ is of characteristic $3$ and $\Lambda$ is isomorphic
  to $\Lambda_1$ or $\Lambda_2$;
\item
  $K$ is of characteristic $2$ and $\Lambda$ is isomorphic
  to one of algebras
  $\Lambda_3(\lambda)$, $\Lambda_4$, $\Lambda_5$, $\Lambda_6$,
  $\Lambda_7$, $\Lambda_8$, $\Lambda_9$, $\Lambda_{10}$.
\end{enumerate}
\end{theorem}

Moreover, we have the following consequence of
\cite[Theorem 3.1]{BHS2} and \cite[Section 5]{BiS3}.

\begin{theorem}
\label{thm:4.3}
The following statements hold.
\begin{enumerate}[(i)]
\item
  $\Lambda_1$ and $\Lambda_2$ are non-isomorphic but derived equivalent
  algebras.
\item
  $\Lambda_4$, $\Lambda_5$, $\Lambda_6$, $\Lambda_7$, $\Lambda_8$
  are pairwise non-isomorphic but derived equivalent
  algebras.
\end{enumerate}
\end{theorem}

The next theorem establishes a relationship between the above
families of algebras (see \cite[Theorem 2.1]{BHS2} and \cite[Section 5]{BiS3}).

\begin{theorem}
\label{thm:4.4}
Let $\Lambda$ be one of the algebras
$\Lambda_1$, $\Lambda_2$, $\Lambda_3(\lambda)$, $\lambda \in K \setminus \{ 0,1 \}$,
$\Lambda_4$, $\Lambda_5$, $\Lambda_6$,
$\Lambda_7$, $\Lambda_8$, $\Lambda_9$, $\Lambda_{10}$,
and $\Lambda'$ be the corresponding algebra
$\Lambda'_1$, $\Lambda'_2$, $\Lambda'_3(\lambda)$, $\lambda \in K \setminus \{ 0,1 \}$,
$\Lambda'_4$, $\Lambda'_5$, $\Lambda'_6$,
$\Lambda'_7$, $\Lambda'_8$, $\Lambda'_9$, $\Lambda'_{10}$.
Then the following statements hold.
\begin{enumerate}[(i)]
\item
  $\dim_K \Lambda = \dim_K \Lambda'$.
\item
  $\Lambda$ is a deformation of $\Lambda'$.
\item
  $\Lambda$ and $\Lambda'$ are socle equivalent.
\end{enumerate}
Moreover, we have
\begin{enumerate}[(i)]
\addtocounter{enumi}{3}
\item
  For $i \in \{ 1,2 \}$, $\Lambda_i \cong \Lambda'_i$
  if and only if $\charact K \neq 3$.
\item
  For $\lambda \in K \setminus \{ 0,1 \}$,
  $\Lambda_3(\lambda) \cong \Lambda'_3(\lambda)$
  if and only if $\charact K \neq 2$.
\item
  For $i \in \{ 4,5,6,7,8,9,10 \}$, $\Lambda_i \cong \Lambda'_i$
  if and only if $\charact K \neq 2$.
\end{enumerate}
\end{theorem}

In the above notation, the algebra $\Lambda'$ is called the
\emph{standard form} of the algebra $\Lambda$.

We also know whether or not a 
non-standard algebras as above is symmetric
(see \cite[Proposition~6.8]{BES2}).

\begin{proposition}
\label{prop:4.5}
The following statements hold.
\begin{enumerate}[(i)]
\item
  For $i \in \{ 1,2,4,5,6,7,8 \}$, $\Lambda_i$
  is a symmetric algebra.
\item
  For $\lambda \in K \setminus \{ 0,1 \}$,
  $\Lambda_3(\lambda)$
  is a symmetric algebra.
\item
  $\Lambda_9$ is a weakly symmetric but not symmetric algebra.
\item
  $\Lambda_{10}$ is not a weakly symmetric algebra.
\end{enumerate}
\end{proposition}

The following describes the periods of the algebras
in question, proved in 
\cite[Sections~7-10]{BES2}.

\begin{theorem}
\label{thm:4.7}
The following statements hold.
\begin{enumerate}[(i)]
\item
  $\Lambda_1$ and $\Lambda_2$ are periodic algebras
  of period $6$.
\item
  For any $\lambda \in K \setminus \{ 0,1 \}$, $\Lambda_3(\lambda)$ is a
  periodic algebra of period $4$. 
\item
  $\Lambda_4$, $\Lambda_5$, $\Lambda_6$, $\Lambda_7$, $\Lambda_8$
  are periodic algebras of period $8$.
\item
  $\Lambda_9$ is a periodic algebra of period $6$.   
\item
  $\Lambda_{10}$ is a periodic algebra of period $6$.   
\end{enumerate}
\end{theorem}

\section{Hochschild cohomology for tubular type $(2,2,2,2)$}%
\label{sec:type2222}

In this section we determine the dimensions of the low Hochschild cohomology spaces
for the exceptional periodic algebras $\Lambda_3'(\lambda)$ and $\Lambda_3(\lambda)$
(in characteristic $2$)
of tubular type $(2,2,2,2)$.

\begin{proposition}
\label{prop:5.1}
Let $A = \Lambda'_3(\lambda)$ for some
$\lambda \in K \setminus \{ 0, 1 \}$.
Then
\begin{enumerate}[(i)]
\item
  $\dim_K HH^0(A) = 6$.
\item
  $\dim_K HH^1(A) = \left\{ \begin{array}{cl} 4, & \charact (K) \neq 2 \\ 6, & \charact (K) = 2 \end{array} \right.$.\vspace{2pt}
\item
  $\dim_K HH^2(A) = \left\{ \begin{array}{cl} 4, & \charact (K) \neq 2 \\ 6, & \charact (K) = 2 \end{array} \right.$.
\end{enumerate}
\end{proposition}

\begin{proof}
It follows from the proof of \cite[Proposition~7.1]{BES2} that $A$ 
admits the first three terms in a minimal projective resolution in $\mod A^e$
\[
  \bP_2 \xrightarrow{d_2}
  \bP_1 \xrightarrow{d_1}
  \bP_0 \xrightarrow{d_0}
  A \rightarrow 0
\]
with
\begin{align*}
 \bP_0 &= 
   P(1,1) \oplus P(2,2)
,
\\
 \bP_1 = \bP_2 &=
   P(1,1) \oplus P(1,2) \oplus P(2,1) \oplus P(2,2)
,
\end{align*}
the differential
$R = d_2$
given by
\begin{align*}
 R(e_1 \otimes e_1)
  &= \varrho(\alpha^2 - \sigma \gamma)
   = e_1 \otimes \alpha + \alpha \otimes e_1
     - e_1 \otimes \gamma - \sigma \otimes e_1
,\\
 R(e_1 \otimes e_2)
  &= \varrho(\alpha \sigma - \sigma \beta)
   = e_1 \otimes \sigma + \alpha \otimes e_2
     - e_1 \otimes \beta - \sigma \otimes e_2
,\\
 R(e_2 \otimes e_1)
  &= \varrho(\gamma \alpha - \beta \gamma)
   = e_2 \otimes \alpha + \gamma \otimes e_1
     - e_2 \otimes \gamma - \beta \otimes e_1
,\\
 R(e_2 \otimes e_2)
  &= \varrho(\lambda \beta^2 - \gamma \sigma)
   = \lambda e_2 \otimes \beta + \lambda \beta \otimes e_2
     - e_2 \otimes \sigma - \gamma \otimes e_2
,
\end{align*}
and the $A$-$A$-bimodule
$\Omega_{A^e}^3(A) = \Ker R$
generated by the following elements in $\bP_2$
\begin{align*}
 \Gamma_1 &=
  \sigma (e_2 \otimes e_1) - (e_1 \otimes e_2) \gamma
  - \alpha (e_1 \otimes e_1) + (e_1 \otimes e_1) \alpha
,
\\
 \Gamma_2 &=
  \gamma (e_1 \otimes e_2) - (e_2 \otimes e_1) \sigma
  + \beta (e_2 \otimes e_2) - (e_2 \otimes e_2) \beta
.
\end{align*}

We fix also the following basis
$\cB = e_1 \cB \cup e_2 \cB$
of the $K$-vector space
$A$:
\begin{align*}
e_1 \cB &=
\{ e_1, \alpha, \alpha^2, \alpha^3, \sigma, \sigma \beta \}
,
\\
e_2 \cB &=
\{ e_2, \beta, \beta^2, \beta^3, \gamma, \gamma \alpha \}
.
\end{align*}

We shall show now that the equalities (i), (ii), (iii) hold.

\smallskip

(i)
We have $HH^0(A) = C(A)$.
Clearly, the center $C(A)$ contains the identity
and the socle $\soc(A)$ of $A$.
Observe that any element in the radical $\rad A$ of $A$
which does not have a component in $\soc (A)$ is of the
form $x=x_1+x_2$, where
\[
  x_1=c_1\alpha + c_2\alpha^2  \in e_1 A e_1,
\quad
  x_2=d_1\beta + d_2\beta^2  \in e_2 A e_2.
\]
Then $x$ commutes with $\alpha$ and $\beta$.
Moreover, $x$ commutes with $\sigma$
and $\gamma$ if and only if $c_1=d_1$.
Hence, the center has the basis
\[
  \{ 1, \alpha+\beta, \alpha^2, \alpha^3, \beta^2, \beta^3\}
\]
over $K$, and consequently $\dim_K HH^0(A)=6$.

\smallskip

(ii)
We will determine $\dim_K HH^1(A)$ using the exact sequence
of $K$-vector spaces (see Proposition~\ref{prop:2.5})
\[
 0
  \to HH^0(A)
  \to \Hom_{A^e}(\bP_0,A)
  \to \Hom_{A^e}\big(\Omega_{A^e}(A),A\big)
  \to HH^1(A)
  \to 0
  .
\]
We have isomorphisms of $K$-vector spaces
\begin{align*}
 \Hom_{A^e} (\bP_0, A) &=
   e_1 A e_1  \oplus e_2 A e_2
,
\\
 \Hom_{A^e} (\bP_1, A) &=
   e_1 A e_1  \oplus e_1 A e_2 \oplus
   e_2 A e_1  \oplus e_2 A e_2 ,
\end{align*}
and hence we obtain
\begin{align*}
 \dim_K \Hom_{A^e} (\bP_0, A) &=
   4+4=8
,
\\
 \dim_K \Hom_{A^e} (\bP_1, A) &=
   4+2+2+4=12
.
\end{align*}
In order to calculate
$\dim_K \Hom_{A^e}(\Omega_{A^e}(A),A)$,
we identify $\Hom_{A^e}(\Omega_{A^e}(A),A)$
with the $K$-vector space
$\{ \varphi \in  \Hom_{A^e} (\bP_1, A) \,|\, \varphi R=0\}$.
Let $\varphi: \bP_1\to A$ be a homomorphism in $\mod A^e$.
Then there are
$a_1, a_2, b_1, b_2, c_0, c_1, c_2, c_3, d_0, d_1, d_2, d_3 \in K$ such that
\begin{align*}
  \varphi(e_1\otimes e_2) &= a_1\sigma + a_2\sigma\beta,\\
  \varphi(e_2\otimes e_1) &= b_1\gamma + b_2\gamma\alpha,\\
  \varphi(e_1\otimes e_1) &= \sum_{i=0}^3 c_i\alpha^i,\\
  \varphi(e_2\otimes e_2) &= \sum_{i=0}^3 d_i\beta^i,
\end{align*}
where $\alpha^0 = e_1$ and $\beta^0 = e_2$.
We have the equalities
\begin{align*}
  \varphi \big( R(e_1\otimes e_2) \big)
  &=
   \bigg(\sum_{i=0}^3 c_i\alpha^i\bigg)\sigma
   + \alpha(a_1\sigma + a_2\sigma\beta)
   - (a_1\sigma+a_2\sigma\beta)\beta
   - \sigma\bigg(\sum_{i=0}^3 d_i\beta^i\bigg)
 \\
  &= (c_0-d_0)\sigma + (c_1-d_1)\sigma\beta
.
\end{align*}
Hence, $\varphi ( R(e_1\otimes e_2) ) = 0$
if and only if $c_0 = d_0$ and $c_1 = d_1$.
Similarly, we have the equalities
\begin{align*}
  \varphi \big( R(e_2\otimes e_1) \big)
  &=
   (b_1\gamma + b_2\gamma\alpha)\alpha
   + \gamma \bigg(\sum_{i=0}^3 c_i\alpha^i\bigg)
   - \bigg(\sum_{i=0}^3 d_i\beta^i\bigg)\gamma
   - \beta(b_1\gamma+b_2\gamma\alpha)\beta
 \\
  &= (c_0-d_0)\gamma + (c_1-d_1)\gamma\alpha
.
\end{align*}
Hence, again $\varphi ( R(e_2\otimes e_1) ) = 0$
if and only if $c_0 = d_0$ and $c_1 = d_1$.
Further, we obtain
\begin{align*}
  \varphi \big( R(e_2\otimes e_2) \big)
  &=
   \lambda \bigg(\sum_{i=0}^3 d_i\beta^i\bigg)\beta
   + \lambda \beta\bigg(\sum_{i=0}^3 d_i\beta^i\bigg)
   - (b_1\gamma + b_2\gamma\alpha)\sigma
   - \gamma(a_1\sigma+a_2\sigma\beta)
 \\
  &= \lambda \Big (2 d_0 \beta + (2 d_1 - a_1 - b_1) \beta^2 + (2 d_2 - a_2 - b_2) \beta^3 \Big)
.
\end{align*}
Hence, $\varphi ( R(e_2\otimes e_2) ) = 0$
if and only if
\[
  2 d_0 = 0,
 \quad
  a_1 + b_1 = 2 d_1,
 \quad
  a_2 + b_2 = 2 d_2.
\]
Finally, we have the equalities
\begin{align*}
  \varphi \big( R(e_1\otimes e_1) \big)
  &=
   \bigg(\sum_{i=0}^3 c_i\alpha^i\bigg)\alpha
   + \alpha \bigg(\sum_{i=0}^3 c_i\alpha^i\bigg)
   - (a_1\sigma+a_2\sigma\beta)\gamma
   - \sigma(b_1\gamma + b_2\gamma\alpha)
 \\
  &= 2 c_0 \alpha + (2 c_1 - a_1 - b_1) \alpha^2 + (2 c_2 - a_2 - b_2) \alpha^3
.
\end{align*}
Thus $\varphi ( R(e_1\otimes e_1) ) = 0$
if and only if
\[
  2 c_0 = 0,
 \quad
  a_1 + b_1 = 2 c_1,
 \quad
  a_2 + b_2 = 2 c_2.
\]

Assume $\charact(K) \neq 2$.
Then $\varphi$ induces a homomorphism
$\Omega_{A^e}(A)\to A$ in $\mod A^e$ if and only if
\[
  c_0=d_0=0,
  \ \,
  c_1=d_1,
  \ \,
  c_2=d_2,
  \ \,
  a_1+b_1=2c_1,
  \ \,
  a_2+b_2=2c_2
  .
\]
Hence $\dim_K \Hom_{A^e}(\Omega_{A^e}(A),A) = 6$.
Therefore, we obtain
\begin{align*}
  \dim_K HH^1(A)
   &= \dim_K HH^0(A) + \dim_K \Hom_{A^e}\big(\Omega_{A^e}(A),A\big)
      - \dim_K \Hom_{A^e}(\bP_0,A)
   \\&
    = 6 + 6 - 8 = 4 .
\end{align*}

Assume $\charact(K) = 2$.
Then $\varphi$ induces a homomorphism
$\Omega_{A^e}(A)\to A$ in $\mod A^e$ if and only if
\[
  a_1+b_1=0,
  \ \,
  a_2+b_2=0,
  \ \,
  c_0=d_0,
  \ \,
  c_1=d_1
  .
\]
Hence $\dim_K \Hom_{A^e}(\Omega_{A^e}(A),A) = 8$.
Therefore, we obtain
\begin{align*}
  \dim_K HH^1(A)
   &= \dim_K HH^0(A) + \dim_K \Hom_{A^e}\big(\Omega_{A^e}(A),A\big)
      - \dim_K \Hom_{A^e}(\bP_0,A)
   \\&
    = 6 + 8 - 8 = 6 .
\end{align*}

\smallskip

(iii)
We will determine $\dim_K HH^2(A)$ using the exact sequence
of $K$-vector spaces (see Proposition~\ref{prop:2.5})
\[
 0
  \to \Hom_{A^e}\big(\Omega_{A^e}(A),A\big)
  \to \Hom_{A^e}(\bP_1,A)
  \to \Hom_{A^e}\big(\Omega_{A^e}^2(A),A\big)
  \to HH^2(A)
  \to 0
  .
\]
In order to calculate
$\dim_K \Hom_{A^e}(\Omega_{A^e}^2(A),A)$,
we identify $\Hom_{A^e}(\Omega_{A^e}^2(A),A)$
with the $K$-vector space
\[
  \Big\{ \varphi \in  \Hom_{A^e} (\bP_2, A) \,\big|\, \varphi \big(\Omega_{A^e}^3(A)\big) = 0\Big\} .
\]
Since $\bP_1 = \bP_2$, a homomorphism
$\varphi \in  \Hom_{A^e} (\bP_2, A)$
is defined as $\varphi \in  \Hom_{A^e} (\bP_1, A)$
in (ii).

We have the equalities
\begin{align*}
  \varphi (\Gamma_1)
  &=
   \sigma \varphi (e_2 \otimes e_1)
   - \varphi (e_1 \otimes e_2) \gamma
   - \alpha \varphi (e_1 \otimes e_1)
   - \varphi (e_1\otimes e_1) \alpha
 \\&
  =
   \sigma (b_1 \gamma + b_2 \gamma \alpha)
   - (a_1 \sigma + a_2 \sigma \beta) \gamma
   - \alpha \bigg(\sum_{i=0}^3 c_i\alpha^i\bigg)
   - \bigg(\sum_{i=0}^3 c_i\alpha^i\bigg)\alpha
 \\&
  =
   (b_1 - a_1) \sigma \gamma
   + (b_2 - a_2) \sigma \gamma \alpha
 \\&
  =
   (b_1 - a_1) \alpha^2
   + (b_2 - a_2) \alpha^3
,
\end{align*}
and
\begin{align*}
  \varphi (\Gamma_2)
  &=
   \gamma \varphi (e_1\otimes e_2)
   - \varphi (e_2\otimes e_1) \sigma
   + \beta \varphi (e_2\otimes e_2)
   - \varphi (e_2\otimes e_2) \beta
 \\&
  =
   \gamma (a_1 \sigma + a_2 \sigma \beta)
   - (b_1 \gamma + b_2 \gamma \alpha) \sigma
   + \beta \bigg(\sum_{i=0}^3 d_i\beta^i\bigg)
   - \bigg(\sum_{i=0}^3 d_i\beta^i\bigg) \beta
 \\&
  =
   (a_1 - b_1) \gamma \sigma
   + (a_2 - b_2) \gamma \alpha  \sigma
 \\&
  =
   \lambda (a_1 - b_1) \beta^2
   + \lambda (a_2 - b_2) \beta^3
.
\end{align*}
Hence, $\varphi(\Gamma_1) = 0$
and $\varphi (\Gamma_2) = 0$
if and only if $a_1 = b_1$ and $a_2 = b_2$.
This shows that
$\dim_K \Hom_{A^e}(\Omega_{A^e}^2(A),A) = 10$.
Using the equality
\begin{align*}
  \dim_K HH^2(A)
   &= \dim_K \Hom_{A^e}\big(\Omega_{A^e}(A),A\big)
      + \dim_K \Hom_{A^e}\big(\Omega_{A^e}^2(A),A\big)
   \\&\ \ \ \,
      - \dim_K \Hom_{A^e}(\bP_1,A)
\end{align*}
we obtain
\begin{align*}
  \dim_K HH^2(A) &= 6 + 10 - 12 = 4 \ \  \mbox{ if $\charact(K) \neq 2$}, \\
  \dim_K HH^2(A) &= 8 + 10 - 12 = 6 \ \  \mbox{ if $\charact(K) = 2$}.
\end{align*}
We note that we have always $\dim_K HH^1(A) = \dim_K HH^2(A)$.
\end{proof}

\begin{proposition}
\label{prop:5.2}
Let $K$ be an algebraically closed field of characteristic 2
and $A = \Lambda_3(\lambda)$ for some
$\lambda \in K \setminus \{ 0, 1 \}$.
Then
\begin{enumerate}[(i)]
\item
  $\dim_K HH^0(A) = 6$.
\item
  $\dim_K HH^1(A) = 4$.
\item
  $\dim_K HH^2(A) = 4$.
\end{enumerate}
\end{proposition}

\begin{proof}
It follows from the proof of \cite[Proposition~7.2]{BES2} that $A$
admits the first three terms in a minimal projective resolution in $\mod A^e$
\[
  \bP_2 \xrightarrow{d_2}
  \bP_1 \xrightarrow{d_1}
  \bP_0 \xrightarrow{d_0}
  A \rightarrow 0
\]
with
\begin{align*}
 \bP_0 &= 
   P(1,1) \oplus P(2,2)
,
\\
 \bP_1 = \bP_2 &=
   P(1,1) \oplus P(1,2) \oplus P(2,1) \oplus P(2,2)
,
\end{align*}
the differential
$R = d_2$
given by
\begin{align*}
 R(e_1 \otimes e_1)
  &= \varrho(\alpha^2 - \sigma \gamma - \alpha^3)
  \\&
   = e_1 \otimes \alpha + \alpha \otimes e_1
     - e_1 \otimes \gamma - \sigma \otimes e_1
     - e_1 \otimes \alpha^2 - \alpha \otimes \alpha - \alpha^2 \otimes e_1
,\\
 R(e_1 \otimes e_2)
  &= \varrho(\alpha \sigma - \sigma \beta)
   = e_1 \otimes \sigma + \alpha \otimes e_2
     - e_1 \otimes \beta - \sigma \otimes e_2
,\\
 R(e_2 \otimes e_1)
  &= \varrho(\gamma \alpha - \beta \gamma)
   = e_2 \otimes \alpha + \gamma \otimes e_1
     - e_2 \otimes \gamma - \beta \otimes e_1
,\\
 R(e_2 \otimes e_2)
  &= \varrho(\lambda \beta^2 - \gamma \sigma)
   = \lambda e_2 \otimes \beta + \lambda \beta \otimes e_2
     - e_2 \otimes \sigma - \gamma \otimes e_2
,
\end{align*}
and the $A$-$A$-bimodule
$\Omega_{A^e}^3(A) = \Ker R$
generated by the following elements in $\bP_2$
\begin{align*}
 \Gamma_1 &=
  \sigma (e_2 \otimes e_1) - (e_1 \otimes e_2) \gamma
  - \alpha (e_1 \otimes e_1) + (e_1 \otimes e_1) \alpha
,
\\
 \Gamma_2 &=
  \gamma (e_1 \otimes e_2) - (e_2 \otimes e_1) \sigma
  + \beta (e_2 \otimes e_2) - (e_2 \otimes e_2) \beta
.
\end{align*}

We use also the basis
$\cB = e_1 \cB \cup e_2 \cB$
of the $K$-vector space $A$,
defined in the proof of Proposition~\ref{prop:5.1}.

We shall show now that the equalities (i), (ii), (iii) hold.

\smallskip

(i)
As in the proof of Proposition~\ref{prop:5.1}
we infer that the center $C(A)$ of $A$ has the basis
\[
  \{ 1, \alpha+\beta, \alpha^2, \alpha^3, \beta^2, \beta^3\}
\]
over $K$, and hence $\dim_K HH^0(A)=\dim_K C(A) = 6$.

\smallskip

(ii)
We proceed as in the proof of the statement (ii)
of Proposition~\ref{prop:5.1}.
We have again
$\dim_K \Hom_{A^e} (\bP_0, A) = 8$
and
$\dim_K \Hom_{A^e} (\bP_1, A) = 12$.
We determine now
$\dim_K \Hom_{A^e}(\Omega_{A^e}(A),A)$,
identifying $\Hom_{A^e}(\Omega_{A^e}(A),A)$
with
\[
  \big\{ \varphi \in  \Hom_{A^e} (\bP_1, A) \,\big|\, \varphi R=0\big\}
.
\]
Take a homomorphism $\varphi: \bP_1\to A$ in $\mod A^e$,
so $\varphi$ is given by
\begin{align*}
  \varphi(e_1\otimes e_2) &= a_1\sigma + a_2\sigma\beta,\\
  \varphi(e_2\otimes e_1) &= b_1\gamma + b_2\gamma\alpha,\\
  \varphi(e_1\otimes e_1) &= \sum_{i=0}^3 c_i\alpha^i,\\
  \varphi(e_2\otimes e_2) &= \sum_{i=0}^3 d_i\beta^i,
\end{align*}
with $\alpha^0 = e_1$ and $\beta^0 = e_2$,
for some $a_1, a_2, b_1, b_2, c_0, c_1, c_2, c_3, d_0, d_1, d_2, d_3$ from $K$.
Then, as in the proof of (ii) in Proposition~\ref{prop:5.1},
we have the equalities
\begin{align*}
  \varphi \big( R(e_1\otimes e_2) \big)
  &= (c_0-d_0)\sigma + (c_1-d_1)\sigma\beta
 ,\\
  \varphi \big( R(e_2\otimes e_1) \big)
  &= (c_0-d_0)\gamma + (c_1-d_1)\gamma\alpha
 ,\\
  \varphi \big( R(e_2\otimes e_2) \big)
  &= \lambda \Big(2 d_0 \beta + (2 d_1 - a_1 - b_1) \beta^2 + (2 d_2 - a_2 - b_2) \beta^3 \Big)
 \\
  &= \lambda \Big((a_1 + b_1) \beta^2 + (a_2 + b_2) \beta^3 \Big)
,
\end{align*}
since $\charact(K) = 2$.
But for $\varphi ( R(e_1\otimes e_1) )$
we have the equalities
\begin{align*}
  \varphi \big( R(e_1\otimes e_1) \big)
  &=
   \varphi (e_1\otimes e_1) \alpha
   + \alpha \varphi (e_1\otimes e_1)
   - \varphi (e_1\otimes e_2) \gamma
   - \sigma \varphi (e_2\otimes e_1)
  \\&\ \ \ \,
   - \varphi (e_1\otimes e_1) \alpha^2
   - \alpha \varphi (e_1\otimes e_1) \alpha
   - \alpha^2 \varphi (e_1\otimes e_1)
  \\
  &=
   2 \alpha \bigg(\sum_{i=0}^3 c_i\alpha^i\bigg)
   - a_1\sigma \gamma - a_2\sigma\beta \gamma
   - b_1\sigma\gamma - b_2\sigma\gamma\alpha
   - 3 \alpha^2 \bigg(\sum_{i=0}^3 c_i\alpha^i\bigg)
  \\
  &=
   (a_1 + b_1) \sigma \gamma + (a_2 + b_2) \sigma \gamma \alpha
   + c_0 \alpha^2 + c_1 \alpha^3
 \\
  &= (c_0 + a_1 + b_1) \alpha^2 + (c_1 +  a_1 + b_1 + a_2 + b_2) \alpha^3
,
\end{align*}
because $\charact (K) = 2$ and $\sigma \gamma = \alpha^2 + \alpha^3$.
Therefore, $\varphi$ induces a homomorphism
$\Omega_{A^e}(A)\to A$ in $\mod A^e$ if and only if
\begin{gather*}
  c_0=d_0,
  \ \,
  c_1=d_1,
  \ \,
  a_1+b_1=0,
  \ \,
  a_2+b_2=0,
  \\
  c_0 + a_1 + b_1 = 0,
  \ \,
  c_1 +  a_1 + b_1 + a_2 + b_2 = 0,
\end{gather*}
or equivalently,
\[
  c_0=d_0=0,
  \ \,
  c_1=d_1=0,
  \ \,
  a_1+b_1=0,
  \ \,
  a_2+b_2=0
  .
\]
Hence, we conclude that
$\dim_K \Hom_{A^e}(\Omega_{A^e}(A),A) = 6$.
Therefore, we obtain
\begin{align*}
  \dim_K HH^1(A)
   &= \dim_K HH^0(A) + \dim_K \Hom_{A^e}\big(\Omega_{A^e}(A),A\big)
 \\&\ \ \ \,
      - \dim_K \Hom_{A^e}(\bP_0,A)
   \\&
    = 6 + 6 - 8 = 4 .
\end{align*}

\smallskip

(iii)
We identify again $\Hom_{A^e}(\Omega_{A^e}^2(A),A)$
with the $K$-vector space
\[
  \big\{ \varphi \in  \Hom_{A^e} (\bP_2, A) \,\big|\, \varphi\big (\Omega_{A^e}^3(A)\big) = 0\big\}.
\]
Since $\bP_1 = \bP_2$, a homomorphism
$\varphi \in  \Hom_{A^e} (\bP_2, A)$
is defined as $\varphi \in  \Hom_{A^e} (\bP_1, A)$
in (ii).
We have the equalities
\begin{align*}
  \varphi(\Gamma_1)
  &=
   \sigma \varphi (e_2 \otimes e_1)
   - \varphi (e_1 \otimes e_2) \gamma
   - \alpha \varphi (e_1 \otimes e_1)
   - \varphi (e_1\otimes e_1) \alpha
 \\&
  =
   (b_1 - a_1) \sigma \gamma
   + (b_2 - a_2) \sigma \gamma \alpha
 \\&
  =
   (b_1 - a_1) (\alpha^2 + \alpha^3)
   + (b_2 - a_2) (\alpha^2 + \alpha^3) \alpha
 \\&
  =
   (b_1 - a_1) \alpha^2
   + (b_1 - a_1 + b_2 - a_2) \alpha^3
,
\end{align*}
because $\sigma \gamma = \alpha^2 + \alpha^3$
and $\alpha^4 = 0$.
Similarly, we have
\begin{align*}
  \varphi(\Gamma_2)
  &=
   \gamma \varphi (e_1\otimes e_2)
   - \varphi (e_2\otimes e_1) \sigma
   + \beta \varphi (e_2\otimes e_2)
   - \varphi (e_2\otimes e_2) \beta
 \\&
  =
   (a_1 - b_1) \gamma \sigma
   + (a_2 - b_2) \gamma \alpha  \sigma
 \\&
  =
   \lambda (a_1 - b_1) \beta^2
   + \lambda (a_2 - b_2) \beta^3
.
\end{align*}
Hence, we conclude that $\varphi(\Gamma_1) = 0$
and $\varphi (\Gamma_2) = 0$
if and only if $a_1 = b_1$ and $a_2 = b_2$.
This shows that
$\dim_K \Hom_{A^e}(\Omega_{A^e}^2(A),A) = 10$.
Therefore,
we obtain
\begin{align*}
  \dim_K HH^2(A)
   &= \dim_K \Hom_{A^e}\big(\Omega_{A^e}(A),A\big)
      + \dim_K \Hom_{A^e}\big(\Omega_{A^e}^2(A),A\big)
   \\&\ \ \ \,
      - \dim_K \Hom_{A^e}(\bP_1,A)
   \\&
    = 6 + 10 - 12 = 4 .
\end{align*}
We note that $\dim_K HH^1(A) = \dim_K HH^2(A)$.
\end{proof}

\begin{proposition}
\label{prop:5.3}
Let $K$ be an algebraically closed field of characteristic 2,
$\Lambda = \Lambda_3(\lambda)$ for some
$\lambda \in K \setminus \{ 0, 1 \}$,
and $A$ a standard self-injective algebra over $K$.
Then $A$ and $\Lambda$ are not derived equivalent.
\end{proposition}

\begin{proof}
Assume that $A$ and $\Lambda$ are derived equivalent.
Then it follows from Theorem~\ref{thm:2.8} that
$A$ and $\Lambda$ are stably equivalent,
and hence the stable Auslander-Reiten quivers $\Gamma_A^s$
and $\Gamma_{\Lambda}^s$ are isomorphic.
In particular, we conclude that $\Gamma_{\Lambda}^s$
consists of stable tubes of rank $1$ and $2$.
Further, by
Theorems \ref{thm:2.9} and \ref{thm:2.10},
$A$ is a representation-infinite periodic algebra
of polynomial growth, and hence is a standard
self-injective algebra of tubular type $(2,2,2,2)$.
Moreover, $A$ is a symmetric algebra by Theorems
\ref{thm:2.7} and \ref{thm:4.1}.
We also note that $K_0(A)$ is isomorphic to $K_0(\Lambda)$,
and so $K_0(A)$ is of rank $2$.
Applying now \cite[Theorem~1]{BiS2} we conclude
that $A$ is isomorphic to an algebra $\Lambda_3'(\mu)$
for some $\mu \in K \setminus \{0,1\}$.
On the other hand, it follows from
Propositions \ref{prop:5.1} and \ref{prop:5.2} that
$\dim_K HH^2(\Lambda) = 4 < 6 = \dim_K HH^2(A)$,
which contradicts Theorem \ref{thm:2.6}.
Therefore, $A$ and $\Lambda$ are not derived equivalent.
\end{proof}

\section{Hochschild cohomology for tubular type $(3,3,3)$}%
\label{sec:type333}

In this section we determine the dimensions
of the low Hochschild cohomology spaces
for the exceptional periodic algebras of tubular type $(3,3,3)$.

\begin{proposition}
\label{prop:6.1}
Let $A$ be one of the algebras
$\Lambda_1'$ or $\Lambda_2'$.
Then
\begin{enumerate}[(i)]
\item
  $\dim_K HH^0(A) = 5$.
\item
  $\dim_K HH^1(A) = \left\{ \begin{array}{cl} 3, & \charact (K \neq 3 \\ 4, & \charact (K) = 3 \end{array} \right.$.\vspace{2pt}
\item
  $\dim_K HH^2(A) = \left\{ \begin{array}{cl} 3, & \charact (K) \neq 3 \\ 4, & \charact (K) = 3 \end{array} \right.$.
\end{enumerate}
\end{proposition}

\begin{proof}
We may assume, by Theorems \ref{thm:2.6} and \ref{thm:4.1},
that $A = \Lambda_1'$.
It follows from the proof of \cite[Proposition~8.1]{BES2} that $A$
admits the first three terms in a minimal projective resolution in $\mod A^e$
\[
  \bP_2 \xrightarrow{d_2}
  \bP_1 \xrightarrow{d_1}
  \bP_0 \xrightarrow{d_0}
  A \rightarrow 0
\]
with
\begin{align*}
 \bP_0 = \bP_2 &=
   P(1,1) \oplus P(2,2)
,
\\
 \bP_1 &=
   P(1,2) \oplus P(2,1) \oplus P(2,2)
,
\end{align*}
the differential
$R = d_2$
given by
\begin{align*}
 R(e_1 \otimes e_1)
  &= \varrho(\beta \alpha \gamma)
   = e_1 \otimes \alpha \gamma + \beta \otimes \gamma
     + \beta \alpha \otimes e_1
,\\
 R(e_2 \otimes e_2)
  &= \varrho(\alpha^2 - \gamma \beta)
   = e_2 \otimes \alpha + \alpha \otimes e_2
     - e_2 \otimes \beta - \gamma \otimes e_2
,
\end{align*}
and the $A$-$A$-bimodule
$\Omega_{A^e}^3(A)$
generated by the following elements in $\bP_2$
\begin{align*}
 \Psi_1' &=
  - e_1 \otimes \beta \gamma
  + \beta \gamma \otimes e_1
  + \beta \otimes \alpha \gamma
  - \beta \alpha \otimes \gamma
,
\\
 \Psi_2' &=
   e_2 \otimes \alpha^4
  - \alpha \otimes \alpha^3
  + \alpha^3 \otimes \alpha
  - \alpha^4 \otimes e_2
  + \gamma \otimes \beta \alpha
  - \alpha \gamma \otimes \beta
.
\end{align*}

We fix also the following basis
$\cB = e_1 \cB \cup e_2 \cB$
of the $K$-vector space
$A$:
\begin{align*}
e_1 \cB &=
\{ e_1, \beta, \beta \alpha, \beta \gamma, \beta \alpha^2, \beta \alpha^2 \gamma \}
,
\\
e_2 \cB &=
\{ e_2, \alpha, \alpha^2, \alpha^3, \alpha^4, \gamma, \alpha \gamma, \alpha^2 \gamma \}
.
\end{align*}

We shall show now that the equalities (i), (ii), (iii) hold.

\smallskip

(i)
One checks that the center $C(A)$
of $A$ has  $K$-basis
\[
  \{ 1, \beta \gamma + \gamma \beta, \alpha^3, \alpha^4, \beta \alpha^2 \gamma\}
\]
and hence $\dim_K HH^0(A) = \dim_K C(A) = 5$.

\smallskip

(ii)
We will determine $\dim_K HH^1(A)$ using the exact sequence
of $K$-vector spaces (see Proposition~\ref{prop:2.5})
\[
 0
  \to HH^0(A)
  \to \Hom_{A^e}(\bP_0,A)
  \to \Hom_{A^e}\big(\Omega_{A^e}(A),A\big)
  \to HH^1(A)
  \to 0
  .
\]
We have isomorphisms of $K$-vector spaces
\begin{align*}
 \Hom_{A^e} (\bP_0, A) &=
  e_1 A e_1  \oplus e_2 A e_2
,
\\
 \Hom_{A^e} (\bP_1, A) &=
   e_1 A e_2 \oplus
   e_2 A e_1  \oplus e_2 A e_2 ,
\end{align*}
and hence we obtain
\begin{align*}
 \dim_K \Hom_{A^e} (\bP_0, A) &=
   3+5=8
,
\\
 \dim_K \Hom_{A^e} (\bP_1, A) &=
   3+3+5=11
.
\end{align*}
We identify $\Hom_{A^e}(\Omega_{A^e}(A),A)$
with the $K$-vector space
\[
  \big\{ \varphi \in  \Hom_{A^e} (\bP_1, A) \,|\, \varphi R=0\big\} .
\]

Let $\varphi: \bP_1\to A$ be a homomorphism in $\mod A^e$.
Then, using the basis $\cB = e_1 \cB \cup e_2 \cB$ of $A$,
we conclude that there exist
$b_0, b_1, b_2, c_0, c_1, c_2, a_0, a_1, a_2, a_3, a_4 \in K$
such that
\begin{align*}
  \varphi(e_1\otimes e_2) &= b_0 \beta + b_1 \beta \alpha + b_2 \beta \alpha^2,\\
  \varphi(e_2\otimes e_1) &= c_0\gamma + c_1\alpha\gamma + c_2\alpha^2\gamma,\\
  \varphi(e_2\otimes e_2) &= a_0 e_2 + a_1 \alpha + a_2 \alpha^2
                             + a_3 \alpha^3 + a_4 \alpha^4.
\end{align*}
We have the equalities
\begin{align*}
  \varphi \big( R(e_1\otimes e_1) \big)
  &=
   \varphi (e_1\otimes e_2) \alpha \gamma
   + \beta \varphi (e_2\otimes e_2) \gamma
   + \beta \alpha \varphi (e_2\otimes e_1)
 \\
  &=
   a_0 \beta \gamma
   + (a_2 + b_1 + c_1) \beta \alpha^2 \gamma ,
 \\
  \varphi \big( R(e_2\otimes e_2) \big)
  &=
   \varphi (e_2\otimes e_2) \alpha
   + \alpha \varphi (e_2\otimes e_2)
   - \varphi (e_2\otimes e_1) \beta
   - \gamma \varphi (e_1\otimes e_2)
 \\
  &=
   2 a_0 \alpha
   + (2 a_1 - b_0 - c_0) \alpha^2
   + (2 a_2 - b_1 - c_1) \alpha^3
   \\&\ \ \ \,
   + (2 a_3 - b_2 - c_2) \alpha^4
.
\end{align*}
Hence,
$\varphi ( R(e_1\otimes e_1) ) = 0$
and
$\varphi ( R(e_2\otimes e_2) ) = 0$
if and only if
\[
  a_0=0,
  \ \,
  a_2+b_1+c_1=0,
  \ \,
  b_0 + c_0 = 2 a_1,
  \ \,
  b_1 + c_1 = 2 a_2,
  \ \,
  b_2 + c_2 = 2 a_3
  .
\]

Assume $\charact(K) \neq 3$.
Then $a_2 + b_1 + c_1 = 0$ and $b_1 + c_1 = 2 a_2$
force $a_2 = 0$.
Thus $\dim_K \Hom_{A^e}(\Omega_{A^e}(A),A) = 6$.
Hence we get
\begin{align*}
  \dim_K HH^1(A)
   &= \dim_K HH^0(A) + \dim_K \Hom_{A^e}\big(\Omega_{A^e}(A),A\big)
   \\&\ \ \ \,
      - \dim_K \Hom_{A^e}(\bP_0,A)
   \\&
    = 5 + 6 - 8 = 3 .
\end{align*}

Assume $\charact(K) = 3$.
Then $a_2 + b_1 + c_1 = 0$ is equivalent to $b_1 + c_1 = 2 a_2$.
Thus $\dim_K \Hom_{A^e}(\Omega_{A^e}(A),A) = 7$.
Hence we get
\begin{align*}
  \dim_K HH^1(A)
   &= \dim_K HH^0(A) + \dim_K \Hom_{A^e}\big(\Omega_{A^e}(A),A\big)
   \\&\ \ \ \,
      - \dim_K \Hom_{A^e}(\bP_0,A)
   \\&
    = 5 + 7 - 8 = 4 .
\end{align*}

\smallskip

(iii)
We calculate $\dim_K HH^2(A)$ using the exact sequence
of $K$-vector spaces (see Proposition~\ref{prop:2.5})
\[
 0
  \to \Hom_{A^e}\big(\Omega_{A^e}(A),A\big)
  \to \Hom_{A^e}(\bP_1,A)
  \to \Hom_{A^e}\big(\Omega_{A^e}^2(A),A\big)
  \to HH^2(A)
  \to 0
  .
\]
We identify $\Hom_{A^e}(\Omega_{A^e}^2(A),A)$
with the $K$-vector space
\[
  \big\{ \varphi \in  \Hom_{A^e} (\bP_2, A) \,|\, \varphi \big(\Omega_{A^e}^3(A)\big) = 0\big\} .
\]
Recall also that
$\bP_2 = A e_1 \otimes e_1 A \oplus A e_2 \otimes e_2 A$.
Let $\varphi: \bP_2\to A$ be a homomorphism in $\mod A^e$.
Then there exist elements
$d_0, d_1, d_2, a_0, a_1, a_2, a_3, a_4 \in K$
such that
\begin{align*}
  \varphi(e_1\otimes e_1) &= d_0 e_1 + d_1 \beta \gamma + d_2 \beta \alpha^2 \gamma,\\
  \varphi(e_2\otimes e_2) &= a_0 e_2 + a_1 \alpha + a_2 \alpha^2
                             + a_3 \alpha^3 + a_4 \alpha^4.
\end{align*}
We have the equalities
\begin{align*}
  \varphi ( \Psi_1' )
  &=
   - \varphi (e_1\otimes e_1) \beta \gamma
   + \beta \gamma \varphi (e_1\otimes e_1)
   + \beta \varphi (e_2\otimes e_2) \alpha \gamma
   - \beta \alpha \varphi (e_2\otimes e_2) \gamma
 \\
  &=
   \sum_{i=0}^4 a_i \beta \alpha^i \alpha \gamma
   - \sum_{i=0}^4 a_i \beta \alpha \alpha^i \gamma
  = 0
,
 \\
  \varphi ( \Psi_2' )
  &=
   \varphi (e_2\otimes e_2) \alpha^4
   - \alpha \varphi (e_2\otimes e_2) \alpha^3
   + \alpha^3 \varphi (e_2\otimes e_2) \alpha
   - \alpha^4 \varphi (e_2\otimes e_2)
 \\&\ \ \ \,
   + \gamma \varphi (e_1\otimes e_1) \beta \gamma
   - \alpha \gamma \varphi (e_1\otimes e_1) \beta
 \\
  &=
   d_0 \gamma \beta \alpha
   - d_0 \alpha \gamma \beta
  = 0
.
\end{align*}
Hence, we get
$\dim_K \Hom_{A^e}(\Omega_{A^e}^2(A),A) = 8$.
Therefore, if $\charact(K) \neq 3$, we obtain
\begin{align*}
  \dim_K HH^2(A)
   &= \dim_K \Hom_{A^e}\big(\Omega_{A^e}(A),A\big)
      + \dim_K \Hom_{A^e}\big(\Omega_{A^e}^2(A),A\big)
   \\&\ \ \ \,
      - \dim_K \Hom_{A^e}(\bP_1,A)
   \\&
    = 6 + 8 - 11 = 3 .
\end{align*}
On the other hand, if $\charact(K) = 3$, we obtain
\begin{align*}
  \dim_K HH^2(A)
   &= \dim_K \Hom_{A^e}\big(\Omega_{A^e}(A),A\big)
      + \dim_K \Hom_{A^e}\big(\Omega_{A^e}^2(A),A\big)
   \\&\ \ \ \,
      - \dim_K \Hom_{A^e}(\bP_1,A)
   \\&
    = 7 + 8 - 11 = 4 .
\end{align*}
We note that we have always $\dim_K HH^1(A) = \dim_K HH^2(A)$.
\end{proof}

\begin{proposition}
\label{prop:6.2}
Let $K$ be an algebraically closed field of characteristic $3$
and
$A$ be one of the algebras
$\Lambda_1$ or $\Lambda_2$.
Then
\begin{enumerate}[(i)]
\item
  $\dim_K HH^0(A) = 5$.
\item
  $\dim_K HH^1(A) = 3$.
\item
  $\dim_K HH^2(A) = 3$.
\end{enumerate}
\end{proposition}

\begin{proof}
We may assume, by Theorems \ref{thm:2.6} and \ref{thm:4.1},
that $A = \Lambda_1$.
It follows from the proof of \cite[Proposition~8.4]{BES2} that $A$ 
admits the first three terms in a minimal projective resolution in $\mod A^e$
\[
  \bP_2 \xrightarrow{d_2}
  \bP_1 \xrightarrow{d_1}
  \bP_0 \xrightarrow{d_0}
  A \rightarrow 0
\]
with
\begin{align*}
 \bP_0 = \bP_2 &=
   P(1,1) \oplus P(2,2)
,
\\
 \bP_1 &=
   P(1,2) \oplus P(2,1) \oplus P(2,2)
,
\end{align*}
the differential
$R = d_2$
given by
\begin{align*}
 R(e_1 \otimes e_1)
  &= \varrho(\beta \alpha \gamma - \beta \alpha^2 \gamma)
  \\&
   = e_1 \otimes \alpha \gamma + \beta \otimes \gamma
     + \beta \alpha \otimes e_1
     - e_1 \otimes \alpha^2 \gamma
     - \beta \otimes \alpha \gamma
   \\&\ \ \ \,
     - \beta \alpha \otimes \gamma
     - \beta \alpha^2 \otimes e_1
,\\
 R(e_2 \otimes e_2)
  &= \varrho(\alpha^2 - \gamma \beta)
   = e_2 \otimes \alpha + \alpha \otimes e_2
     - e_2 \otimes \beta - \gamma \otimes e_2
,
\end{align*}
and the $A$-$A$-bimodule
$\Omega_{A^e}^3(A)$
generated by the following elements in $\bP_2$
\begin{align*}
 \Psi_1 &=
  - e_1 \otimes \beta \gamma
  + \beta \gamma \otimes e_1
  + \beta \otimes \alpha \gamma
  - \beta \alpha \otimes \gamma
  + \beta \alpha^2 \otimes \gamma
  - \beta \otimes \alpha^2 \gamma
,
\\
 \Psi_2 &=
  e_2 \otimes \alpha^4
  - \alpha \otimes \alpha^3
  + \alpha^3 \otimes \alpha
  - \alpha^4 \otimes e_2
  + \gamma \otimes \beta \alpha
  - \alpha \gamma \otimes \beta
  - \alpha^4 \otimes \alpha
  + \alpha \otimes \alpha^4
.
\end{align*}

We use the basis
$\cB = e_1 \cB \cup e_2 \cB$
of
$A$,
defined in the proof of Proposition~\ref{prop:6.1}.

We shall show now that the equalities (i), (ii), (iii) hold.

\smallskip

(i)
The center $C(A)$ of $A$ has the $K$-linear basis
\[
  \{ 1, \beta \gamma + \gamma \beta, \alpha^3, \alpha^4, \beta \alpha^2 \gamma\}
\]
and hence $\dim_K HH^0(A) = \dim_K C(A) = 5$.

\smallskip

(ii)
We proceed as in the proof
of Proposition~\ref{prop:6.1}.
We have again
$\dim_K \Hom_{A^e} (\bP_0, A) = 8$
and
$\dim_K \Hom_{A^e} (\bP_1, A) = 11$.
We calculate
$\dim_K \Hom_{A^e}(\Omega_{A^e}(A),A)$,
identifying $\Hom_{A^e}(\Omega_{A^e}(A),A)$
with
the $K$-vector space
$\{ \varphi \in  \Hom_{A^e} (\bP_1, A) \,|\, \varphi R=0\}$.
Let $\varphi: \bP_1\to A$
be a homomorphism in $\mod A^e$.
Then there exist elements
$b_0, b_1, b_2, c_0, c_1, c_2, a_0, a_1, a_2, a_3, a_4 \in K$
such that
\begin{align*}
  \varphi(e_1\otimes e_2) &= b_0 \beta + b_1 \beta \alpha + b_2 \beta \alpha^2,\\
  \varphi(e_2\otimes e_1) &= c_0\gamma + c_1\alpha\gamma + c_2\alpha^2\gamma,\\
  \varphi(e_2\otimes e_2) &= a_0 e_2 + a_1 \alpha + a_2 \alpha^2
                             + a_3 \alpha^3 + a_4 \alpha^4.
\end{align*}
We have the equalities
\begin{align*}
  \varphi \big( R(e_1\otimes e_1) \big)
  &=
   \varphi (e_1\otimes e_2) \alpha \gamma
   + \beta \varphi (e_2\otimes e_2) \gamma
   + \beta \alpha \varphi (e_2\otimes e_1)
   - \varphi (e_1\otimes e_2) \alpha^2 \gamma
  \\&\ \ \ \,
   - \beta \varphi (e_2\otimes e_2) \alpha \gamma
   - \beta \alpha \varphi (e_2\otimes e_2) \gamma
   - \beta \alpha^2 \varphi (e_2\otimes e_1)
 \\
  &=
   a_0 \beta \gamma
   + (b_0 + b_1 + c_1 + a_0 - a_1 + a_2) \beta \alpha^2 \gamma ,
\end{align*}
because $\beta \alpha \gamma = \beta \alpha^2 \gamma$
and $\charact(K) = 3$.
Similarly, we have
\begin{align*}
  \varphi \big( R(e_2\otimes e_2) \big)
  &=
   \varphi (e_2\otimes e_2) \alpha
   + \alpha \varphi (e_2\otimes e_2)
   - \varphi (e_2\otimes e_1) \beta
   - \gamma \varphi (e_1\otimes e_2)
 \\
  &=
   2 a_0 \alpha
   + (2 a_1 - c_0 - b_0) \alpha^2
   + (2 a_2 - c_1 - b_1) \alpha^3
   + (2 a_3 - c_2 - b_2) \alpha^4
 \\
  &=
   - a_0 \alpha
   - (a_1 + c_0 + b_0) \alpha^2
   - (a_2 + c_1 + b_1) \alpha^3
   - (a_3 + c_2 + b_2) \alpha^4
,
\end{align*}
because $\charact(K) = 3$.
Hence, we conclude that
$\varphi ( R(e_1\otimes e_1) ) = 0$
and
$\varphi ( R(e_2\otimes e_2) ) = 0$
if and only if
\begin{gather*}
  a_0=0,
 \ \,
  b_0 + b_1 + c_1 + a_0 - a_1 + a_2=0,
\\
  a_1 + c_0 + b_0 = 0,
 \ \,
  a_2 + c_1 + b_1 = 0,
 \ \,
  a_3 + c_2 + b_2 = 0.
\end{gather*}
This implies that
$\dim_K \Hom_{A^e}(\Omega_{A^e}(A),A) = 6$.
Therefore, we obtain
\begin{align*}
  \dim_K HH^1(A)
   &= \dim_K HH^0(A) + \dim_K \Hom_{A^e}\big(\Omega_{A^e}(A),A\big)
  \\&\ \ \ \,
      - \dim_K \Hom_{A^e}(\bP_0,A)
   \\&
    = 5 + 6 - 8 = 3 .
\end{align*}

\smallskip

(iii)
We identify $\Hom_{A^e}(\Omega_{A^e}^2(A),A)$
with the $K$-vector space
\[
  \big\{ \varphi \in  \Hom_{A^e} (\bP_2, A) \,|\, \varphi \big(\Omega_{A^e}^3(A)\big) = 0\big\}.
\]
Recall also that
$\bP_2 = A e_1 \otimes e_1 A \oplus A e_2 \otimes e_2 A$.
Take a homomorphism $\varphi: \bP_2\to A$ in $\mod A^e$.
Then there exist elements
$d_0, d_1, d_2, a_0, a_1, a_2, a_3, a_4 \in K$
such that
\begin{align*}
  \varphi(e_1\otimes e_1) &= d_0 e_1 + d_1 \beta \gamma + d_2 \beta \alpha^2 \gamma,\\
  \varphi(e_2\otimes e_2) &= a_0 e_2 + a_1 \alpha + a_2 \alpha^2
                             + a_3 \alpha^3 + a_4 \alpha^4.
\end{align*}
Then again we have the equalities
\begin{align*}
  \varphi ( \Psi_1 )
  &=
   - \varphi (e_1\otimes e_1) \beta \gamma
   + \beta \gamma \varphi (e_1\otimes e_1)
   + \beta \varphi (e_2\otimes e_2) \alpha \gamma
   - \beta \alpha \varphi (e_2\otimes e_2) \gamma
 \\&\ \ \ \,
   + \beta \alpha^2 \varphi (e_2\otimes e_2) \gamma
   - \beta \varphi (e_2\otimes e_2) \alpha^2 \gamma
 \\
  &=
   \sum_{i=0}^4 a_i \beta \alpha^2 \alpha^i \gamma
   - \sum_{i=0}^4 a_i \beta \alpha^i \alpha^2 \gamma
  = 0
,
 \\
  \varphi ( \Psi_2 )
  &=
   \varphi (e_2\otimes e_2) \alpha^4
   - \alpha \varphi (e_2\otimes e_2) \alpha^3
   + \alpha^3 \varphi (e_2\otimes e_2) \alpha
   - \alpha^4 \varphi (e_2\otimes e_2)
 \\&\ \ \ \,
   + \gamma \varphi (e_1\otimes e_1) \beta \gamma
   - \alpha \gamma \varphi (e_1\otimes e_1) \beta
   - \alpha^4 \varphi (e_2\otimes e_2) \alpha
   + \alpha \varphi (e_2\otimes e_2) \alpha^4
 \\
  &=
   0
.
\end{align*}
Hence, we obtain
$\dim_K \Hom_{A^e}(\Omega_{A^e}^2(A),A) = 8$.
Therefore, we get
\begin{align*}
  \dim_K HH^2(A)
   &= \dim_K \Hom_{A^e}\big(\Omega_{A^e}(A),A\big)
      + \dim_K \Hom_{A^e}\big(\Omega_{A^e}^2(A),A\big)
   \\&\ \ \ \,
      - \dim_K \Hom_{A^e}(\bP_1,A)
   \\&
    = 6 + 8 - 11 = 3 .
\end{align*}
We note that $\dim_K HH^1(A) = \dim_K HH^2(A)$.
\end{proof}

\begin{proposition}
\label{prop:6.3}
Let $K$ be an algebraically closed field of characteristic $3$,
$\Lambda$ one of the non-standard algebras
$\Lambda_1$ or $\Lambda_2$ over $K$,
and $A$ a standard self-injective algebra over $K$.
Then $A$ and $\Lambda$ are not derived equivalent.
\end{proposition}

\begin{proof}
We know from Theorem~\ref{thm:4.3}
that $\Lambda_1$ and $\Lambda_2$ are
derived equivalent.
Assume that $A$ and $\Lambda$ are derived equivalent.
Then it follows from Theorem~\ref{thm:2.8} that
$A$ and $\Lambda$ are stably equivalent,
and hence the stable Auslander-Reiten quivers $\Gamma_A^s$
and $\Gamma_{\Lambda}^s$ are isomorphic.
In particular, we conclude that
$\Gamma_{A}^s$
consists of stable tubes of rank $1$ and $3$.
Further, by
Theorems \ref{thm:2.9} and \ref{thm:2.10},
$A$ is a representation-infinite periodic algebra
of polynomial growth, and hence a standard
self-injective algebra of tubular type $(3,3,3)$.
Moreover, $A$ is a symmetric algebra by
Theorem~\ref{thm:2.7} and Proposition~\ref{prop:4.5}.
We also note that $K_0(A)$ is isomorphic to $K_0(\Lambda)$,
and so $K_0(A)$ is of rank $2$.
Applying now \cite[Theorem~1]{BiS2} we conclude
that $A$ is isomorphic to one of the algebras
$\Lambda_1'$ or $\Lambda_2'$.
On the other hand, it follows from
Theorem~\ref{thm:2.6} and
Propositions \ref{prop:6.1} and \ref{prop:6.2} that
\[
  \dim_K HH^2(\Lambda)
  = \dim_K HH^2(\Lambda_1)
  = 3
  < 4
  = \dim_K HH^2(\Lambda_1')
  = \dim_K HH^2(A)
.
\]
Therefore, applying Theorem \ref{thm:2.6} again,
we conclude that
$A$ and $\Lambda$ are not derived equivalent.
\end{proof}

We will describe now the dimensions of the low Hochschild cohomology spaces
of $\Lambda_9$.

\begin{proposition}
\label{prop:6.4}
Let $K$ be an algebraically closed field of characteristic $2$
and $A = \Lambda_9$.
Then
\begin{enumerate}[(i)]
\item
  $\dim_K HH^0(A) = 5$.
\item
  $\dim_K HH^1(A) = 1$.
\item
  $\dim_K HH^2(A) = 2$.
\end{enumerate}
\end{proposition}

\begin{proof}
It follows from the proof of \cite[Proposition~8.7]{BES2}
that $A$ admits
the first three terms of a minimal projective resolution in $\mod A^e$
\[
  \bP_2 \xrightarrow{d_2}
  \bP_1 \xrightarrow{d_1}
  \bP_0 \xrightarrow{d_0}
  A \rightarrow 0
\]
with
\begin{align*}
 \bP_0 = \bP_2 &=
   \bigoplus_{i=1}^4 P(i,i)
,
\\
 \bP_1 &=
   \bigoplus_{i=1}^3
    \big(P(i,4) \oplus P(4,i)\big)
,
\end{align*}
the differential
$R = d_2$
given by
\begin{align*}
 R(e_1 \otimes e_1)
  &= \varrho(\alpha \beta + \alpha \delta \gamma \beta)
   = e_1 \otimes \beta + \alpha \otimes e_1
     + e_1 \otimes \delta \gamma \beta
     + \alpha \otimes \gamma \beta
     + \alpha \delta \otimes \beta
   \\&\ \ \ \,
     + \alpha \delta \gamma \otimes e_1
,\\
 R(e_2 \otimes e_2)
  &= \varrho(\xi \varepsilon)
   = e_2 \otimes \varepsilon + \xi \otimes e_2
,\\
 R(e_3 \otimes e_3)
  &= \varrho(\gamma \delta)
   = e_3 \otimes \delta + \gamma \otimes e_3
,\\
 R(e_4 \otimes e_4)
  &= \varrho(\beta \alpha + \varepsilon \xi + \delta \gamma)
   =
    e_4 \otimes \alpha + \beta \otimes e_4
    + e_4 \otimes \xi + \varepsilon \otimes e_4
    + e_4 \otimes \gamma 
   \\&\ \ \ \,
    + \delta \otimes e_4
,
\end{align*}
and the $A$-$A$-bimodule
$\Omega_{A^e}^3(A) = \Ker R$
generated by the following elements in $\bP_2$
\begin{align*}
 \Psi_1 &=
   e_1 \otimes \alpha\delta\gamma\beta
  + \alpha \otimes \delta\gamma\beta
  + \alpha\delta \otimes \gamma\beta
  + \alpha\varepsilon \otimes \xi\beta
  + \alpha\delta\gamma \otimes \beta
  + \alpha\delta\gamma \otimes \delta\gamma\beta
   \\&\ \ \ \,
  + \alpha\delta\gamma\beta \otimes e_1
   + \alpha \delta \gamma \otimes \delta \gamma \beta
,
\\
 \Psi_2 &=
  e_2 \otimes \xi\beta\alpha\varepsilon
  + \xi \otimes \beta\alpha\varepsilon
  + \xi\beta \otimes \alpha\varepsilon
  + \xi\delta \otimes \gamma\varepsilon
  + \xi\beta\alpha \otimes \varepsilon
  + \xi\beta\alpha\varepsilon \otimes e_2
,
\\
 \Psi_3 &=
  e_3 \otimes \gamma\beta\alpha\delta
  + \gamma \otimes \beta\alpha\delta
  + \gamma\beta \otimes \alpha\delta
  + \gamma\varepsilon \otimes \xi\delta
  + \gamma\varepsilon\xi \otimes \delta
  + \gamma\beta\alpha\delta \otimes e_3
   \\&\ \ \ \,
   + \gamma \beta \alpha \otimes \beta \alpha \delta
,
\\
 \Psi_4 &=
  e_4 \otimes \beta\alpha\delta\gamma
  + \beta \otimes \alpha\delta\gamma
  + \varepsilon \otimes \xi\beta\alpha
  + \delta \otimes \gamma\beta\alpha
  + \beta\alpha \otimes \delta\gamma
  + \delta\gamma \otimes \beta\alpha
   \\&\ \ \ \,
  + \beta\alpha\delta \otimes \gamma
  + \delta\gamma\varepsilon \otimes \xi
  + \delta\gamma\beta \otimes \alpha
  + \beta\alpha\delta\gamma \otimes e_4
   + \delta \gamma \beta \otimes \alpha \delta \gamma
   + \beta \alpha \delta \otimes \gamma \beta \alpha
.
\end{align*}

We fix the following basis
$\cB = e_1 \cB \cup e_2 \cB \cup e_3 \cB \cup e_4 \cB$
of the $K$-vector space $A$:
\begin{align*}
 e_1 \cB
  &=
   \{ e_1, \alpha, \alpha \delta, \alpha \varepsilon,
      \alpha \delta \gamma, \alpha \delta \gamma \beta \}
,\\
 e_2 \cB
  &=
   \{ e_2, \xi, \xi \beta, \xi \delta,
      \xi \beta \alpha, \xi \beta \alpha \varepsilon \}
,\\
 e_3 \cB
  &=
   \{ e_3, \gamma, \gamma \beta, \gamma \varepsilon,
      \gamma \beta \alpha, \gamma \beta \alpha \delta \}
,\\
 e_4 \cB
  &=
   \{ e_4, \beta, \varepsilon, \delta, \beta \alpha,
      \delta \gamma, \beta \alpha \delta,
      \delta \gamma \varepsilon, \delta \gamma \beta,
      \beta \alpha \delta \gamma \}
.
\end{align*}

We shall show now that the equalities (i), (ii), (iii) hold.

\smallskip

(i)
The center $C(A)$ of $A$ has the $K$-linear basis
\[
  \big\{
   1,
   \alpha \delta \gamma \beta,
   \xi \beta \alpha \varepsilon,
   \gamma \beta \alpha \delta,
   \beta \alpha \delta \gamma
  \big\},
\]
and hence $\dim_K HH^0(A) = \dim_K C(A) = 5$.

\smallskip

For calculations of
$\dim_K HH^1(A)$
and
$\dim_K HH^2(A)$
we use the first terms of the minimal projective resolution
of $A$ in $\mod A^e$ described above.
In particular,
we have $K$-linear isomorphisms
\[
 \Hom_{A^e} (\bP_0, A)
   \cong \bigoplus_{i=1}^4 e_i A e_i
 ,
 \ \ \,
 \Hom_{A^e} (\bP_1, A)
   \cong \bigoplus_{i=1}^3 (e_i A e_4 \oplus e_4 A e_i)
\]
and hence
\begin{align*}
 \dim_K \Hom_{A^e} (\bP_0, A) &=
   2+2+2+4=10
,
\\
 \dim_K \Hom_{A^e} (\bP_1, A) &=
   2+2+2+2+2+2=12
.
\end{align*}

\smallskip
(ii)
We calculate $\dim_K HH^1(A)$ using the exact sequence
of $K$-vector spaces
\[
 0
  \to HH^0(A)
  \to \Hom_{A^e}(\bP_0,A)
  \to \Hom_{A^e}\big(\Omega_{A^e}(A),A\big)
  \to HH^1(A)
  \to 0
  .
\]
We identify $\Hom_{A^e}(\Omega_{A^e}(A),A)$
with
$\{ \varphi \in  \Hom_{A^e} (\bP_1, A) \,|\, \varphi R=0\}$.
Let $\varphi: \bP_1\to A$
be a homomorphism in $\mod A^e$.
Then
\begin{align*}
  \varphi(e_1\otimes e_4) &= a_1 \alpha + a_2 \alpha \delta \gamma,
&
  \varphi(e_4\otimes e_1) &= b_1 \beta + b_2 \delta \gamma \beta,
  \\
  \varphi(e_2\otimes e_4) &= c_1 \xi + c_2 \xi \beta \alpha,
&
  \varphi(e_4\otimes e_2) &= d_1 \varepsilon + d_2 \delta \gamma \varepsilon,
  \\
  \varphi(e_3\otimes e_4) &= g_1 \gamma + g_2 \gamma \varepsilon \xi,
&
  \varphi(e_4\otimes e_3) &= h_1 \delta + h_2 \beta \alpha \delta,
\end{align*}
for some elements
$a_1, a_2, b_1, b_2, c_1, c_2, d_1, d_2, g_1, g_2, h_1, h_2 \in K$.
We have the equalities
\begin{align*}
  \varphi \big( R(e_1\otimes e_1) \big)
  &=
   \varphi(e_1 \otimes e_4) \beta
   + \alpha\varphi(e_4 \otimes e_1)
   - \varphi(e_1 \otimes e_4) \delta \gamma \beta
  \\&\ \ \ \,
   - \alpha \varphi(e_4 \otimes e_3) \gamma \beta
   - \alpha \delta \varphi(e_3 \otimes e_4) \beta
   - \alpha \delta \gamma \varphi(e_4 \otimes e_1)
  \\&=
   (a_1 + b_1) \alpha \beta
   + (a_2 + b_2 - a_1 - b_1 - g_1 - h_1) \alpha \delta \gamma \beta
  \\&=
   (a_2 + b_2 - g_1 - h_1) \alpha \delta \gamma \beta
  ,
  \\
  \varphi \big( R(e_2\otimes e_2) \big)
  &=
   \varphi(e_2 \otimes e_4) \varepsilon
   + \xi \varphi(e_4 \otimes e_2)
  \\&=
   (c_2 + d_2) \xi \beta \alpha \varepsilon
  ,
  \\
  \varphi \big( R(e_3\otimes e_3) \big)
  &=
   \varphi(e_3 \otimes e_4) \delta
   + \gamma \varphi(e_4 \otimes e_3)
  \\&=
   (g_2 + h_2) \gamma \beta \alpha \delta
  ,
  \\
  \varphi \big( R(e_4\otimes e_4) \big)
  &=
   \varphi(e_4 \otimes e_1) \alpha
   + \beta \varphi(e_1 \otimes e_4)
   + \varphi(e_4 \otimes e_3) \gamma
  \\&\ \ \ \,
   + \delta \varphi(e_3 \otimes e_4)
   + \varphi(e_4 \otimes e_2) \xi
   + \varepsilon \varphi(e_2 \otimes e_4)
  \\&=
   (a_1 + b_1) \beta \alpha
   + (g_1 + h_1) \delta \gamma
   + (c_1 + d_1) \varepsilon \xi
  \\&\ \ \ \,
   + (a_2 + b_2 + c_2 + d_2 + g_2 + h_2) \beta \alpha \delta \gamma
  \\&=
   (a_1 + b_1 + c_1 + d_1 ) \alpha \beta
   + (g_1 + h_1 + c_1 + d_1) \delta \gamma
  \\&\ \ \ \,
   + (a_2 + b_2 + c_2 + d_2 + g_2 + h_2) \beta \alpha \delta \gamma
.
\end{align*}
Hence, $\varphi R = 0$
if and only if
\begin{gather*}
  a_2 + b_2 = g_1 + h_1,
  \ \,
  c_2 + d_2 = 0,
  \ \,
  g_2 + h_2 = 0,
  \\
  a_1 + b_1 + c_1 + d_1 = 0,
  \ \,
  g_1 + h_1 + c_1 + d_1 = 0,
  \ \,
  a_2 + b_2 + c_2 + d_2 + g_2 + h_2 = 0.
\end{gather*}
Therefore, we conclude that $\dim_K \Hom_{A^e}(\Omega_{A^e}(A),A) = 12 - 6 = 6$.
Then we get
\begin{align*}
  \dim_K HH^1(A)
   &= \dim_K HH^0(A) + \dim_K \Hom_{A^e}\big(\Omega_{A^e}(A),A\big)
  \\&\ \ \ \,
      - \dim_K \Hom_{A^e}(\bP_0,A)
   \\&
    = 5 + 6 - 10 = 1 .
\end{align*}

\smallskip

(iii)
We identify $\Hom_{A^e}(\Omega_{A^e}^2(A),A)$
with
$\{ \varphi \in  \Hom_{A^e} (\bP_2, A) \,|\, \varphi (\Omega_{A^e}^3(A)) = 0\}$.
Let $\varphi: \bP_2 \to A$ be a homomorphism in $\mod A^e$.
Then there exist elements
$a_1, b_1, a_2, b_2, a_3, b_3, a_4, b_4, c, d \in K$
such that
\begin{align*}
  \varphi(e_1\otimes e_1) &= a_1 e_1 + b_1 \omega_1, \\
  \varphi(e_2\otimes e_2) &= a_2 e_2 + b_2 \omega_2, \\
  \varphi(e_3\otimes e_3) &= a_3 e_3 + b_3 \omega_3, \\
  \varphi(e_4\otimes e_4) &= a_4 e_4 + b_4 \omega_4 + c \beta \alpha + d \delta \gamma,
\end{align*}
where
$\omega_1 = \alpha \delta \gamma \beta$,
$\omega_2 = \xi \beta \alpha \varepsilon$,
$\omega_1 = \gamma \beta \alpha \delta$,
$\omega_1 = \beta \alpha \delta \gamma$.
Since the $A$-$A$-bimodule $\Omega_{A^e}^3(A)$
is generated by the elements
$\Psi_1, \Psi_2, \Psi_3, \Psi_4$,
we have $\varphi(\Omega_{A^e}^3(A)) = 0$
if and only if $\varphi(\Psi_i) = 0$
for any $i\in\{1,2,3,4\}$.
Using $\charact(K) = 2$,
we obtain the equalities
\begin{align*}
  \varphi(\Psi_1) &= (a_2 + a_3) \omega_1, \\
  \varphi(\Psi_2) &= (a_1 + a_3) \omega_2, \\
  \varphi(\Psi_3) &= (a_1 + a_2) \omega_3, \\
  \varphi(\Psi_4) &= 0.
\end{align*}
Hence, $\varphi(\Psi_i) = 0$
for all $i\in\{1,2,3,4\}$
if and only if
$a_2 + a_3 = 0$,
$a_1 + a_3 = 0$,
$a_1 + a_2 = 0$,
or equivalently
$a_1 = a_2 = a_3$.
It follows that
$\dim_K \Hom_{A^e}(\Omega_{A^e}^2(A),A) = 10 - 2 = 8$.
Therefore, we get
\begin{align*}
  \dim_K HH^2(A)
   &= \dim_K \Hom_{A^e}\big(\Omega_{A^e}(A),A\big)
      + \dim_K \Hom_{A^e}\big(\Omega_{A^e}^2(A),A\big)
   \\&\ \ \ \,
      - \dim_K \Hom_{A^e}(\bP_1,A)
   \\&
    = 6 + 8 - 12 = 2 .
\end{align*}
\end{proof}

We note that the part (iii) of Proposition~\ref{prop:6.4}
corrects \cite[Theorem~3.1]{AK2}.

\begin{proposition}
\label{prop:6.5}
Let $A = \Lambda_9'$.
Then
\begin{enumerate}[(i)]
\item
  $\dim_K HH^0(A) = 5$.
\item
  $\dim_K HH^1(A) = \left\{ \begin{array}{cl} 1, & \charact (K) \neq 2 \\ 2, & \charact (K) = 2 \end{array} \right.$.\vspace{2pt}
\item
  $\dim_K HH^2(A) = \left\{ \begin{array}{cl} 0, & \charact (K) \neq 2 \\ 3, & \charact (K) = 2 \end{array} \right.$.
\end{enumerate}
\end{proposition}

\begin{proof}
(i)
The center $C(A)$ of $A$ has the $K$-linear basis
\[
  \big\{
   1,
   \alpha \delta \gamma \beta,
   \xi \beta \alpha \varepsilon,
   \gamma \beta \alpha \delta,
   \beta \alpha \delta \gamma
  \big\},
\]
and hence $\dim_K HH^0(A) = \dim_K C(A) = 5$.

\smallskip

The first
three terms in a minimal projective resolution
\[
  \bP_2 \xrightarrow{R}
  \bP_1 \xrightarrow{d}
  \bP_0 \xrightarrow{d_0}
  A \rightarrow 0
\]
of $A$ in $\mod A^e$ are as
for the algebra $\Lambda_9$,
so we have
\begin{align*}
 \bP_0 &= \bP_2 =
   \bigoplus_{i=1}^4 P(i,i)
,
\\
 \bP_1 &=
   \bigoplus_{i=1}^3
    \big(P(i,4) \oplus P(4,i)\big)
.
\end{align*}
Hence, we conclude that there are $K$-linear isomorphisms
\[
 \Hom_{A^e} (\bP_0, A)
   \cong \bigoplus_{i=1}^4 e_i A e_i
 ,
 \ \ \,
 \Hom_{A^e} (\bP_1, A)
   \cong \bigoplus_{i=1}^3 (e_i A e_4 \oplus e_4 A e_i)
 ,
\]
and then we get
\[
 \dim_K \Hom_{A^e} (\bP_0, A) =10
 \quad
 \mbox{ and }
 \quad
 \dim_K \Hom_{A^e} (\bP_1, A) =12
.
\]
Moreover, according to the recipe described in
Section~\ref{sec:pre-results},
the differential
$R = d_2 : \bP_2 \to \bP_1$
is given by
\begin{align*}
 R(e_1 \otimes e_1)
  &= \varrho(\alpha \beta)
   = e_1 \otimes \beta + \alpha \otimes e_1
,\\
 R(e_2 \otimes e_2)
  &= \varrho(\xi \varepsilon)
   = e_2 \otimes \varepsilon + \xi \otimes e_2
,\\
 R(e_3 \otimes e_3)
  &= \varrho(\gamma \delta)
   = e_3 \otimes \delta + \gamma \otimes e_3
,\\
 R(e_4 \otimes e_4)
  &= \varrho(\beta \alpha + \varepsilon \xi + \delta \gamma)
   =
    e_4 \otimes \alpha + \beta \otimes e_4
    + e_4 \otimes \xi + \varepsilon \otimes e_4
   \\&\ \ \ \,
    + e_4 \otimes \gamma + \delta \otimes e_4
.
\end{align*}
We also fix the following basis
$\cB = e_1 \cB \cup e_2 \cB \cup e_3 \cB \cup e_4 \cB$
of the $K$-vector space $A$:
\begin{align*}
 e_1 \cB
  &=
   \{ e_1, \alpha, \alpha \delta, \alpha \varepsilon,
      \alpha \delta \gamma, \alpha \delta \gamma \beta \}
,\\
 e_2 \cB
  &=
   \{ e_2, \xi, \xi \beta, \xi \delta,
      \xi \beta \alpha, \xi \beta \alpha \varepsilon \}
,\\
 e_3 \cB
  &=
   \{ e_3, \gamma, \gamma \beta, \gamma \varepsilon,
      \gamma \beta \alpha, \gamma \beta \alpha \delta \}
,\\
 e_4 \cB
  &=
   \{ e_4, \beta, \varepsilon, \delta, \beta \alpha,
      \delta \gamma, \beta \alpha \delta,
      \delta \gamma \varepsilon, \delta \gamma \beta,
      \beta \alpha \delta \gamma \}
.
\end{align*}

\smallskip

(ii)
We calculate $\dim_K HH^1(A)$ using the exact sequence
of $K$-vector spaces
\[
 0
  \to HH^0(A)
  \to \Hom_{A^e}(\bP_0,A)
  \to \Hom_{A^e}\big(\Omega_{A^e}(A),A\big)
  \to HH^1(A)
  \to 0
  .
\]
In order to compute $\dim_K \Hom_{A^e}(\Omega_{A^e}(A),A)$
we identify $\Hom_{A^e}(\Omega_{A^e}(A),A)$
with the $K$-vector space
$\{ \varphi \in  \Hom_{A^e} (\bP_1, A) \,|\, \varphi R=0\}$.
Let $\varphi: \bP_1\to A$
be a homomorphism in $\mod A^e$.
Then
\begin{align*}
  \varphi(e_1\otimes e_4) &= a_1 \alpha + a_2 \alpha \delta \gamma,
&
  \varphi(e_4\otimes e_1) &= b_1 \beta + b_2 \delta \gamma \beta,
  \\
  \varphi(e_2\otimes e_4) &= c_1 \xi + c_2 \xi \beta \alpha,
&
  \varphi(e_4\otimes e_2) &= d_1 \varepsilon + d_2 \delta \gamma \varepsilon,
  \\
  \varphi(e_3\otimes e_4) &= g_1 \gamma + g_2 \gamma \varepsilon \xi,
&
  \varphi(e_4\otimes e_3) &= h_1 \delta + h_2 \beta \alpha \delta,
\end{align*}
for some elements
$a_1, a_2, b_1, b_2, c_1, c_2, d_1, d_2, g_1, g_2, h_1, h_2 \in K$.
We have the equalities
\begin{align*}
  \varphi \big( R(e_1\otimes e_1) \big)
  &=
   \varphi(e_1 \otimes e_4) \beta
   + \alpha\varphi(e_4 \otimes e_1)
   =
   (a_2 + b_2) \alpha \delta \gamma \beta
  ,
  \\
  \varphi \big( R(e_2\otimes e_2) \big)
  &=
   \varphi(e_2 \otimes e_4) \varepsilon
   + \xi \varphi(e_4 \otimes e_2)
   =
   (c_2 - d_2) \xi \beta \alpha \varepsilon
  ,
  \\
  \varphi \big( R(e_3\otimes e_3) \big)
  &=
   \varphi(e_3 \otimes e_4) \delta
   + \gamma \varphi(e_4 \otimes e_3)
   =
   (-g_2 + h_2) \gamma \beta \alpha \delta
  ,
  \\
  \varphi \big( R(e_4\otimes e_4) \big)
  &=
   \varphi(e_4 \otimes e_1) \alpha
   + \beta \varphi(e_1 \otimes e_4)
   + \varphi(e_4 \otimes e_2) \xi
   + \varepsilon \varphi(e_2 \otimes e_4)
  \\&\ \ \ \,
   + \varphi(e_4 \otimes e_3) \gamma
   + \delta \varphi(e_3 \otimes e_4)
  \\&=
   (a_1 + b_1 - c_1 - d_1) \alpha \beta
   + (- c_1 - d_1 + g_1 + h_1) \delta \gamma
  \\&\ \ \ \,
   + (a_2 - b_2 + c_2 + d_2 + g_2 + h_2) \beta \alpha \delta \gamma
.
\end{align*}
Hence, $\varphi$ factors through $\Omega_{A^e}(A)$
if and only if
\begin{gather*}
  a_2 + b_2 = 0,
  \ \,
  c_2 - d_2 = 0,
  \ \,
  - g_2 + h_2 = 0,
  \\
  a_1 + b_1 = c_1 + d_1,
  \ \,
  c_1 + d_1 = g_1 + h_1,
  \ \,
  a_2 - b_2 + c_2 + d_2 + g_2 + h_2 = 0.
\end{gather*}

Assume first that $\charact(K) \neq 2$.
Then the equations are equivalent to
\begin{gather*}
  a_2 = - b_2,
  \ \,
  c_2 = d_2,
  \ \,
  g_2 = h_2,
  \\
  a_1 + b_1 = c_1 + d_1 = g_1 + h_1,
  \ \,
  a_2 + c_2 + h_2 = 0.
\end{gather*}
We may choose $a_2$ and $c_2$, and then
$h_2$ and $g_2$ are fixed.
Then we may choose $a_1, c_1, d_1, h_1$
arbitrarily, and get a unique homomorphism.
Hence we obtain \linebreak
$\dim_K \Hom_{A^e}(\Omega_{A^e}(A),A) = 6$.
Therefore, we get
\[
  \dim_K HH^1(A) = 5 + 6 - 10 = 1 .
\]

Assume now $\charact(K) = 2$.
Then the last equality follows from the others.
We can choose $a_1, a_2, c_1, c_2, d_1, g_2, h_1$
arbitrarily, and then get a unique homomorphism.
Hence we obtain $\dim_K \Hom_{A^e}(\Omega_{A^e}(A),A) = 7$.
Therefore, we get
\[
  \dim_K HH^1(A) = 5 + 7 - 10 = 2 .
\]

\smallskip

(iii)
We calculate $\dim_K HH^2(A)$ using the exact sequence
of $K$-vector spaces
\[
 0
  \to \Hom_{A^e}\big(\Omega_{A^e}(A),A\big)
  \to \Hom_{A^e}(\bP_1,A)
  \to \Hom_{A^e}\big(\Omega_{A^e}^2(A),A\big)
  \to HH^2(A)
  \to 0
  .
\]
We identify $\Hom_{A^e}(\Omega_{A^e}^2(A),A)$
with the $K$-vector space
\[
  \big\{ \varphi \in  \Hom_{A^e} (\bP_2, A) \,\big|\, \varphi \big(\Omega_{A^e}^3(A)\big) = 0\big\}.
\]
We first indicate the $A$-$A$-bimodule generators of $\Omega_{A^e}^3(A)$.
Consider the following elements in $\bP_2$
\begin{align*}
 \zeta_1 &=
   e_1 \otimes \alpha \delta \gamma \beta
   - \alpha \otimes \delta \gamma \beta
   + \alpha \delta \otimes \gamma \beta
   - \alpha \varepsilon \otimes \xi \beta
   - \alpha \delta \gamma \otimes \beta
   + \alpha \delta \gamma \beta \otimes e_1
,
\\
 \zeta_2 &=
   e_2 \otimes \xi \beta \alpha \varepsilon
   - \xi \otimes \beta \alpha \varepsilon
   + \xi \beta \otimes \alpha \varepsilon
   - \xi \delta \otimes \gamma \varepsilon
   - \xi \beta \alpha \otimes \varepsilon
   + \xi \beta \alpha \varepsilon \otimes e_2
,
\\
 \zeta_3 &=
   e_3 \otimes \gamma \beta \alpha \delta
   - \gamma \otimes \beta \alpha \delta
   + \gamma \beta \otimes \alpha \delta
   - \gamma \varepsilon \otimes \xi \delta
   - \gamma \beta \alpha \otimes \delta
   + \gamma \beta \alpha \delta \otimes e_3
,
\\
 \zeta_4 &=
   e_4 \otimes \beta \alpha \delta \gamma
   - \beta \otimes \alpha \delta \gamma
   - \varepsilon \otimes \xi \beta \alpha
   + \delta \otimes \gamma \beta \alpha
   + \beta \alpha \otimes \delta \gamma
\\&\ \ \ \,
   - \delta \gamma \otimes \beta \alpha
   - \beta \alpha \delta \otimes \gamma
   - \delta \gamma \varepsilon \otimes \xi
   + \delta \gamma \beta \otimes \alpha
   + \beta \alpha \delta \gamma \otimes e_4
.
\end{align*}
Observe that
$\zeta_i \in e_i \bP_2 e_i$ for any $i\in\{1,2,3,4\}$.
Then a straightforward checking shows that
$R(\zeta_i) = 0$ for any $i\in\{1,2,3,4\}$.
We also note that the right $A$-modules
$\zeta_1 A$,
$\zeta_2 A$,
$\zeta_3 A$,
$\zeta_4 A$
are isomorphic to the right $A$-modules
$e_1 A$,
$e_2 A$,
$e_3 A$,
$e_4 A$,
and their socles contain linearly independent elements
of $\bP_2$
\[
  \alpha \delta \gamma \beta
  \otimes
  \alpha \delta \gamma \beta
 ,
 \ \,
  \xi \beta \alpha \varepsilon
  \otimes
  \xi \beta \alpha \varepsilon
 ,
 \ \,
  \gamma \beta \alpha \delta
  \otimes
  \gamma \beta \alpha \delta
 ,
 \ \,
  \beta \alpha \delta \gamma
  \otimes
  \beta \alpha \delta \gamma
 ,
\]
and hence
$\zeta_1 A \oplus \zeta_2 A \oplus \zeta_3 A \oplus \zeta_4 A$
is a right $A^e$-submodule of $\Ker R = \Omega_{A^e}^3(A)$.
Moreover, $\Omega_{A}^3(S_i) \cong S_i$ in $\mod A$
for any $i\in\{1,2,3,4\}$, implies that
$\Omega_{A^e}^3(A) \cong {}_1 A_{\mu}$
for a $K$-algebra automorphism $\mu$ of $A$,
and hence $\dim_K \Omega_{A^e}^3(A) = \dim_K A$.
Therefore,
$\zeta_1$,
$\zeta_2$,
$\zeta_3$,
$\zeta_4$
are $A$-$A$-bimodule generators of $\Omega_{A^e}^3(A)$.
Summing up, we conclude that
\begin{align*}
  \big\{ \varphi \in  \Hom_{A^e} (\bP_2, A) \,|\, &\varphi \big(\Omega_{A^e}^3(A)\big) = 0\big\}
 \\ &
  =
  \big\{ \varphi \in  \Hom_{A^e} (\bP_2, A) \,|\, \varphi (\zeta_i) = 0
        \mbox{ for } i\in\{1,2,3,4\}  \big\}
  .
\end{align*}
Let
$\omega_1 = \alpha \delta \gamma \beta$,
$\omega_2 = \xi \beta \alpha \varepsilon$,
$\omega_3 = \gamma \beta \alpha \delta$,
$\omega_4 = \beta \alpha \delta \gamma$.
Take a homomorphism $\varphi: \bP_2 \to A$ in $\mod A^e$.
Then
\begin{align*}
  \varphi(e_1\otimes e_1) &= a_1 e_1 + b_1 \omega_1, \\
  \varphi(e_2\otimes e_2) &= a_2 e_2 + b_2 \omega_2, \\
  \varphi(e_3\otimes e_3) &= a_3 e_3 + b_3 \omega_3, \\
  \varphi(e_4\otimes e_4) &= a_4 e_4 + b_4 \omega_4 + c \beta \alpha + d \delta \gamma,
\end{align*}
for some elements
$a_1, b_1, a_2, b_2, a_3, b_3, a_4, b_4, c, d \in K$.
We have the equalities
\begin{align*}
  \varphi(\zeta_1)
   &=
    \varphi(e_1\otimes e_1) \alpha \delta \gamma \beta
    - \alpha \varphi(e_4\otimes e_4) \delta \gamma \beta
    + \alpha \delta \varphi(e_3\otimes e_3) \gamma \beta
 \\&\ \ \ \,
    - \alpha \varepsilon \varphi(e_2\otimes e_2) \xi \beta
    - \alpha \delta \gamma \varphi(e_4\otimes e_4) \beta
    + \alpha \delta \gamma \beta \varphi(e_1\otimes e_1)
 \\&
   =
    (2 a_1 + a_2 + a_3 - 2 a_4) \omega_1,
  \\
  \varphi(\zeta_2)
   &=
    \varphi(e_2\otimes e_2) \xi \beta \alpha \varepsilon
    - \xi \varphi(e_4\otimes e_4) \beta \alpha \varepsilon
    + \xi \beta \varphi(e_1\otimes e_1) \alpha \varepsilon
 \\&\ \ \ \,
    - \xi \delta \varphi(e_3\otimes e_3) \gamma \varepsilon
    - \xi \beta \alpha \varphi(e_4\otimes e_4) \varepsilon
    + \xi \beta \alpha \varepsilon \varphi(e_2\otimes e_2)
 \\&
   =
    (a_1 + 2 a_2 + a_3 - 2 a_4) \omega_2,
  \\
  \varphi(\zeta_3)
   &=
    \varphi(e_3\otimes e_3) \gamma \beta \alpha \delta
    - \gamma \varphi(e_4\otimes e_4) \beta \alpha \delta
    + \gamma \beta \varphi(e_1\otimes e_1) \alpha \delta
 \\&\ \ \ \,
    - \gamma \varepsilon \varphi(e_2\otimes e_2) \xi \delta
    - \gamma \beta \alpha \varphi(e_4\otimes e_4) \delta
    + \gamma \beta \alpha \delta \varphi(e_3\otimes e_3)
 \\&
   =
    (a_1 + a_2 + 2 a_3 - 2 a_4) \omega_3,
 \\
  \varphi(\zeta_4)
   &=
    \varphi(e_4\otimes e_4) \beta \alpha \delta \gamma
    - \beta \varphi(e_1\otimes e_1) \alpha \delta \gamma
    - \varepsilon \varphi(e_2\otimes e_2) \xi \beta \alpha
 \\&\ \ \ \,
    + \delta \varphi(e_3\otimes e_3) \gamma \beta \alpha
    + \beta \alpha \varphi(e_4\otimes e_4) \delta \gamma
    - \delta \gamma \varphi(e_4\otimes e_4) \beta \alpha
 \\&\ \ \ \,
    - \beta \alpha \delta \varphi(e_3\otimes e_3) \gamma
    - \delta \gamma \varepsilon \varphi(e_2\otimes e_2) \xi
    + \delta \gamma \beta \varphi(e_1\otimes e_1) \alpha
 \\&\ \ \ \,
    + \beta \alpha \delta \gamma \varphi(e_4\otimes e_4)
 \\&
   =
    - 2 (a_2 + a_3 - 2 a_4) \omega_4.
\end{align*}
Hence, $\varphi$ factors through $\Omega_{A^e}^2(A)$
if and only if
\begin{align*}
 2 a_1 + a_2 + a_3 - 2 a_4 &= 0, \\
 a_1 + 2 a_2 + a_3 - 2 a_4 &= 0, \\
 a_1 + a_2 + 2 a_3 - 2 a_4 &= 0, \\
 2 (a_2 + a_3 - 2 a_4) &= 0
.
\end{align*}

Assume $\charact(K) \neq 2$.
Then the equations are equivalent to
$a_1 = a_2 = a_3 = a_4 = 0$,
and hence
$\dim_K \Hom_{A^e}(\Omega_{A^e}^2(A),A) = 10 - 4 = 6$.
Then we get
\[
  \dim_K HH^2(A) = 6 + 6 - 12 = 0 .
\]

Assume $\charact(K) = 2$.
Then the equations are equivalent to
$a_1 = a_2 = a_3$,
and hence
$\dim_K \Hom_{A^e}(\Omega_{A^e}^2(A),A) = 10 - 2 = 8$.
Then we get
\[
  \dim_K HH^2(A) = 7 + 8 - 12 = 3 .
\]
\end{proof}

We note that
$\dim_K HH^2(\Lambda_9') = 3$
for $\charact(K) = 2$ coincides
with \cite[Theorem~4.1]{AK2}.

\begin{proposition}
\label{prop:6.6}
Let $K$ be an algebraically closed field of characteristic 2,
$\Lambda$ the non-standard algebra $\Lambda_9$ over $K$,
and $A$ a standard self-injective algebra over $K$.
Then $A$ and $\Lambda$ are not derived equivalent.
\end{proposition}

\begin{proof}
Assume that $A$ and $\Lambda$ are derived equivalent.
Then it follows from Theorem~\ref{thm:2.8} that
$A$ and $\Lambda$ are stably equivalent,
and hence
the stable Auslander-Reiten quivers $\Gamma_A^s$
and $\Gamma_{\Lambda}^s$ are isomorphic.
In particular, we conclude that $\Gamma_{A}^s$
consists of stable tubes of ranks $1$ and $3$.
Further, by
Theorems \ref{thm:2.9} and \ref{thm:2.10},
$A$ is a representation-infinite periodic algebra
of polynomial growth, and hence is a standard
self-injective algebra of tubular type $(3,3,3)$.
Moreover, $\Lambda$ is not a symmetric algebra by
Proposition~\ref{prop:4.5}.
Hence, applying Theorem~\ref{thm:2.7},
we infer that $A$ is not a symmetric algebra.
We also note that the Grothendieck groups $K_0(A)$
and $K_0(\Lambda)$ are isomorphic,
and then $K_0(A)$ is of rank $4$.
Therefore,
$A$ is an exceptional
(in the sense of \cite{BiS1,S2})
 standard non-symmetric
self-injective algebra of tubular type $(3,3,3)$.
Applying
\cite[Theorem~1]{BiS2} and \cite[Theorem~4.2]{BiS3}
we conclude that $A$ is isomorphic to the algebra
given by the quiver
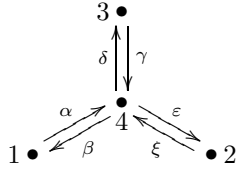
\[
  \xymatrix@C=1.9pc@R=.75pc{
      & \bullet  \ar@<+.5ex>[dd]^{\gamma}
     \save[] +<-2.5mm,0mm> *{3} \restore
      \\ \\
      & \bullet  \ar@<+.5ex>[uu]^{\delta}
           \ar@<+.5ex>[dr]^{\varepsilon}
           \ar@<+.5ex>[dl]^{\beta}
           \save[] +<0mm,-2.5mm> *{4} \restore
           \\
      \bullet  \ar@<+.5ex>[ru]^{\alpha}
     \save[] +<-2.5mm,0mm> *{1} \restore
      &&  \bullet  \ar@<+.5ex>[lu]^{\xi}
     \save[] +<2.5mm,0mm> *{2} \restore
  }
\]
and the relations
\[
  \beta \alpha + \varepsilon \xi + \delta \gamma =  0,
 \ \,
  \alpha \delta = 0,
 \ \,
  \xi \beta = 0,
 \ \,
  \gamma \varepsilon = 0.
\]
We note that it is the canonical mesh algebra
$\Lambda(\bG_2)$ of Dynkin type
(in the sense of \cite[Secion~7]{ESk2}),
which has been shown in \cite[Proposition~8.1]{Du2}
to be a periodic algebra of period $3$ (in characteristic $2$).
Therefore, since $\Lambda = \Lambda_9$
is periodic algebra of period $6$,
applying Theorem~\ref{thm:2.10},
we conclude that $A$ and $\Lambda$ are not derived equivalent.
This also follows from Theorem~\ref{thm:2.6} and the fact
that
$\dim_K C(\Lambda) = 5$
and
$\dim_K C(\Lambda(\bG_2)) = 2$.
\end{proof}

\section{Hochschild cohomology for tubular type $(2,4,4)$}%
\label{sec:type244}

Our first aim is to determine the dimensions
of the low Hochschild cohomology spaces
of the exceptional periodic algebras of tubular type $(2,4,4)$.

\begin{proposition}
\label{prop:7.1}
Let $A$ be one of the algebras
$\Lambda_4'$,
$\Lambda_5'$,
$\Lambda_6'$,
$\Lambda_7'$,
$\Lambda_8'$.
Then
\begin{enumerate}[(i)]
\item
  $\dim_K HH^0(A) = 5$.
\item
  $\dim_K HH^1(A) = \left\{ \begin{array}{cl} 2, & \charact (K \neq 2 \\ 3, & \charact (K) = 2 \end{array} \right.$.\vspace{2pt}
\item
  $\dim_K HH^2(A) = \left\{ \begin{array}{cl} 2, & \charact (K) \neq 2 \\ 3, & \charact (K) = 2 \end{array} \right.$.
\end{enumerate}
\end{proposition}

\begin{proof}
We may assume, by Theorems \ref{thm:2.6} and \ref{thm:4.1},
that $A = \Lambda_6'$.
It follows from the proof of \cite[Proposition~9.1]{BES2}
that $A$ admits
the first three terms of a minimal projective resolution in $\mod A^e$
\[
  \bP_2 \xrightarrow{d_2}
  \bP_1 \xrightarrow{d_1}
  \bP_0 \xrightarrow{d_0}
  A \rightarrow 0
\]
with
\begin{align*}
 \bP_0 &=
   P(1,1) \oplus P(2,2) \oplus P(3,3)
,
\\
 \bP_1 &=
   P(1,2) \oplus P(2,1) \oplus P(2,3) \oplus P(3,2)
,
\\
 \bP_2 &=
   P(1,1) \oplus P(1,3) \oplus P(3,1) \oplus P(2,2)
,
\end{align*}
the differential
$R = d_2$
given by
\begin{align*}
 R(e_1 \otimes e_1)
  &= \varrho(\alpha \beta)
   = e_1 \otimes \beta
     + \alpha \otimes e_1
,\\
 R(e_1 \otimes e_3)
  &= \varrho(\alpha \delta \gamma \delta)
   = e_1 \otimes \delta \gamma \delta + \alpha \otimes \gamma \delta
     + \alpha \delta \otimes \delta + \alpha \delta \gamma \otimes e_3
,\\
 R(e_3 \otimes e_1)
  &= \varrho(\gamma \delta \gamma \beta)
   = e_3 \otimes \delta \gamma \beta + \gamma \otimes \gamma \beta
     + \gamma \delta \otimes \beta + \gamma \delta \gamma \otimes e_1
,\\
 R(e_2 \otimes e_2)
  &= \varrho(\beta \alpha - \delta \gamma \delta \gamma)
   = e_2 \otimes \alpha + \beta \otimes e_2
     - e_2 \otimes \gamma \delta \gamma
     - \delta \otimes \delta \gamma
 \\&\ \ \ \,
     - \delta \gamma \otimes \gamma
     - \delta \gamma \delta \otimes e_2
,
\end{align*}
and the $A$-$A$-bimodule
$\Omega_{A^e}^3(A) = \Ker R$
generated by the following elements in $\bP_2$
\begin{align*}
 \Gamma_{12} &=
  \alpha \otimes e_2
  - e_1 \otimes \alpha
  + e_1 \otimes \gamma
,
\\
 \Gamma_{21} &=
  \beta \otimes e_1
  - e_2 \otimes \beta
  - \delta \otimes e_1
,
\\
 \Gamma_{33} &=
  \gamma \beta \otimes e_3
  - e_3 \otimes \alpha \delta
  - \gamma \otimes \delta \gamma \delta
  + \gamma \delta \gamma \otimes \delta
.
\end{align*}

We fix also the following basis
$\cB = e_1 \cB \cup e_2 \cB \cup e_3 \cB$
of the $K$-vector space $A$:
\begin{align*}
e_1 \cB &=
\{ e_1, \alpha, \alpha \delta, \alpha \delta \gamma, \alpha \delta \gamma \beta
\}
,
\\
e_2 \cB &=
\{
e_2, \beta, \delta, \delta \gamma, \delta \gamma \delta,
\delta \gamma \beta, \beta \alpha, \beta \alpha \delta, \beta \alpha \delta \gamma
\}
,
\\
e_3 \cB &=
\{
e_3, \gamma, \gamma \beta, \gamma \delta, \gamma \delta \gamma, \gamma \delta \gamma \delta,
\gamma \beta \alpha, \gamma \beta \alpha \delta
\}
.
\end{align*}

We shall show now that the equalities (i), (ii), (iii) hold.

\smallskip

(i)
A direct checking shows that the center $C(A)$
of $A$ has the $K$-linear basis
\[
  \big\{ 1, (\delta \gamma)^2 + (\gamma \delta)^2, \alpha \delta \gamma \beta,
     \beta \alpha \delta \gamma, \gamma \beta \alpha \delta \big\}
\]
and hence $\dim_K HH^0(A) = \dim_K C(A) = 5$.

\smallskip

(ii)
We will determine $\dim_K HH^1(A)$ using the exact sequence
of $K$-vector spaces (see Proposition~\ref{prop:2.5})
\[
 0
  \to HH^0(A)
  \to \Hom_{A^e}(\bP_0,A)
  \to \Hom_{A^e}\big(\Omega_{A^e}(A),A\big)
  \to HH^1(A)
  \to 0
  .
\]
We have isomorphisms of $K$-vector spaces
\begin{align*}
 \Hom_{A^e} (\bP_0, A) &=
  e_1 A e_1 \oplus e_2 A e_2 \oplus e_3 A e_3
,
\\
 \Hom_{A^e} (\bP_1, A) &=
   e_1 A e_2 \oplus
   e_2 A e_1 \oplus
   e_2 A e_3 \oplus e_3 A e_2 ,
\end{align*}
and hence we obtain
\begin{align*}
 \dim_K \Hom_{A^e} (\bP_0, A) &=
   2+4+4=10
,
\\
 \dim_K \Hom_{A^e} (\bP_1, A) &=
   2+2+3+3=10
.
\end{align*}
We identify $\Hom_{A^e}(\Omega_{A^e}(A),A)$
with the $K$-vector space
\[
  \big\{ \varphi \in  \Hom_{A^e} (\bP_1, A) \,|\, \varphi R=0\big\} .
\]
and use the chosen basis $\cB = e_1 \cB \cup e_2 \cB \cup e_3 \cB$ of $A$.
Recall also that
\[
 \bP_2
 = A e_1 \otimes e_1 A \oplus
   A e_1 \otimes e_3 A \oplus
   A e_3 \otimes e_1 A \oplus
   A e_2 \otimes e_2 A.
\]

Let $\varphi: \bP_1\to A$ be a homomorphism in $\mod A^e$.
Then there exist
$a_1, a_2, b_1, b_2$, $c_1, c_2, c_3, d_1, d_2, d_3 \in K$
such that
\begin{align*}
  \varphi(e_1\otimes e_2) &= a_1 \alpha + a_2 \alpha \delta \gamma,\\
  \varphi(e_2\otimes e_1) &= b_1 \beta + b_2 \delta \gamma \beta,\\
  \varphi(e_2\otimes e_3) &= c_1 \delta + c_2 \delta \gamma \delta + c_3 \beta \alpha \delta,\\
  \varphi(e_3\otimes e_2) &= d_1 \gamma + d_2 \gamma \delta \gamma + d_3 \gamma \beta \alpha.
\end{align*}
We have the equalities
\begin{align*}
  \varphi \big( R(e_1\otimes e_1) \big)
  &=
   \varphi (e_1 \otimes e_2) \beta
   + \alpha \varphi (e_2 \otimes e_1)
 \\
  &=
   (a_2 + b_2) \alpha \delta \gamma \beta ,
 \\
  \varphi \big( R(e_1 \otimes e_3) \big)
  &=
   \varphi (e_1 \otimes e_2) \delta \gamma \delta
   + \alpha \varphi (e_2 \otimes e_3) \gamma \delta
   + \alpha \delta \varphi (e_3 \otimes e_2) \delta
 \\&\ \ \ \,
   + \alpha \delta \gamma \varphi (e_2 \otimes e_3)
 \\
  &=
   0 ,
 \\
  \varphi \big( R(e_3 \otimes e_1) \big)
  &=
   \varphi (e_3 \otimes e_2) \delta \gamma \beta
   + \gamma \varphi (e_2 \otimes e_3) \gamma \beta
   + \gamma \delta \varphi (e_3 \otimes e_2) \beta
 \\&\ \ \ \,
   + \gamma \delta \gamma \varphi (e_2 \otimes e_1)
 \\
  &=
   0 ,
 \\
  \varphi \big( R(e_2\otimes e_2) \big)
  &=
   \varphi (e_2 \otimes e_1) \alpha
   + \beta \varphi (e_1 \otimes e_2)
   - \varphi (e_2 \otimes e_3) \gamma \delta \gamma
 \\&\ \ \ \,
   - \delta \varphi (e_3 \otimes e_2) \delta \gamma
   - \gamma \delta \varphi (e_2 \otimes e_3) \gamma
   - \delta \gamma \delta \varphi (e_3 \otimes e_2)
 \\
  &=
   (a_1 + b_1 - 2 c_1 - 2 d_1) \beta \alpha
   + (a_2 + b_2 - 2 c_2 - 2 d_2) \gamma \beta \alpha \delta
.
\end{align*}
Hence,
$\varphi$ factors through $\Omega_{A^e}(A)$
if and only if
\[
  a_2 + b_2=0,
  \ \,
  a_1 + b_1 - 2 c_1 - 2 d_1 = 0,
  \ \,
  a_2 + b_2 - 2 c_2 - 2 d_2 = 0
  .
\]

Assume $\charact(K) \neq 2$.
We can choose
$a_1, b_1, b_2, c_1, c_3, d_2, d_3$
arbitrarily,
and then the other parameters are uniquely determined.
Thus $\dim_K \Hom_{A^e}(\Omega_{A^e}(A),A) = 7$.
Hence we get
\begin{align*}
  \dim_K HH^1(A)
   &= \dim_K HH^0(A) + \dim_K \Hom_{A^e}\big(\Omega_{A^e}(A),A\big)
 \\&\ \ \ \,
      - \dim_K \Hom_{A^e}(\bP_0,A)
   \\&
    = 5 + 7 - 10 = 2 .
\end{align*}

Assume $\charact(K) = 2$.
Then $\varphi$ factors through $\Omega_{A^e}(A)$
if and only if
$a_1 = b_1$ and $a_2 = b_2$.
Thus $\dim_K \Hom_{A^e}(\Omega_{A^e}(A),A) = 10 - 2 = 8$.
Hence we get
\begin{align*}
  \dim_K HH^1(A)
   &= \dim_K HH^0(A) + \dim_K \Hom_{A^e}\big(\Omega_{A^e}(A),A\big)
 \\&\ \ \ \,
      - \dim_K \Hom_{A^e}(\bP_0,A)
   \\&
    = 5 + 8 - 10 = 3 .
\end{align*}

\smallskip

(iii)
We calculate $\dim_K HH^2(A)$ using the exact sequence
of $K$-vector spaces
\[
 0
  \to \Hom_{A^e}\big(\Omega_{A^e}(A),A\big)
  \to \Hom_{A^e}(\bP_1,A)
  \to \Hom_{A^e}\big(\Omega_{A^e}^2(A),A\big)
  \to HH^2(A)
  \to 0
  .
\]
We identify $\Hom_{A^e}(\Omega_{A^e}^2(A),A)$
with the $K$-vector space
\[
  \big\{ \varphi \in  \Hom_{A^e} (\bP_2, A) \,|\, \varphi (\Omega_{A^e}^3(A)) = 0\big\} .
\]
which is equal to the $K$-vector space
\[
  \big\{ \varphi \in \Hom_{A^e} (\bP_2, A) \,|\,
  \varphi (\Gamma_{12}) = 0, \varphi (\Gamma_{21}) = 0, \varphi (\Gamma_{33}) = 0  \big\}.
\]

Let $\varphi: \bP_2\to A$ be a homomorphism in $\mod A^e$.
Then there exist elements
$a_1, a_2, b_1, b_2, b_3, b_4$, $c_3, d_3 \in K$
such that
\begin{align*}
  \varphi(e_1\otimes e_1) &= a_1 e_1 + a_2 \alpha \delta \gamma \beta , \\
  \varphi(e_2\otimes e_2) &= b_1 e_2 + b_2 \delta \gamma + b_3 \beta \alpha
                             + b_4 \beta \alpha \delta \gamma , \\
  \varphi(e_1\otimes e_3) &= c_3 \alpha \delta , \\
  \varphi(e_3\otimes e_1) &= d_3 \gamma \beta .
\end{align*}
We have the equalities
\begin{align*}
  \varphi ( \Gamma_{12} )
  &=
   \alpha \varphi (e_2\otimes e_2)
   - \varphi (e_1\otimes e_1) \alpha
   + \varphi (e_1\otimes e_3) \gamma
 \\
  &=
   (b_1 - a_1) \alpha
   + (b_2 + c_3) \alpha \delta \gamma
,
 \\
  \varphi ( \Gamma_{21} )
  &=
   \beta \varphi (e_1\otimes e_1)
   - \varphi (e_2\otimes e_2) \beta
   - \delta \varphi (e_3\otimes e_1)
 \\
  &=
   (a_1 - b_1) \beta
   - (b_2 + d_3) \delta \gamma \beta
,
 \\
  \varphi ( \Gamma_{33} )
  &=
   \gamma \beta \varphi (e_1 \otimes e_3)
   - \varphi (e_3\otimes e_1) \alpha \delta
   - \gamma \varphi (e_2\otimes e_2) \delta \gamma \delta
   + \gamma \delta \gamma \varphi (e_2\otimes e_2) \delta
 \\
  &=
   (c_3 - d_3) \gamma \beta \alpha \delta
.
\end{align*}
Hence,
$\varphi$ factors through $\Omega_{A^e}^2(A)$
if and only if
\[
  a_1 = b_1,
  \ \,
  c_3 = d_3 = -b_2
  .
\]
Thus
$\dim_K \Hom_{A^e}(\Omega_{A^e}^2(A),A) = 8 - 3 = 5$.
Therefore,
\begin{align*}
  \dim_K HH^2(A)
   &= \dim_K \Hom_{A^e}\big(\Omega_{A^e}(A),A\big)
      + \dim_K \Hom_{A^e}\big(\Omega_{A^e}^2(A),A\big)
   \\&\ \ \ \,
      - \dim_K \Hom_{A^e}(\bP_1,A)
\end{align*}
equals
$2 = 7+5-10$ if $\charact(K) \neq 2$,
and
$3 = 8+5-10$ if $\charact(K) = 2$.

We note that we have always $\dim_K HH^1(A) = \dim_K HH^2(A)$.
\end{proof}

\begin{proposition}
\label{prop:7.2}
Let $K$ be an algebraically closed field of characteristic $2$
and $A$ be one of the algebras
$\Lambda_4$,
$\Lambda_5$,
$\Lambda_6$,
$\Lambda_7$,
$\Lambda_8$.
Then
\begin{enumerate}[(i)]
\item
  $\dim_K HH^0(A) = 5$.
\item
  $\dim_K HH^1(A) = 2$.
\item
  $\dim_K HH^2(A) = 2$.
\end{enumerate}
\end{proposition}

\begin{proof}
We may assume, by Theorems \ref{thm:2.6} and \ref{thm:4.1},
that $A = \Lambda_6$.
It follows from the proof of \cite[Proposition~9.4]{BES2}
that $A$ admits
the first three terms of a minimal projective resolution in $\mod A^e$
\[
  \bP_2 \xrightarrow{d_2}
  \bP_1 \xrightarrow{d_1}
  \bP_0 \xrightarrow{d_0}
  A \rightarrow 0
\]
with
\begin{align*}
 \bP_0 &=
   P(1,1) \oplus P(2,2) \oplus P(3,3)
,
\\
 \bP_1 &=
   P(1,2) \oplus P(2,1) \oplus P(2,3) \oplus P(3,2)
,
\\
 \bP_2 &=
   P(1,1) \oplus P(1,3) \oplus P(3,1) \oplus P(2,2)
,
\end{align*}
the differential
$R = d_2$
given by
\begin{align*}
 R(e_1 \otimes e_1)
  &= \varrho(\alpha \beta - \alpha \delta \gamma \beta)
   = e_1 \otimes \beta
     + \alpha \otimes e_1
     - e_1 \otimes \delta \gamma \beta
     - \alpha \otimes \gamma \beta
 \\&\ \ \ \,
     - \alpha \delta \otimes \beta
     - \alpha \delta \gamma \otimes e_1.
,\\
 R(e_1 \otimes e_3)
  &= \varrho(\alpha \delta \gamma \delta)
   = e_1 \otimes \delta \gamma \delta + \alpha \otimes \gamma \delta
     + \alpha \delta \otimes \delta + \alpha \delta \gamma \otimes e_3
,\\
 R(e_3 \otimes e_1)
  &= \varrho(\gamma \delta \gamma \beta)
   = e_3 \otimes \delta \gamma \beta + \gamma \otimes \gamma \beta
     + \gamma \delta \otimes \beta + \gamma \delta \gamma \otimes e_1
,\\
 R(e_2 \otimes e_2)
  &= \varrho(\beta \alpha - \delta \gamma \delta \gamma)
   = e_2 \otimes \alpha + \beta \otimes e_2
     - e_2 \otimes \gamma \delta \gamma
     - \delta \otimes \delta \gamma
 \\&\ \ \ \,
     - \delta \gamma \otimes \gamma
     - \delta \gamma \delta \otimes e_2
,
\end{align*}
and the $A$-$A$-bimodule
$\Omega_{A^e}^3(A) = \Ker R$
generated by the following elements in $\bP_2$
\begin{align*}
 \Delta_{12} &=
  \alpha \otimes e_2
  + e_1 \otimes \alpha
  + e_1 \otimes \gamma
  + \alpha \delta \gamma \otimes e_2
  + e_1 \otimes \gamma \delta \gamma
,
\\
 \Delta_{21} &=
  \beta \otimes e_1
  + e_2 \otimes \beta
  + \delta \otimes e_1
  + e_2 \otimes \delta \gamma \beta
  + \delta \gamma \delta \otimes e_1
,
\\
 \Delta_{33} &=
  \gamma \beta \otimes e_3
  + e_3 \otimes \alpha \delta
  + \gamma \otimes \delta \gamma \delta
  + \gamma \delta \gamma \otimes \delta
.
\end{align*}

We shall use the basis
$\cB = e_1 \cB \cup e_2 \cB \cup e_3 \cB$
of
$A$,
defined in the proof of Proposition~\ref{prop:7.1}.

We shall show now that the equalities (i), (ii), (iii) hold.

\smallskip

(i)
The center $C(A)$ of $A$ has the $K$-linear basis
\[
  \{ 1, (\delta \gamma)^2 + (\gamma \delta)^2,
     \alpha \delta \gamma \beta,
     \beta \alpha \delta \gamma,
     \gamma \beta \alpha \delta \}
\]
and hence $\dim_K HH^0(A) = \dim_K C(A) = 5$.

\smallskip

(ii)
We proceed as in the proof
of Proposition~\ref{prop:7.1}.
We have again
$\dim_K \Hom_{A^e} (\bP_0, A) = 10$
and
$\dim_K \Hom_{A^e} (\bP_1, A) = 10$.
We use the same basis
$\cB = e_1 \cB \cup e_2 \cB \cup e_3 \cB$ of $A$ 
as in the proof of Proposition~\ref{prop:7.1}.

Let $\varphi: \bP_1\to A$ be a homomorphism in $\mod A^e$.
Then
\begin{align*}
  \varphi(e_1\otimes e_2) &= a_1 \alpha + a_2 \alpha \delta \gamma,\\
  \varphi(e_2\otimes e_1) &= b_1 \beta + b_2 \delta \gamma \beta,\\
  \varphi(e_2\otimes e_3) &= c_1 \delta + c_2 \delta \gamma \delta + c_3 \beta \alpha \delta,\\
  \varphi(e_3\otimes e_2) &= d_1 \gamma + d_2 \gamma \delta \gamma + d_3 \gamma \beta \alpha,
\end{align*}
for some elements
$a_1, a_2, b_1, b_2, c_1, c_2, c_3, d_1, d_2, d_3 \in K$.
We have the equalities
\begin{align*}
  \varphi \big( R(e_1\otimes e_1) \big)
  &=
   \varphi (e_1 \otimes e_2) \beta
   + \alpha \varphi (e_2 \otimes e_1)
   - \varphi (e_1 \otimes e_2) \delta \gamma \beta
 \\&\ \ \ \,
   - \alpha \varphi (e_2 \otimes e_3) \gamma \beta
   - \alpha \delta \varphi (e_3 \otimes e_2) \beta
   - \alpha \delta \gamma \varphi (e_2 \otimes e_1)
 \\
  &=
   (a_2 + b_2 - c_1 - d_1) \alpha \delta \gamma \beta ,
 \\
  \varphi \big( R(e_1 \otimes e_3) \big)
  &=
   \varphi (e_1 \otimes e_2) \delta \gamma \delta
   + \alpha \varphi (e_2 \otimes e_3) \gamma \delta
   + \alpha \delta \varphi (e_3 \otimes e_2) \delta
 \\&\ \ \ \,
   + \alpha \delta \gamma \varphi (e_2 \otimes e_3)
 \\
  &=
   0 ,
 \\
  \varphi \big( R(e_3 \otimes e_1) \big)
  &=
   \varphi (e_3 \otimes e_2) \delta \gamma \beta
   + \gamma \varphi (e_2 \otimes e_3) \gamma \beta
   + \gamma \delta \varphi (e_3 \otimes e_2) \beta
 \\&\ \ \ \,
   + \gamma \delta \gamma \varphi (e_2 \otimes e_1)
 \\
  &=
   0 ,
 \\
  \varphi \big( R(e_2\otimes e_2) \big)
  &=
   \varphi (e_2 \otimes e_1) \alpha
   + \beta \varphi (e_1 \otimes e_2)
   - \varphi (e_2 \otimes e_3) \gamma \delta \gamma
 \\&\ \ \ \,
   - \delta \varphi (e_3 \otimes e_2) \delta \gamma
   - \delta \gamma \varphi (e_2 \otimes e_3) \gamma
   - \delta \gamma \delta \varphi (e_3 \otimes e_2)
 \\
  &=
   (a_1 + b_1 - 2 c_1 - 2 d_1) \beta \alpha
   + (a_2 + b_2 - 2 c_2 - 2 d_2) \gamma \beta \alpha \delta
 \\
  &=
   (a_1 + b_1) \beta \alpha
   + (a_2 + b_2) \gamma \beta \alpha \delta
.
\end{align*}
Hence,
$\varphi$ factors through $\Omega_{A^e}(A)$
if and only if
\[
  a_1 + b_1 = 0,
  \ \,
  a_2 + b_2 = 0
  \ \,
  c_1 + d_1 = 0
  .
\]
It follows that
$\dim_K \Hom_{A^e}(\Omega_{A^e}(A),A) = 10 - 3 = 7$.
Therefore, we obtain
\begin{align*}
  \dim_K HH^1(A)
   &= \dim_K HH^0(A) + \dim_K \Hom_{A^e}\big(\Omega_{A^e}(A),A\big)
  \\&\ \ \ \,
      - \dim_K \Hom_{A^e}(\bP_0,A)
   \\&
    = 5 + 7 - 10 = 2 .
\end{align*}

\smallskip

(iii)
Let $\varphi: \bP_2\to A$ be a homomorphism in $\mod A^e$.
Then
\begin{align*}
  \varphi(e_1\otimes e_1) &= a_1 e_1 + a_2 \alpha \delta \gamma \beta , \\
  \varphi(e_2\otimes e_2) &= b_1 e_2 + b_2 \delta \gamma + b_3 \beta \alpha
                             + b_4 \beta \alpha \delta \gamma , \\
  \varphi(e_1\otimes e_3) &= c_3 \alpha \delta , \\
  \varphi(e_3\otimes e_1) &= d_3 \gamma \beta ,
\end{align*}
for some elements
$a_1, a_2, b_1, b_2, b_3, b_4$, $c_3, d_3 \in K$.

Recall that the $A$-$A$-bimodule $\Omega_{A^e}^3(A)$
is generated by the elements
$\Delta_{12}$,
$\Delta_{21}$,
$\Delta_{33}$.
We have the equalities
\begin{align*}
  \varphi ( \Delta_{12} )
  &=
   \alpha \varphi (e_2 \otimes e_2)
   + \varphi (e_1 \otimes e_1) \alpha
   + \varphi (e_1 \otimes e_3) \gamma
   + \alpha \delta \gamma \varphi (e_2 \otimes e_2)
 \\&\ \ \ \,
   + \varphi (e_1 \otimes e_3) \gamma \delta \gamma \delta
 \\
  &=
   (a_1 + b_1) \alpha
   + (b_1 + b_2 + c_3) \alpha \delta \gamma
,
 \\
  \varphi ( \Delta_{21} )
  &=
   \beta \varphi (e_1\otimes e_1)
   + \varphi (e_2\otimes e_2) \beta
   + \delta \varphi (e_3\otimes e_1)
   + \varphi (e_2 \otimes e_2) \delta \gamma \beta
 \\&\ \ \ \,
   + \delta \gamma \delta \varphi (e_3 \otimes e_1)
 \\
  &=
   (a_1 + b_1) \beta
   + (b_1 + b_2 + d_3) \delta \gamma \beta
,
 \\
  \varphi ( \Delta_{33} )
  &=
   \gamma \beta \varphi (e_1 \otimes e_3)
   + \varphi (e_3\otimes e_1) \alpha \delta
   + \gamma \varphi (e_2\otimes e_2) \delta \gamma \delta
   + \gamma \delta \gamma \varphi (e_2\otimes e_2) \delta
 \\
  &=
   (c_3 + d_3) \gamma \beta \alpha \delta
.
\end{align*}
Hence,
$\varphi$ factors through $\Omega_{A^e}^2(A)$
if and only if
\[
  a_1 = b_1,
  \ \,
  c_3 = d_3
  \ \,
  b_1 + b_2 + c_3 = 0
  .
\]
It follows that
$\dim_K \Hom_{A^e}(\Omega_{A^e}^2(A),A) = 8 - 3 = 5$.
Therefore, we obtain
\begin{align*}
  \dim_K HH^2(A)
   &= \dim_K \Hom_{A^e}\big(\Omega_{A^e}(A),A\big)
      + \dim_K \Hom_{A^e}\big(\Omega_{A^e}^2(A),A\big)
   \\&\ \ \ \,
      - \dim_K \Hom_{A^e}(\bP_1,A)
   \\&
    = 7 + 5 - 10 = 2 .
\end{align*}

Observe that $\dim_K HH^1(A) = \dim_K HH^2(A)$.
\end{proof}

\begin{proposition}
\label{prop:7.3}
Let $K$ be an algebraically closed field of characteristic $2$,
$\Lambda$ one of the non-standard algebras
$\Lambda_4$, $\Lambda_5$, $\Lambda_6$, $\Lambda_7$
or $\Lambda_8$ over $K$,
and $A$ a standard self-injective algebra over $K$.
Then $A$ and $\Lambda$ are not derived equivalent.
\end{proposition}

\begin{proof}
We know from Theorem~\ref{thm:4.3}
that
the algebras
$\Lambda_4$, $\Lambda_5$, $\Lambda_6$, $\Lambda_7$
or $\Lambda_8$ are derived equivalent.
Assume that $A$ and $\Lambda$ are derived equivalent.
Then it follows from Theorem~\ref{thm:2.8} that
$A$ and $\Lambda$ are stably equivalent, and hence
the stable Auslander-Reiten quivers $\Gamma_A^s$
and $\Gamma_{\Lambda}^s$ are isomorphic.
In particular, we conclude that
$\Gamma_{A}^s$
consists of stable tubes of ranks $1$, $2$ and $4$.
Further, by
Theorems \ref{thm:2.9} and \ref{thm:2.10},
$A$ is a representation-infinite periodic algebra
of polynomial growth, and hence a standard
self-injective algebra of tubular type $(2,4,4)$.
Moreover, $A$ is a symmetric algebra by
Theorem~\ref{thm:2.7} and Proposition~\ref{prop:4.5}.
We also know that $K_0(A)$ is of rank $3$.
Applying now \cite[Theorem~1]{BiS2} we infer
that $A$ is isomorphic to one of the algebras
$\Lambda_4'$, $\Lambda_5'$, $\Lambda_6'$, $\Lambda_7'$
or $\Lambda_8'$.
On the other hand, it follows from
Theorem~\ref{thm:2.6} and
Propositions \ref{prop:7.1} and \ref{prop:7.2} that
\[
  \dim_K HH^2(\Lambda)
  = \dim_K HH^2(\Lambda_6)
  = 2
  < 3
  = \dim_K HH^2(\Lambda_6')
  = \dim_K HH^2(A)
.
\]
Therefore, applying Theorem \ref{thm:2.6} again,
we conclude that
$A$ and $\Lambda$ are not derived equivalent.
\end{proof}

\section{Hochschild cohomology for tubular type $(2,3,6)$}%
\label{sec:type236}

Our first aim is to determine the dimensions
of the low Hochschild cohomology spaces
of the exceptional algebras $\Lambda_{10}'$ and $\Lambda_{10}$
(in characteristic $2$) of tubular type $(2,3,6)$.

\begin{proposition}
\label{prop:8.1}
Let $A = \Lambda_{10}'$.
Then
\begin{enumerate}[(i)]
\item
  $\dim_K HH^0(A) = 2$.
\item
  $\dim_K HH^1(A) = \left\{ \begin{array}{cl} 0, & \charact (K \neq 2 \\ 1, & \charact (K) = 2 \end{array} \right.$.\vspace{2pt}
\item
  $\dim_K HH^2(A) = \left\{ \begin{array}{cl} 0, & \charact (K) \neq 2 \\ 1, & \charact (K) = 2 \end{array} \right.$.
\end{enumerate}
\end{proposition}

\begin{proof}
It follows from the proof of \cite[Proposition~10.1]{BES2}
that $A$ admits
the first three terms of a minimal projective resolution in $\mod A^e$
\[
  \bP_2 \xrightarrow{d_2}
  \bP_1 \xrightarrow{d_1}
  \bP_0 \xrightarrow{d_0}
  A \rightarrow 0
\]
with
\begin{align*}
 \bP_0 &=
   P(1,1) \oplus P(2,2) \oplus P(3,3) \oplus P(4,4) \oplus P(5,5)
,
\\
 \bP_1 &=
   P(1,2) \oplus P(2,1) \oplus P(1,3) \oplus P(3,1) \oplus P(2,5)
          \oplus P(5,3) \oplus P(3,4) \oplus P(4,2)
,
\\
 \bP_2 &=
   P(1,1) \oplus P(2,3) \oplus P(3,2) \oplus P(4,5) \oplus P(5,4)
,
\end{align*}
the differential
$R = d_2$
given by
\begin{align*}
 R(e_1 \otimes e_1)
  &= \varrho(\sigma \delta - \gamma \xi)
   = e_1 \otimes \delta
     + \sigma \otimes e_1
     - e_1 \otimes \xi
     - \gamma \otimes e_1
,\\
 R(e_2 \otimes e_3)
  &= \varrho(\eta \mu - \xi \sigma)
   = e_2 \otimes \mu
     + \eta \otimes e_3
     - e_2 \otimes \sigma
     - \xi \otimes e_3
,\\
 R(e_3 \otimes e_2)
  &= \varrho(\beta \alpha - \delta \gamma)
   = e_3 \otimes \alpha
     + \beta \otimes e_2
     - e_3 \otimes \gamma
     - \delta \otimes e_2
,\\
 R(e_4 \otimes e_5)
  &= \varrho(\alpha \eta)
   = e_4 \otimes \eta
     + \alpha \otimes e_5
,\\
 R(e_5 \otimes e_4)
  &= \varrho(\mu \beta)
   = e_5 \otimes \beta
     + \mu \otimes e_4
.
\end{align*}
Moreover, the $A$-$A$-bimodule
$\Omega_{A^e}^3(A) = \Ker R$
is generated by the following elements in $\bP_2$
\begin{align*}
 \zeta_{1} &=
   e_1 \otimes (\sigma\delta)^2
  - \sigma\delta \otimes \gamma\xi
  + (\sigma\delta)^2 \otimes e_1
  - \gamma \otimes \delta\sigma\delta
  + \gamma\xi\gamma \otimes \delta
  + \sigma \otimes \xi\gamma\xi
 \\&\ \ \ \,
  - \sigma\delta\sigma \otimes \xi
  + \gamma\eta \otimes \alpha\xi
  - \sigma\beta \otimes \mu\delta
,
\\
 \zeta_{2} &=
    e_2 \otimes \beta\alpha\xi\gamma
  - \xi\gamma \otimes \beta\alpha
  + \eta\mu \otimes \xi\gamma
  - \eta\mu\delta\sigma \otimes e_2
  - \eta \otimes \alpha\xi\gamma
  + \xi\gamma\eta \otimes \alpha
 \\&\ \ \ \,
  + \xi \otimes \gamma\xi\gamma
  - \xi\gamma\xi \otimes \gamma
,
\\
 \zeta_{3} &=
   e_3 \otimes \xi\gamma\eta\mu
  - \delta\sigma \otimes \eta\mu
  - \beta \otimes \mu\delta\sigma
  + \delta\sigma\beta \otimes \mu
  - \delta \otimes \sigma\delta\sigma
  + \delta\sigma\delta \otimes \sigma
 \\&\ \ \ \,
  + \beta\alpha \otimes \delta\sigma
  - \delta\sigma\beta\alpha \otimes e_3
,
\\
 \zeta_{4} &=
   e_4 \otimes \mu\delta\sigma\beta
  - \alpha \otimes \delta\sigma\beta
  - \alpha\xi \otimes \sigma\beta
  + \alpha\xi\gamma \otimes \beta
  - \alpha\xi\gamma\eta \otimes e_4
,
\\
 \zeta_{5} &=
   e_5 \otimes \alpha\xi\gamma\eta
  - \mu \otimes \xi\gamma\eta
  + \mu\delta \otimes \gamma\eta
  + \mu\delta\sigma \otimes \eta
  - \mu\delta\sigma\beta \otimes e_5
.
\end{align*}

We fix also the following basis
$\cB = e_1 \cB \cup e_2 \cB \cup e_3 \cB \cup e_4 \cB \cup e_5 \cB$
of the $K$-vector space $A$:
\begin{align*}
e_1 \cB &=
\{
e_1,
\sigma \delta,
(\sigma \delta)^2,
\gamma,
\gamma \xi \gamma,
\sigma,
\sigma \delta \sigma,
\gamma \eta,
\sigma \beta
\}
,
\\
e_2 \cB &=
\{
e_2,
\xi \gamma,
\eta \mu,
\eta \mu \delta \sigma,
\eta,
\xi \gamma \eta,
\xi,
\xi \gamma \xi
\}
,
\\
e_3 \cB &=
\{
e_3,
\delta \sigma,
\beta,
\delta \sigma \beta,
\delta,
\delta \sigma \delta,
\beta \alpha,
\delta \sigma \beta \alpha
\}
,
\\
e_4 \cB &=
\{
e_4,
\alpha,
\alpha \xi,
\alpha \xi \gamma,
\alpha \xi \gamma \eta
\}
,
\\
e_5 \cB &=
\{
e_5,
\mu,
\mu \delta,
\mu \delta \sigma,
\mu \delta \sigma \beta
\}
.
\end{align*}

We shall show now that the equalities (i), (ii), (iii) hold.

\smallskip

(i)
A direct checking shows that the center $C(A)$
of $A$ has the $K$-linear basis
\[
  \big\{ 1, (\delta \gamma)^2 \big\} ,
\]
and hence $\dim_K HH^0(A) = \dim_K C(A) = 2$.

\smallskip

(ii)
We will determine $\dim_K HH^1(A)$ using the exact sequence
of $K$-vector spaces (Proposition~\ref{prop:2.5})
\[
 0
  \to HH^0(A)
  \to \Hom_{A^e}(\bP_0,A)
  \to \Hom_{A^e}\big(\Omega_{A^e}(A),A\big)
  \to HH^1(A)
  \to 0
  .
\]
We have isomorphisms of $K$-vector spaces
\begin{align*}
 \Hom_{A^e} (\bP_0, A) &=
  \bigoplus_{i=1}^5 e_i A e_i
,
\\
 \Hom_{A^e} (\bP_1, A) &=
   e_1 A e_2 \oplus
   e_2 A e_1 \oplus
   e_1 A e_3 \oplus
   e_3 A e_1 \oplus
   e_2 A e_5 \oplus
   e_5 A e_3 
 \\&\ \ \ \,
\oplus
   e_3 A e_4 \oplus
   e_4 A e_2 ,
\end{align*}
and hence we obtain
\begin{align*}
 \dim_K \Hom_{A^e} (\bP_0, A) &=
   4+2+2+1+1=10
,
\\
 \dim_K \Hom_{A^e} (\bP_1, A) &=
   2+2+2+2+2+2+2+2=16
.
\end{align*}
We identify $\Hom_{A^e}(\Omega_{A^e}(A),A)$
with the $K$-vector space
\[
  \big\{ \varphi \in  \Hom_{A^e} (\bP_1, A) \,|\, \varphi R=0\big\}.
\]

Let $\varphi: \bP_1\to A$ be a homomorphism in $\mod A^e$.
Then,
using the basis
$\cB = e_1 \cB \cup e_2 \cB \cup e_3 \cB \cup e_4 \cB \cup e_5 \cB$ of $A$,
we conclude that there exist
$c_1, c_2, c_3, c_4, d_1, d_2, d_3, d_4$, $a_2, a_3, a_4, a_5, b_2, b_3, b_4$, $b_5 \in K$
such that
\begin{align*}
  \varphi(e_1\otimes e_2) &= c_1 \gamma + d_1 \gamma \xi \gamma,
 &
  \varphi(e_2\otimes e_1) &= c_2 \xi + d_2 \xi \gamma \xi,
 \\
  \varphi(e_1\otimes e_3) &= c_3 \sigma + d_3 \sigma \delta \sigma,
 &
  \varphi(e_3\otimes e_1) &= c_4 \delta + d_4 \delta \sigma \delta,
 \\
  \varphi(e_4\otimes e_2) &= a_4 \alpha + b_4 \alpha \xi \gamma,
 &
  \varphi(e_3\otimes e_4) &= a_3 \beta + b_3 \delta \sigma \beta,
 \\
  \varphi(e_5 \otimes e_3) &= a_5 \mu + b_5 \mu \delta \sigma,
 &
  \varphi(e_2 \otimes e_5) &= a_2 \eta + b_2 \xi \gamma \eta.
\end{align*}
Recall that
$\bP_2
 = P(1,1) \oplus
   P(2,3) \oplus
   P(3,2) \oplus
   P(4,5) \oplus
   P(5,4)$.

We have the equalities
\begin{align*}
  \varphi \big( R(e_1\otimes e_1) \big)
  &=
   \varphi (e_1 \otimes e_3) \delta
   + \sigma \varphi (e_3 \otimes e_1)
   - \varphi (e_1 \otimes e_2) \xi
   - \gamma \varphi (e_2 \otimes e_1)
 \\
  &=
   (- c_1 - c_2 + c_3 + c_4) \sigma \delta
   + (- d_1 - d_2 + d_3 + d_4) (\sigma \delta)^2
  ,
 \\
  \varphi \big( R(e_2 \otimes e_3) \big)
  &=
   \varphi (e_2 \otimes e_5) \mu
   + \eta \varphi (e_5 \otimes e_3)
   - \varphi (e_2 \otimes e_1) \sigma
   - \xi \varphi (e_1 \otimes e_3)
 \\
  &=
   (a_2 + a_5 - c_2 - c_3) \eta \mu
   + (b_2 + b_5 - d_2 - d_3) \xi \gamma \eta \mu
  ,
 \\
  \varphi \big( R(e_3\otimes e_2) \big)
  &=
   \varphi (e_3 \otimes e_4) \alpha
   + \beta \varphi (e_4 \otimes e_2)
   - \varphi (e_3 \otimes e_1) \gamma
   - \delta \varphi (e_1 \otimes e_2)
 \\
  &=
   (a_3 + a_4 - c_4 - c_1) \beta \alpha
   + (b_3 + b_4 - d_4 - d_1) \sigma \delta \beta \alpha
  ,
 \\
  \varphi \big( R(e_4 \otimes e_5) \big)
  &=
   \varphi (e_4 \otimes e_2) \eta
   + \alpha \varphi (e_2 \otimes e_5)
 \\
  &=
   (b_2 + b_4) \alpha \xi \gamma \eta ,
 \\
  \varphi \big( R(e_5 \otimes e_4) \big)
  &=
   \varphi (e_5 \otimes e_3) \beta
   + \mu \varphi (e_3 \otimes e_4)
 \\
  &=
   (b_3 + b_5) \mu \delta \sigma \beta
.
\end{align*}
Then
$\varphi$ factors through $\Omega_{A^e}(A)$
if and only if
\begin{align*}
  b_2 + b_4 &= 0
  ,
 &
  b_3 + b_5 &= 0
  ,
 \\
  c_3 + c_4 - (c_1 + c_2) &= 0
  ,
 &
  d_3 + d_4 - (d_1 + d_2) &= 0
  ,
 \\
  a_2 + a_5 - (c_2 + c_3) &= 0
  ,
 &
  b_2 + b_5 - (d_2 + d_3) &= 0
  ,
 \\
  a_3 + a_4 - (c_1 + c_4) &= 0
  ,
 &
  b_3 + b_4 - (d_1 + d_4) &= 0
  .
\end{align*}

Consider first the equations for the $b_i$ and the $d_i$.
Note that these are independent of the $a_i$ and $c_i$.
We may chose $b_2$ and $b_3$ arbitrarily,
then by $b_4$ and $b_5$ are fixed by the equations.
Moreover,
$d_2 + d_3 = b_2 + b_5 = - (b_3 + b_4) = d_1 + d_4$ are fixed.
We may choose $d_2$, this fixes $d_3$.
We must have have
$d_4 - d_1 = d_2 - d_3$.
If $\charact(K) \neq 2$, this determines
$d_1$ and $d_4$ uniquely.
On the other hand,
if $\charact(K) = 2$, then one of the $d_1$, $d_4$
can be chosen and the other is then fixed.
This shows that this space
has dimension $3$ if $\charact(K) \neq 2$
and dimension $4$ otherwise.

Now consider the equations for the $a_i$ and the $c_i$
We may choose $c_1, c_2, c_3$ arbitrarily,
then $c_4$ is fixed, and as well $a_2 + a_5$ and $a_3 + a_4$.
We can choose $a_2$ and $a_3$ arbitrarily and then
all remaining parameters are fixed by the equations.
Hence the space has dimension $5$.

Summing up, we conclude that
$\dim_K \Hom_{A^e}(\Omega_{A^e}(A),A) = 8$
if $\charact(K) \neq 2$
and
$\dim_K \Hom_{A^e}(\Omega_{A^e}(A),A) = 9$
if $\charact(K) = 2$.
Therefore, we obtain
\begin{align*}
  \dim_K HH^1(A)
   &= 2 + 8 - 10 = 0
   \quad \mbox{ if } \charact(K) \neq 2,
  \\
  \dim_K HH^1(A)
   &= 2 + 9 - 10 = 1
   \quad \mbox{ if } \charact(K) = 2.
\end{align*}

\smallskip

(iii)
We calculate $\dim_K HH^2(A)$ using the exact sequence
of $K$-vector spaces
\[
 0
  \to \Hom_{A^e}\big(\Omega_{A^e}(A),A\big)
  \to \Hom_{A^e}(\bP_1,A)
  \to \Hom_{A^e}\big(\Omega_{A^e}^2(A),A\big)
  \to HH^2(A)
  \to 0
  .
\]
We may identify $\Hom_{A^e}(\Omega_{A^e}^2(A),A)$
with the $K$-vector space
\[
  \big\{ \varphi \in  \Hom_{A^e} (\bP_2, A) \,|\, \varphi (\zeta_i) = 0 \mbox{ for } i \in \{ 1,2,3,4,5 \}  \big\} .
\]

Let $\varphi: \bP_2\to A$ be a homomorphism in $\mod A^e$.
Then,
using the basis
$\cB$ of $A$,
we conclude that there exist
$c_1, c_2, c_3, c_4, c_5, d_1, d_2, d_3, f_1 \in K$
such that
\begin{align*}
 \varphi(e_1 \otimes e_1)
  &= c_1 e_1 + d_1 \sigma \delta + f_1 (\sigma \delta)^2
,\\
 \varphi(e_2 \otimes e_3)
  &= c_2 \eta \mu + d_2 \eta \mu \delta \sigma
,\\
 \varphi(e_3 \otimes e_2)
  &= c_3 \beta \alpha +d_3 \delta \sigma \beta \alpha
,\\
 \varphi(e_4 \otimes e_5)
  &= c_4 \alpha \xi \gamma \eta
,\\
 \varphi(e_5 \otimes e_4)
  &= c_5 \mu \delta \sigma \beta
.
\end{align*}
We have the equalities
\begin{align*}
 \varphi(\zeta_1) &=
   \varphi (e_1 \otimes e_1) (\sigma \delta)^2
   - \sigma \delta \varphi (e_1 \otimes e_1) \gamma \xi
   + (\sigma \delta)^2 \varphi (e_1 \otimes e_1)
 \\&\ \ \ \,
   - \gamma \varphi (e_2 \otimes e_3) \delta \sigma \delta
   + \gamma \xi \gamma \varphi (e_2 \otimes e_3) \delta
   + \sigma \varphi (e_3 \otimes e_2) \xi \gamma \xi
 \\&\ \ \ \,
   - \sigma \delta \sigma \varphi (e_3 \otimes e_2) \xi
   + \gamma \eta \varphi (e_5 \otimes e_4) \alpha \xi
   - \sigma \beta \varphi (e_4 \otimes e_5) \mu \delta
 \\&=
   c_1 (\sigma \delta)^2
,
\\
 \varphi(\zeta_2) &=
   \varphi (e_2 \otimes e_3) \beta \alpha \xi \gamma
   - \xi \gamma \varphi (e_2 \otimes e_3) \beta \alpha
   + \eta \mu \varphi (e_3 \otimes e_2) \xi \gamma
 \\&\ \ \ \,
   - \eta \mu \delta \sigma \varphi (e_3 \otimes e_2)
   - \eta \varphi (e_5 \otimes e_4) \alpha \xi \gamma
   + \xi \gamma \eta \varphi (e_5 \otimes e_4) \alpha
 \\&\ \ \ \,
   + \xi \varphi (e_1 \otimes e_1) \gamma \xi \gamma
   - \xi \gamma \xi \varphi (e_1 \otimes e_1) \gamma
 \\&=
   0
,
\\
 \varphi(\zeta_3) &=
   \varphi (e_3 \otimes e_2) \xi \gamma \eta \mu
   - \delta \sigma \varphi (e_3 \otimes e_2) \eta \mu
   - \beta \varphi (e_4 \otimes e_5) \mu \delta \sigma
 \\&\ \ \ \,
   + \delta \sigma \beta \varphi (e_4 \otimes e_5) \mu
   - \delta \varphi (e_1 \otimes e_1) \sigma \delta \sigma
   + \delta \sigma \delta \varphi (e_1 \otimes e_1) \sigma
 \\&\ \ \ \,
   + \beta \alpha \varphi (e_2 \otimes e_3) \delta \sigma
   - \delta \sigma \beta \alpha \varphi (e_2 \otimes e_3)
 \\&=
   0
,
\\
 \varphi(\zeta_4) &=
   \varphi (e_4 \otimes e_5) \mu \delta \sigma \beta
   - \alpha \varphi (e_2 \otimes e_3) \delta \sigma \beta
   - \alpha \xi \varphi (e_1 \otimes e_1) \sigma \beta
  \\& \ \ \ \,
   + \alpha \xi \gamma \varphi (e_2 \otimes e_3) \beta
   - \alpha \xi \gamma \eta \varphi (e_5 \otimes e_4)
 \\&=
   0
,
\\
 \varphi(\zeta_5) &=
   \varphi (e_5 \otimes e_4) \alpha \xi \gamma \eta
   - \mu \varphi (e_3 \otimes e_2) \xi \gamma \eta
   + \mu \delta \varphi (e_1 \otimes e_1) \gamma \eta
  \\& \ \ \ \,
   + \mu \delta \sigma \varphi (e_3 \otimes e_2) \eta
   - \mu \delta \sigma \beta \varphi (e_4 \otimes e_5)
 \\&=
   0
.
\end{align*}
Hence, $\varphi(\zeta_i) = 0$
for $i \in \{1,2,3,4,5\}$
if and only if $c_1 = 0$.
This implies that
$\dim_K \Hom_{A^e}(\Omega_{A^e}^2(A),A) = 9 -1 = 8$.
Therefore, we obtain
\begin{align*}
  \dim_K HH^2(A)
   &= 8 + 8 - 16 = 0
   \quad \mbox{ if } \charact(K) \neq 2,
  \\
  \dim_K HH^2(A)
   &= 9 + 8 - 16 = 1
   \quad \mbox{ if } \charact(K) = 2.
\end{align*}
\end{proof}

\begin{proposition}
\label{prop:8.2}
Let $K$ be an algebraically closed field of characteristic $2$
and $A = \Lambda_{10}$.
Then
\begin{enumerate}[(i)]
\item
  $\dim_K HH^0(A) = 2$.
\item
  $\dim_K HH^1(A) = 0$.
\item
  $\dim_K HH^2(A) = 0$.
\end{enumerate}
\end{proposition}

\begin{proof}
It follows from the proof of \cite[Proposition~10.3]{BES2}
that $A$ admits
the first three terms of a minimal projective resolution in $\mod A^e$
\[
  \bP_2 \xrightarrow{d_2}
  \bP_1 \xrightarrow{d_1}
  \bP_0 \xrightarrow{d_0}
  A \rightarrow 0
\]
with
\begin{align*}
 \bP_0 &=
   P(1,1) \oplus P(2,2) \oplus P(3,3) \oplus P(4,4) \oplus P(5,5)
,
\\
 \bP_1 &=
   P(1,2) \oplus P(2,1) \oplus P(1,3) \oplus P(3,1) \oplus P(2,5)
  \\& \ \ \ \,
          \oplus P(5,3) \oplus P(3,4) \oplus P(4,2)
,
\\
 \bP_2 &=
   P(1,1) \oplus P(2,3) \oplus P(3,2) \oplus P(4,5) \oplus P(5,4)
,
\end{align*}
the differential
$R = d_2$
given by
\begin{align*}
 R(e_1 \otimes e_1)
  &= \varrho\big(\sigma \delta - \gamma \xi - (\sigma \delta)^2\big)
   = e_1 \otimes \delta
     + \sigma \otimes e_1
     - e_1 \otimes \xi
     - \gamma \otimes e_1
  \\& \ \ \ \,
     - e_1 \otimes \delta \sigma \delta
     - \sigma \otimes \sigma \delta
     - \sigma \delta \otimes \delta
     - \sigma \delta \sigma \otimes e_1
,\\
 R(e_2 \otimes e_3)
  &= \varrho(\eta \mu - \xi \sigma)
   = e_2 \otimes \mu
     + \eta \otimes e_3
     - e_2 \otimes \sigma
     - \xi \otimes e_3
,\\
 R(e_3 \otimes e_2)
  &= \varrho(\beta \alpha - \delta \gamma)
   = e_3 \otimes \alpha
     + \beta \otimes e_2
     - e_3 \otimes \gamma
     - \delta \otimes e_2
,\\
 R(e_4 \otimes e_5)
  &= \varrho(\alpha \eta)
   = e_4 \otimes \eta
     + \alpha \otimes e_5
,\\
 R(e_5 \otimes e_4)
  &= \varrho(\mu \beta)
   = e_5 \otimes \beta
     + \mu \otimes e_4
.
\end{align*}
Moreover, the $A$-$A$-bimodule
$\Omega_{A^e}^3(A) = \Ker R$
is generated by the following elements in $\bP_2$
\begin{align*}
 \Psi_{1} &=
   e_1 \otimes (\sigma \delta)^2
   + \sigma \delta \otimes \big(\gamma \xi+(\sigma \delta)^2\big)
   +(\sigma \delta)^2 \otimes e_1
   + \gamma \otimes \delta \sigma \delta
   + \gamma \xi \gamma \otimes \delta
  \\& \ \ \ \,
   + \sigma \otimes \xi \gamma \xi
   + \sigma \delta \sigma \otimes \xi
   + \gamma \eta \otimes \alpha \xi
   + \sigma \beta \otimes \mu \delta
 + (\sigma \delta)^2 \otimes \sigma \delta + \sigma \delta \sigma \otimes \xi \gamma \xi
,
\\
 \Psi_{2} &=
   e_2 \otimes \beta \alpha \xi \gamma
   + \xi \gamma \otimes \beta \alpha
   + \eta \mu \otimes \xi \gamma
   + \eta \mu \delta \sigma \otimes e_2
   + \eta \otimes \alpha \xi \gamma
  \\& \ \ \ \,
   + \xi \gamma \eta \otimes \alpha
   + \xi \otimes \gamma \xi \gamma
   + \xi \gamma \xi \otimes (\gamma + \gamma \xi \gamma)
 + \xi \gamma \xi \otimes \gamma \xi \gamma
,
\\
 \Psi_{3} &=
   e_3 \otimes \xi \gamma \eta \mu
   + \delta \sigma \otimes \eta \mu
   + \beta \otimes \mu \delta \sigma
   + \delta \sigma \beta \otimes \mu
   + \delta \otimes \sigma \delta \sigma
  \\& \ \ \ \,
   + \delta \sigma \delta \otimes \sigma
   + \beta \alpha \otimes \delta \sigma
   + \delta \sigma \beta \alpha \otimes e_3
,
\\
 \Psi_{4} &=
   e_4 \otimes \mu \delta \sigma \beta
   + \alpha \otimes \delta \sigma \beta
   + \alpha \xi \otimes \sigma \beta
   + \alpha \xi \gamma \otimes \beta
   + \alpha \xi \gamma \eta \otimes e_4
,
\\
 \Psi_{5} &=
   e_5 \otimes \alpha \xi \gamma \eta
   + \mu \otimes \xi \gamma \eta
   + \mu \delta \otimes \gamma \eta
   + \mu \delta \sigma \otimes \eta
   + \mu \delta \sigma \beta \otimes e_5
.
\end{align*}

We use also the basis
$\cB = e_1 \cB \cup e_2 \cB \cup e_3 \cB \cup e_4 \cB \cup e_5 \cB$
of the $K$-vector space $A$,
defined in the proof of Proposition~\ref{prop:8.1}.

We shall show now that the equalities (i), (ii), (iii) hold.

\smallskip

(i)
A direct checking shows that the center $C(A)$
of $A$ has the $K$-linear basis
\[
  \big\{ 1, (\delta \gamma)^2 \big\} ,
\]
and hence $\dim_K HH^0(A) = \dim_K C(A) = 2$.

\smallskip

(ii)
We determine $\dim_K HH^1(A)$ using the exact sequence
of $K$-vector spaces
\[
 0
  \to HH^0(A)
  \to \Hom_{A^e}(\bP_0,A)
  \to \Hom_{A^e}\big(\Omega_{A^e}(A),A\big)
  \to HH^1(A)
  \to 0
  .
\]
We have
\begin{align*}
 \dim_K \Hom_{A^e} (\bP_0, A) &=
   10
,
\\
 \dim_K \Hom_{A^e} (\bP_1, A) &=
   16
,
\end{align*}
as in the proof
of Proposition~\ref{prop:8.1}.
Recall that the differential
$R = d_2 : \bP_2 \to \bP_1$
is given by
\begin{align*}
 R(e_1 \otimes e_1)
  &
   = e_1 \otimes \delta
     + \sigma \otimes e_1
     - e_1 \otimes \xi
     - \gamma \otimes e_1
     - e_1 \otimes \delta \sigma \delta
  \\& \ \ \ \,
     - \sigma \otimes \sigma \delta
     - \sigma \delta \otimes \delta
     - \sigma \delta \sigma \otimes e_1
,\\
 R(e_2 \otimes e_3)
  &
   = e_2 \otimes \mu
     + \eta \otimes e_3
     - e_2 \otimes \sigma
     - \xi \otimes e_3
,\\
 R(e_3 \otimes e_2)
  &
   = e_3 \otimes \alpha
     + \beta \otimes e_2
     - e_3 \otimes \gamma
     - \delta \otimes e_2
,\\
 R(e_4 \otimes e_5)
  &
   = e_4 \otimes \eta
     + \alpha \otimes e_5
,\\
 R(e_5 \otimes e_4)
  &
   = e_5 \otimes \beta
     + \mu \otimes e_4
.
\end{align*}

Let $\varphi: \bP_1 \to A$ be a homomorphism in $\mod A^e$.
Then, as in the proof of Proposition~\ref{prop:8.1}, we have
\begin{align*}
  \varphi(e_1\otimes e_2) &= c_1 \gamma + d_1 \gamma \xi \gamma,
 &
  \varphi(e_2\otimes e_1) &= c_2 \xi + d_2 \xi \gamma \xi,
 \\
  \varphi(e_1\otimes e_3) &= c_3 \sigma + d_3 \sigma \delta \sigma,
 &
  \varphi(e_3\otimes e_1) &= c_4 \delta + d_4 \delta \sigma \delta,
 \\
  \varphi(e_4\otimes e_2) &= a_4 \alpha + b_4 \alpha \xi \gamma,
 &
  \varphi(e_3\otimes e_4) &= a_3 \beta + b_3 \delta \sigma \beta,
 \\
  \varphi(e_5 \otimes e_3) &= a_5 \mu + b_5 \mu \delta \sigma,
 &
  \varphi(e_2 \otimes e_5) &= a_2 \eta + b_2 \xi \gamma \eta.
\end{align*}
for some elements
$c_1, c_2, c_3, c_4, d_1, d_2, d_3, d_4, a_2, a_3, a_4, a_5, b_2, b_3, b_4$, $b_5 \in K$.
We have the equalities
\begin{align*}
  \varphi \big( R(e_1\otimes e_1) \big)
  &=
   \varphi (e_1 \otimes e_3) \delta
   + \sigma \varphi (e_3 \otimes e_1)
   - \varphi (e_1 \otimes e_2) \xi
   - \gamma \varphi (e_2 \otimes e_1)
  \\& \ \ \ \,
   - \varphi (e_1 \otimes e_3) \delta \sigma \delta
   - \sigma \varphi (e_3 \otimes e_1) \sigma \delta
   - \sigma \delta \varphi (e_1 \otimes e_3) \delta
  \\& \ \ \ \,
   - \sigma \delta \sigma \varphi (e_3 \otimes e_1)
 \\
  &=
   \Bigg( \sum_{i=1}^4 c_i \Bigg) \sigma \delta
   + \Bigg( c_1 + c_2 + \sum_{i=1}^4 d_i \Bigg) (\sigma \delta)^2
  ,
\end{align*}
because $\charact(K) = 2$, and
\begin{align*}
  \varphi \big( R(e_2 \otimes e_3) \big)
  &=
   \varphi (e_2 \otimes e_5) \mu
   + \eta \varphi (e_5 \otimes e_3)
   - \varphi (e_2 \otimes e_1) \sigma
   - \xi \varphi (e_1 \otimes e_3)
 \\
  &=
   (a_2 + a_5 - c_2 - c_3) \eta \mu
   + (b_2 + b_5 - d_2 - d_3) \xi \gamma \eta \mu
  ,
 \\
  \varphi \big( R(e_3\otimes e_2) \big)
  &=
   \varphi (e_3 \otimes e_4) \alpha
   + \beta \varphi (e_4 \otimes e_2)
   - \varphi (e_3 \otimes e_1) \gamma
   - \delta \varphi (e_1 \otimes e_2)
 \\
  &=
   (a_3 + a_4 - c_4 - c_1) \beta \alpha
   + (b_3 + b_4 - d_4 - d_1) \sigma \delta \beta \alpha
  ,
 \\
  \varphi \big( R(e_4 \otimes e_5) \big)
  &=
   \varphi (e_4 \otimes e_2) \eta
   + \alpha \varphi (e_2 \otimes e_5)
 \\
  &=
   (b_2 + b_4) \alpha \xi \gamma \eta ,
 \\
  \varphi \big( R(e_5 \otimes e_4) \big)
  &=
   \varphi (e_5 \otimes e_3) \beta
   + \mu \varphi (e_3 \otimes e_4)
 \\
  &=
   (b_3 + b_5) \mu \delta \sigma \beta
.
\end{align*}
Then
$\varphi$ factors through $\Omega_{A^e}(A)$
if and only if
\begin{align*}
  b_2 + b_4 &= 0
  ,
 &
  b_3 + b_5 &= 0
  ,
 \\
  \sum_{i=1}^4 c_i &= 0
  ,
 &
  c_1 + c_2 + \sum_{i=1}^4 d_i &= 0
  ,
 \\
  a_2 + a_5 - (c_2 + c_3) &= 0
  ,
 &
  b_2 + b_5 - (d_2 + d_3) &= 0
  ,
 \\
  a_3 + a_4 - (c_1 + c_4) &= 0
  ,
 &
  b_3 + b_4 - (d_1 + d_4) &= 0
  .
\end{align*}
Consider first the equations involving the $b_i$ and the $d_i$.
We substitute $b_4 = b_2$ and $b_5 = b_3$.
Then the remaining equations with
$b_i, d_i$ are equivalent with
\[
  b_2 + b_5 = d_1 + d_4 = d_2 + d_3 .
\]
We may choose $b_2$, $b_5$ and $d_1$, $d_2$,
and then the remaining $b_i$ and $d_i$
are fixed.
Now we consider the remaining parameters.
We have $d_1 + d_2 + d_3 + d_4 = 0$
from the above conditions.
Then it follows that
$c_1 = c_2$ and $c_3 = c_4$.
This again implies
\[
  a_2 + a_5 = c_2 + c_3 = a_3 + a_4 .
\]
We may choose $c_1$, $c_3$, $a_2$ and $a_3$,
and then the remaining parameters
are fixed by the equations.
Summing up, we see that
$\dim_K \Hom_{A^e}(\Omega_{A^e}(A),A) = 8$.
Therefore, we conclude that
\[
  \dim_K HH^1(A)
    = 2 + 8 - 10 = 0
  .
\]

\smallskip

(iii)
We may identify $\Hom_{A^e}(\Omega_{A^e}^2(A),A)$
with the $K$-vector space
\[
  \big\{ \varphi \in  \Hom_{A^e} (\bP_2, A) \,|\, \varphi (\Psi_i) = 0 \mbox{ for } i \in \{ 1,2,3,4,5 \}  \big\} .
\]

Let $\varphi: \bP_2\to A$ be a homomorphism in $\mod A^e$.
Then
\begin{align*}
 \varphi(e_1 \otimes e_1)
  &= c_1 e_1 + d_1 \sigma \delta + f_1 (\sigma \delta)^2
,\\
 \varphi(e_2 \otimes e_3)
  &= c_2 \eta \mu + d_2 \eta \mu \delta \sigma
,\\
 \varphi(e_3 \otimes e_2)
  &= c_3 \beta \alpha +d_3 \delta \sigma \beta \alpha
,\\
 \varphi(e_4 \otimes e_5)
  &= c_4 \alpha \xi \gamma \eta
,\\
 \varphi(e_5 \otimes e_4)
  &= c_5 \mu \delta \sigma \beta
,
\end{align*}
for some elements
$c_1, c_2, c_3, c_4, c_5, d_1, d_2, d_3, f_1 \in K$.
Then we conclude that
\begin{align*}
 \varphi(\Psi_1) &
   = c_1 (\sigma \delta)^2
,
&
 \varphi(\Psi_2) &
   = 0
,
&
 \varphi(\Psi_3) &
   = 0
,
&
 \varphi(\Psi_4) &
   = 0
,
&
 \varphi(\Psi_5) &
   = 0
.
\end{align*}
Hence, $\varphi(\Psi_i) = 0$
for $i \in \{1,2,3,4,5\}$
if and only if $c_1 = 0$.
This gives
\[
  \dim_K \Hom_{A^e}\big(\Omega_{A^e}^2(A),A\big) = 9 -1 = 8.
\]
Therefore, we obtain
\begin{align*}
  \dim_K HH^2(A)
   &= \dim_K \Hom_{A^e}\big(\Omega_{A^e}(A),A\big)
      + \dim_K \Hom_{A^e}\big(\Omega_{A^e}^2(A),A\big)
   \\&\ \ \ \,
      - \dim_K \Hom_{A^e}(\bP_1,A)
   \\&
    = 8 + 8 - 16 = 0 .
\end{align*}
\end{proof}

\begin{proposition}
Let $K$ be an algebraically closed field of characteristic $2$,
$\Lambda$ the non-standard algebra
$\Lambda_{10}$ over $K$,
and $A$ a standard self-injective algebra over $K$.
Then $A$ and $\Lambda$ are not derived equivalent.
\end{proposition}

\begin{proof}
Assume that $A$ and $\Lambda$ are derived equivalent.
Then it follows from Theorem~\ref{thm:2.8} that
$A$ and $\Lambda$ are stably equivalent, and hence
the stable Auslander-Reiten quivers $\Gamma_A^s$
and $\Gamma_{\Lambda}^s$ are isomorphic.
In particular, we conclude that
$\Gamma_{A}^s$
consists of stable tubes of ranks $1$, $2$, $3$ and $6$.
Further, by
Theorems \ref{thm:2.9} and \ref{thm:2.10},
$A$ is a representation-infinite periodic algebra
of polynomial growth, and hence is a standard
self-injective algebra of tubular type $(2,3,6)$.
Moreover, the Grothendieck groups $K_0(A)$
and $K_0(\Lambda)$ are isomorphic,
and then $K_0(A)$ is of rank $5$.
Then it follows from
\cite[Theorem~5.4]{LS2} that $A$ is either
isomorphic to $\Lambda'_{10}$ or
to the preprojective algebra
$P(\bA_5)$ of Dynkin type $\bA_5$
given by the quiver
\[
  \xymatrix{ \bullet \ar@<+.5ex>[r]^{\alpha}
     \save[] +<0mm,-2.5mm> *{4} \restore
    & \bullet \ar@<+.5ex>[l]^{\beta}
                    \ar@<+.5ex>[r]^{\delta}
     \save[] +<0mm,-2.5mm> *{2} \restore
    & \bullet \ar@<+.5ex>[l]^{\gamma}
                    \ar@<+.5ex>[r]^{\sigma}
     \save[] +<0mm,-2.5mm> *{1} \restore
    & \bullet \ar@<+.5ex>[l]^{\xi}
                    \ar@<+.5ex>[r]^{\eta}
     \save[] +<0mm,-2.5mm> *{3} \restore
    & \bullet \ar@<+.5ex>[l]^{\mu}
     \save[] +<0mm,-2.5mm> *{5} \restore
  }
\]
and the relations
\[
  \alpha \beta = 0,
 \ \,
  \beta \alpha = \delta \gamma,
 \ \,
  \gamma \delta = \sigma \xi,
 \ \,
  \xi \sigma = \eta \mu,
 \ \,
  \mu \eta = 0 .
\]
It is known that $P(\bA_5)$
is a periodic algebra of period $6$
(see \cite[2.10]{ESn}).
Moreover, it follows from
\cite[Corollary~6.3 and Lemma~7.2]{ESn2}
that
\[
  \dim_K HH^1\big(P(\bA_5)\big)
  = \dim_K HH^2\big(P(\bA_5)\big)
  = 2
.
\]
Recall also that, by
Proposition~\ref{prop:8.1},
we have
\[
  \dim_K HH^1(\Lambda'_{10})
  = \dim_K HH^2(\Lambda'_{10})
  = 1
,
\]
because $\charact(K) = 2$.
On the other hand, by
Proposition~\ref{prop:8.2},
we have
\[
  \dim_K HH^1(\Lambda_{10})
  = \dim_K HH^2(\Lambda_{10})
  = 0
.
\]
Therefore, applying Theorem \ref{thm:2.6} again,
we conclude that
$A$ and $\Lambda$ are not derived equivalent.
\end{proof}

\section*{Acknowledgements}

The research was done during the visits of the first and third
named authors at the Mathematical Institute in Oxford
and the second named author at the Nicolaus Copernicus University
in Toru\'n.
The authors thank all these institutions for the warm hospitality.

\end{document}